\documentclass[a4]{amsart} 
\input xypic
\usepackage{amssymb}

\newtheorem{theorem}{Theorem}[section]
\newtheorem{proposition}[theorem]{Proposition}
\newtheorem{lemma}[theorem]{Lemma}
\newtheorem{corollary}[theorem]{Corollary}

\theoremstyle{definition}
\newtheorem{definition}[theorem]{Definition}
\newtheorem{remark}[theorem]{Remark}

\newcommand{\nil}{\ensuremath{\mathrm{nil}}} 
\newcommand{\cocat}{\ensuremath{\mathrm{cocat}}} 
\newcommand{\wcocat}{\ensuremath{\mathrm{wcocat}}} 
 
\newcommand{\uxa}{\ensuremath{(\underline{X},\underline{A})}} 
\newcommand{\cxx}{\ensuremath{(\underline{CX},\underline{X})}} 
\newcommand{\clxx}{\ensuremath{(\underline{C\Omega X},\underline{\Omega X})}} 
\newcommand{\hocolim}{\ensuremath{\underrightarrow{\mathrm{hocolim}}}} 

\newcommand{\hlgy}[1]{\ensuremath{H_{*}(#1)}}
\newcommand{\rhlgy}[1]{\ensuremath{\widetilde{H}_{*}(#1)}}

\newcounter{bean}
\newenvironment{letterlist}{\begin{list}{\rm ({\alph{bean}})}
      {\usecounter{bean}\setlength{\rightmargin}{\leftmargin}}}
      {\end{list}}

\newcommand{\namedright}[3]{\ensuremath{#1\stackrel{#2}
 {\longrightarrow}#3}}
\newcommand{\nameddright}[5]{\ensuremath{#1\stackrel{#2}
 {\longrightarrow}#3\stackrel{#4}{\longrightarrow}#5}}
\newcommand{\namedddright}[7]{\ensuremath{#1\stackrel{#2}
 {\longrightarrow}#3\stackrel{#4}{\longrightarrow}#5
  \stackrel{#6}{\longrightarrow}#7}}

\newcommand{\larrow}{\relbar\!\!\relbar\!\!\rightarrow}
\newcommand{\llarrow}{\relbar\!\!\relbar\!\!\larrow}
\newcommand{\lllarrow}{\relbar\!\!\relbar\!\!\llarrow}

\newcommand{\lnamedright}[3]{\ensuremath{#1\stackrel{#2}
 {\larrow}#3}}
\newcommand{\lnameddright}[5]{\ensuremath{#1\stackrel{#2}
 {\larrow}#3\stackrel{#4}{\larrow}#5}}
\newcommand{\lnamedddright}[7]{\ensuremath{#1\stackrel{#2}
 {\larrow}#3\stackrel{#4}{\larrow}#5
  \stackrel{#6}{\larrow}#7}}
\newcommand{\llnamedright}[3]{\ensuremath{#1\stackrel{#2}
 {\llarrow}#3}}
\newcommand{\llnameddright}[5]{\ensuremath{#1\stackrel{#2}
 {\llarrow}#3\stackrel{#4}{\llarrow}#5}}
\newcommand{\llnamedddright}[7]{\ensuremath{#1\stackrel{#2}
 {\llarrow}#3\stackrel{#4}{\llarrow}#5
  \stackrel{#6}{\llarrow}#7}}
\newcommand{\lllnamedright}[3]{\ensuremath{#1\stackrel{#2}
 {\lllarrow}#3}}
\newcommand{\lllnameddright}[5]{\ensuremath{#1\stackrel{#2}
 {\lllarrow}#3\stackrel{#4}{\lllarrow}#5}}
\newcommand{\lllnamedddright}[7]{\ensuremath{#1\stackrel{#2}
 {\lllarrow}#3\stackrel{#4}{\lllarrow}#5
  \stackrel{#6}{\lllarrow}#7}}

\newcommand{\qqed}{\hfill\Box}

\begin{document}

%%% Title

\title{The dual polyhedral product, cocategory and nilpotence}
\author{Stephen Theriault}
\address{Mathematical Sciences, University
         of Southampton, Southampton SO17 1BJ, United Kingdom}
\email{S.D.Theriault@soton.ac.uk}

%\date{}

%%% Abstract
\begin{abstract} 
The notion of a dual polyhedral product is introduced as a generalization  
of Hovey's definition of Lusternik-Schnirelmann cocategory. Properties established 
from homotopy decompositions that relate the based loops on a polyhedral 
product to the based loops on its dual are used to show that if $X$ is a 
simply-connected space then the weak cocategory of $X$ equals the 
homotopy nilpotency class of~$\Omega X$. This answers a fifty year old 
problem posed by Ganea. The methods are applied to determine 
the homotopy nilpotency class of quasi-$p$-regular exceptional Lie 
groups and sporadic $p$-compact groups. 
\end{abstract}

\maketitle

\section{Introduction} 
\label{sec:intro} 

This paper establishes strong relationships between three different 
concepts in topology: polyhedral products, cocategory and homotopy 
nilpotency. Polyhedral products play a fundamental role in toric topology 
and have a growing number of applications to other areas of mathematics, 
such as group actions on graphs, intersections of quadrics, and coordinate 
subspace arrangements. Cocategory is dual to Lusternik-Schnirelmann 
category; the latter has been heavily studied since it was introduced in the 
1930s, partly because of its connection to counting the critical points 
of functions between smooth manifolds. Homotopy nilpotency is the 
topological analogue of nilpotency in group theory, and it has 
powerful implications for homotopy theory, especially in the stable case. 

Throughout the paper, assume that all spaces have the homotopy type 
of path-connected $CW$-complexes. 
It has long been thought that cocategory and homotopy nilpotency are 
very closely linked. One problem in establishing a good link is settling 
on the right definition of cocategory. Several different definitions exist, each trying 
to dualize some feature of Lusternik-Schirelmann category. We use Hovey's 
definition, but others include those by Ganea~\cite{G1}, Hopkins~\cite{Hop} 
and Murillo-Viruel~\cite{MV}. In all cases, a connection is made between  
the definition of cocategory in question and homotopy nilpotence. For 
example, Hovey and Murillo-Viruel show that, in their own notion of cocategory, 
if the cocategory of a simply-connected space $X$ is $m$ then iterated Samelson 
products of length~$\geq m+1$ formed from the homotopy groups of $\Omega X$ all 
vanish. Our main result is to show that if $X$ is a simply-connected space 
then the weak cocategory of $X$ (using Hovey's definition) is precisely equal 
to the homotopy nilpotency class of $\Omega X$. This has the added benefit  
of identifying Murillo-Viruel's and Hovey's notions of weak cocategory. 

To do this we first generalize the notion of cocategory to dual polyhedral 
products, and then use loop space decompositions to compare the 
based loops on a polyhedral product to the based loops on its dual. 
This results in a natural filtration of the based loops on a polyhedral 
product that should have many applications beyond those in this paper. 
A more detailed description is obtained in the special case of a thin 
product, which is relevant to cocategory. 

The methods developed in the paper are sufficiently powerful to allow for 
explicit calculations of the homotopy nilpotency classes of quasi-$p$-regular 
exceptional Lie groups and nonmodular $p$-compact groups. These cases 
would have previously been considered as completely unapproachable. 

To describe our results more carefully, several definitions are required. 

\medskip 

\noindent
\textbf{The dual polyhedral product}. 
Let $K$ be an abstract simplicial complex on $m$ vertices, and assume the 
empty set belongs to $K$. For $1\leq i\leq m$,
let $(X_{i},A_{i})$ be a pair of pointed $CW$-complexes, where $A_{i}$ is 
a pointed subspace of $X_{i}$. Let $\uxa=\{(X_{i},A_{i})\}_{i=1}^{m}$ be 
the set of $CW$-pairs. For each simplex (face) $\sigma\in K$, let 
$\uxa^{\sigma}$ be the subspace of $\prod_{i=1}^{m} X_{i}$ defined by
\[\uxa^{\sigma}=\prod_{i=1}^{m} Y_{i}\qquad
       \mbox{where}\qquad Y_{i}=\left\{\begin{array}{ll}
                                             X_{i} & \mbox{if $i\in\sigma$} \\
                                             A_{i} & \mbox{if $i\notin\sigma$}.
                                       \end{array}\right.\]
The \emph{polyhedral product} determined by \uxa\ and $K$ is
\[\uxa^{K}=\bigcup_{\sigma\in K}\uxa^{\sigma}\subseteq\prod_{i=1}^{m} X_{i}.\]
For example, suppose each $A_{i}$ is a point. If $K$ is a disjoint union
of $m$ points then $(\underline{X},\underline{\ast})^{K}$ is the wedge
$X_{1}\vee\cdots\vee X_{m}$, and if $K$ is the standard $(m-1)$-simplex
then $(\underline{X},\underline{\ast})^{K}$ is the product
$X_{1}\times\cdots\times X_{m}$. 

The polyhedral product is therefore a colimit over all the faces of $K$, 
ordered by inclusion. There is another way to describe the polyhedral 
product as a colimit. Let $[m]=\{1,\ldots,m\}$ and let $I=\{i_{1},\ldots,i_{k}\}$ 
be a subset of $[m]$. Let $K_{I}$ be the full subcomplex of $K$ on the 
vertex set $I$, that is, the faces of $K_{I}$ are those faces of $K$ whose 
vertices are all in $I$. There is a map of simplicial complexes 
\(\namedright{K_{I}}{}{K}\) 
including $K_{I}$ into $K$, and if the subsets of $[m]$ are ordered by 
inclusion then there is a system of simplicial maps 
\(\namedright{K_{I}}{}{K_{J}}\) 
whenever $I\subseteq J$. These induce maps of polyhedral products 
\(\iota_{I,J}\colon\namedright{\uxa^{K_{I}}}{}{\uxa^{K_{J}}}\). 
Notice that the only possible face of $K$ which is not in some~$K_{I}$ 
for a proper subset $I$ of $[m]$ is $\sigma=[m]$. But in this case $K$ 
equals the full simplex $\Delta^{m-1}$, and the polyhedral product $\uxa^{K}$ 
equals $X_{1}\times\cdots\times X_{m}$. Therefore, if $K\neq\Delta^{m-1}$, then 
\[\uxa^{K}=\bigcup_{I\subsetneq [m]}\uxa^{K_{I}}.\] 

The description of the polyhedral product as a colimit of its full subcomplexes 
lets us define a dual notion. Fix a simplicial complex $K$ on the vertex 
set $[m]$. Since each $\uxa^{K_{I}}$ is a pointwise inclusion into $\uxa^{K}$, 
it is a cofibration but not a fibration so the appropriate dual notion for 
homotopy theory is not an inverse limit but a homotopy inverse limit.  
Order the subsets of $[m]$ by reverse inclusion. If $I\subseteq J$ then 
there does not exist a simplicial map 
\(\namedright{K_{J}}{}{K_{I}}\). 
However, there is a map on the level of polyhedral products. 
If $X^{I}$ is the product $X_{i_{1}}\times\cdots\times X_{i_{k}}$ 
then, as in~\cite[2.2.3]{DS}, the projection 
\(\namedright{X^{J}}{}{X^{I}}\) 
induces a map of polyhedral products 
\(\varphi_{J,I}\colon\namedright{\uxa^{K_{J}}}{}{\uxa^{K_{I}}}\) 
with the property that $\varphi_{J,I}\circ\iota_{I,J}$ is the identity map on $\uxa^{K_{I}}$. 
Assemble $\{\uxa^{K_{I}}\mid I\subseteq [m]\}$ and 
$\{\varphi_{J,I}\mid I\subseteq J\subseteq [m]\}$ as the vertices and edges 
of an $m$-cube in order to take a homotopy inverse limit. 

\begin{definition} 
If $K\neq\Delta^{m-1}$ then the \emph{dual polyhedral product} is  
\[\uxa^{K}_{D}=\underset{\substack{\leftarrow\!\!\relbar\!\!\relbar\!\!\relbar\!\!\relbar \\ 
      I\subsetneq [m]}}{\mbox{holim}}\,\uxa^{K_{I}}.\] 
\end{definition} 
 
%Categorically, let $P(m)^{op}$ be the poset of all subsets of $[m]$, ordered 
%by reverse inclusion. Let $P(m)^{op}\backslash [m]$ be the full subcategory 
%of $P(m)^{op}$ having all objects except the initial object $[m]$. Then 
%the dual polyhedral product is a bifunctor from $P(m)^{op}\times\mathcal{K}(m)$ 
%to the category of pointed topological spaces and continuous maps.  
\medskip 

\noindent 
\textbf{Cocategory}. 
A special case of the dual polyhedral product is the thin product. Let $K$ be 
the simplicial complex on the vertex set $[m]$ consisting of $m$ disjoint 
points. In each pair of spaces $(X_{i},A_{i})$, take $A_{i}=\ast$. Then for 
$I\subseteq [m]$ the full subcomplex $K_{I}$ consists of $\vert I\vert$ 
disjoint points, so 
\[(\underline{X},\underline{\ast})^{K_{I}}=\bigvee_{i\in I} X_{i}.\] 
If $J\subseteq I$, the projection 
\(\namedright{K_{I}}{}{K_{J}}\) 
induces a map of polyhedral products  
\(\namedright{\bigvee_{i\in I} X_{i}}{}{\bigvee_{j\in J} X_{j}}\) 
which sends $X_{i}$ to itself if $i\in J$ or to the basepoint if $i\notin J$. 
Write $\underline{X}$ for the set of spaces $\{X_{1},\cdots,X_{m}\}$. 

\begin{definition} 
The \emph{thin product} of pointed spaces $X_{1},\ldots X_{m}$ is the space 
\[P^{m}(\underline{X})=(\underline{X},\underline{\ast})^{K}_{D}\] 
where $K$ is the simplicial complex consisting of $m$ disjoint points. 
\end{definition} 

The thin product was defined by Hovey~\cite{Hov} as a dual to the fat wedge, 
and he used it to define a notion of cocategory that is dual to 
Lusternik-Schnirelmann category.  Some of its properties have been determined 
by Hovey~\cite{Hov} and Anick~\cite[Lemma 6.24]{A} in the case when $m=3$. 

In the case when each $X_{i}$ equals a common space $X$ for 
$1\leq i\leq m$, write $P^{m}(X)$ for $P^{m}(\underline{X})$. Let 
\[\nabla\colon\namedright{\bigvee_{i=1}^{m} X}{}{X}\] 
be the $m$-fold folding map. 

\begin{definition} 
Let $X$ be a pointed space. The \emph{cocategory} of $X$ is the least $m$ 
for which there exists an extension 
\[\diagram 
      \bigvee_{i=1}^{m+1} X\rto\dto^{\nabla} & P^{m+1}(X)\dldashed|>\tip \\ 
      X. & 
  \enddiagram\] 
In this case, write $\cocat(X)=m$. 
\end{definition} 

For example, when $m=1$ then $P^{2}(X)$ is the homotopy inverse 
limit of the system 
\(X\longrightarrow\ast\longleftarrow X\).  
That is, $P^{2}(X)\simeq X\times X$. So $\cocat(X)=1$ if and only if $X$ is an 
$H$-space. 

There is a weaker version of cocategory which will be important. Define 
the space $F^{m+1}(X)$ and the map $f^{m+1}(X)$ by the homotopy fibration 
\[\lllnameddright{F^{m+1}(X)}{f^{m+1}(X)}{\bigvee_{i=1}^{m+1} X}{}{P^{m+1}(X)}.\] 

\begin{definition} 
Let $X$ be a pointed space. The \emph{weak cocategory} of $X$ is the least $m$ 
for which the composite 
\[\spreaddiagramcolumns{-.8pc}\diagram 
       F^{m+1}(X)\rrto^-{f^{m+1}(X)} & & \bigvee_{i=1}^{m+1} X\dto^{\nabla} \\ 
       & & X 
  \enddiagram\] 
is null homotopic. In this case, write $\wcocat(X)=m$. 
\end{definition} 

Notice that the definitions immediately imply that $\wcocat(X)\leq\cocat(X)$. 
\medskip 

\noindent 
\textbf{Homotopy nilpotency}. 
An \emph{$H$-group} is a homotopy associative $H$-space with a 
homotopy inverse. Let $G$ be an $H$-group. The commutator 
\(\bar{c}\colon\namedright{G\times G}{}{G}\) 
is defined pointwise by $\bar{c}(x,y)=xyx^{-1}y^{-1}$. Observe that 
the restriction of $\bar{c}$ to the wedge is null homotopic so $\bar{c}$ 
extends to a map 
\[c\colon\namedright{G\wedge G}{}{G}.\] 
Since $\Sigma G\wedge G$ is a retract of $\Sigma(G\times G)$, the homotopy 
class of $c$ is uniquely determined by that of $\bar{c}$. The map $c$ 
is the Samelson product of the identity map on $G$ with itself. For an 
integer $m\geq 1$, let $G^{(m+1)}$ be the $(m+1)$-fold smash 
product of $G$ with itself. Define the iterated Samelson product 
\[c_{m}\colon\namedright{G^{(m+1)}}{}{G}\] 
by $c_{m}=c\circ(1\wedge c)\circ\cdots\circ(1\wedge\cdots 1\wedge c)$. 
Notice that $c_{m}$ has a universal property: any Samelson product of 
length $m+1$ on $G$ factors through $c_{m}$. 

\begin{definition} 
Let $G$ be an $H$-group. The \emph{homotopy nilpotency class} of $G$ is 
the least $m$ such that $c_{m}$ is null homotopic but $c_{m-1}$ is not. In 
this case, write $\nil(G)=m$. 
\end{definition} 

For example, $\nil(G)=1$ if and only if $G$ is homotopy commutative. Homotopy 
nilpotency in this formulation is due to Berstein and Ganea~\cite{BGe}, who 
related it to a notion of cocategory different from Hovey's. A different notion 
of homotopy nilpotency is due to Biedermann and Dwyer~\cite{BD} which 
occurs in the context of Goodwillie towers. A series of recent papers have explored 
the relationship between the two types of homotopy nilpotency and the various  
types of cocategory~\cite{BB, CS,CSV,E}.  

Our main theorem identifies (Hovey's) weak cocategory and (Berstein-Ganea's) 
homotopy nilpotency. 

\begin{theorem} 
   \label{main} 
   Let $X$ be a simply-connected space. Then $\wcocat(X)=m$ if and 
   only if \mbox{$\nil(\Omega X)=m$}.  
\end{theorem} 

One interesting consequence of Theorem~\ref{main} is the identification 
of Murillo-Viruel's and Hovey's notions of weak cocategory. The author 
would like to thank the referee for pointing this out. Write 
$\mathrm{MVcocat}(X)$ and $\mathrm{MVwcocat}(X)$ for the Murillo-Viruel 
definition of cocategory (see~\cite[Definitions 3.4 and 3.9]{MV} for explicit definitions). 
In~\cite[Remark 3.16]{MV} it is shown that $\mathrm{MVcocat}(X)\leq\cocat(X)$, 
and the same argument shows that $\mathrm{MVwcocat}(X)\leq\wcocat(X)$. 
On the other hand, in~\mbox{\cite[Remark 4.13]{MV}} it is shown that 
$\nil(\Omega X)\leq\mathrm{MVwcocat}(X)$. Thus, using Theorem~\ref{main} 
we obtain a string of inequalities 
\[\wcocat(X)\leq\nil(\Omega X)\leq\mathrm{MVwcocat}(X)\leq\wcocat(X)\] 
proving the following. 

\begin{corollary} 
   \label{MVH} 
   Let $X$ be a simply-connected space. Then $\mathrm{MVwcocat}(X)=\wcocat(X)$.~$\qqed$ 
\end{corollary} 

The approach to proving Theorem~\ref{main} is to consider the homotopy fibration 
\(\llnameddright{F^{m}(\underline{X})}{f^{m}(\underline{X})}{\bigvee_{i=1}^{m} X_{i}} 
     {}{P^{m}(\underline{X})}\). 
We identify the homotopy type of $F^{m}(\underline{X})$ as a certain wedge 
of suspensions, and the homotopy class of $f^{m}(\underline{X})$ as a wedge 
sum of Whitehead products. In the case when each~$X_{i}$ equals a 
common space $X$, this lets us play off the definition of weak cocategory, 
which involves the map $f^{m}(X)$, with the homotopy nilpotentcy 
of $\Omega X$ by taking the adjoints of the Whitehead products to 
obtain Samelson products. 

The identifications for $F^{m}(\underline{X})$ and $f^{m}(\underline{X})$ 
are obtained as special cases of much more general phenomena involving 
dual polyhedral products. The definition of $\uxa^{K}_{D}$ as a homotopy 
inverse limit implies that there is a map 
\(\namedright{\uxa^{K}}{}{\uxa^{K}_{D}}\). 
Let $F_{[m]}$ be its homotopy fibre. By comparing the homotopy types 
of $\Omega\uxa^{K}$ and $\Omega\uxa^{K}_{D}$, we show that there 
is a homotopy equivalence 
\begin{equation} 
   \label{introdecomp} 
   \Omega\uxa^{K}\simeq\Omega\uxa^{K}_{D}\times\Omega F_{[m]}. 
\end{equation}  
In a way that can be made precise via certain idempotents, $\Omega\uxa^{K}_{D}$ 
contains all the information about $\Omega\uxa^{K}$ that involves only proper 
subsets of the ingredient pairs $(X_{i},A_{i})$, while $\Omega F_{[m]}$ captures all 
of the information about $\Omega\uxa^{K}$ that involves all $m$ pairs $(X_{i},A_{i})$ 
simultaneously. This leads to a filtration of the homotopy theory of 
$\Omega\uxa^{K}$ obtained from its full subcomplexes. 

\begin{theorem} 
   \label{introuxadecomp} 
   For any polyhedral product $\uxa^{K}$, there is a homotopy equivalence  
   \[\Omega\uxa^{K}\simeq\prod_{I\subseteq [m]}\Omega F_{I}\] 
   where $F_{I}$ is the homotopy fibre of the map 
   \(\namedright{\uxa^{K_{I}}}{}{\uxa^{K_{I}}_{D}}\). 
\end{theorem}

In the special case when $\uxa$ is of the form $\cxx$, where $CX$ is the 
reduced cone on $X$, more can be said. The polyhedral product $\cxx^{K}$ 
has been well studied. In~\cite{BBCG} it is shown that there is a homotopy equivalence 
\begin{equation} 
  \label{BBCGsusp} 
  \Sigma\cxx^{K}\simeq\bigvee_{I\notin K}\Sigma^{2} 
       (\vert K_{I}\vert\wedge\widehat{X}^{I}) 
\end{equation} 
where $\vert K_{I}\vert$ is the geometric realization of $K_{I}$, and for 
$I=\{i_{1},\ldots,i_{k}\}$, we have $\widehat{X}^{I}=X_{i_{1}}\wedge\cdots\wedge X_{i_{k}}$. 
In particular, $\Sigma\cxx^{K}$ is a wedge of suspensions of iterated smashes 
of the spaces $X_{i}$. 

A great deal of work has been done in~\cite{BGr,GT2,GT3,IK1,IK2} to determine 
for which simplicial complexes the decomposition~(\ref{BBCGsusp}) desuspends. 
In~\cite{IK2} the notion of a totally homology fillable simplicial complex was 
introduced, which includes the more well known families of shifted, shellable 
and sequentially Cohen-Macaulay complexes. It was shown that if $K$ is 
totally homology fillable then there is a homotopy equivalence 
\begin{equation} 
  \label{BBCGdesusp} 
  \cxx^{K}\simeq\bigvee_{I\notin K}\Sigma (\vert K_{I}\vert\wedge\widehat{X}^{I})   
\end{equation} 
and $\Sigma\vert K_{I}\vert$ is homotopy equivalent to a wedge of spheres. 
Consequently, $\cxx^{K}$ is homotopy equivalent to a wedge of suspensions of 
iterated smashes of the spaces $X_{i}$. 

When $K$ is totally homology fillable we show that the spaces $F_{I}$ in 
Theorem~\ref{introuxadecomp} are also homotopy equivalent to wedges of suspensions of 
iterated smashes of the $X_{i}$'s. Further, in the special case of the thin product, 
the spaces $F_{[m]}$ and $F^{m}(\underline{X})$ are homotopy equivalent, and the map 
\(\llnamedright{F^{m}(\underline{X})}{f^{m}(\underline{X})}{\bigvee_{i=1}^{m} X_{i}}\) 
is a wedge sum of iterated Whitehead products. 
\bigskip 

This paper is organized as follows. In Part I we give homotopy decompositions of 
$\Omega\uxa^{K}$ and $\Omega\uxa^{K}_{D}$ via certain idempotents. 
Section~\ref{sec:loopdecomp} establishes the basic decomposition in terms 
of certain telescopes and Section~\ref{sec:properties} identifies the telescopes as 
certain loop spaces. Section~\ref{sec:cxxcase} refines the decomposition in 
the case of $\Omega\cxx^{K}$ and $\Omega\cxx^{K}_{D}$ when $K$ is 
totally homology fillable by showing that the factors are the based loops 
on certain wedges of suspensions. In Section~\ref{sec:DJK} these results are 
then transferred to give analogous decompositions in the case of 
$\Omega(\underline{X},\underline{\ast})^{K}$ and 
$\Omega(\underline{X},\underline{\ast})^{K}$ for the same complexes $K$. 

In Part II the role of Whitehead products is investigated.  
Section~\ref{sec:WhGanea} recounts some of the homotopy theory 
surrounding a wedge of two spaces. In Section~\ref{sec:WhPorter} a fundamental 
result is proved that does not seem to be in the literature: we show
that Porter's decomposition of the homotopy fibre of the inclusion of a 
wedge into a product can be altered by a homotopy equivalence so that 
the maps from the fibre into the total space are described by Whitehead 
products. Along the way we also give a refined decomposition in the 
case when each space in the wedge is a suspension. Finally, in 
Section~\ref{sec:Whthin} all this is applied to give a homotopy decomposition 
of the homotopy fibre of the map from the wedge into the thin product, and to 
identify the maps from the fibre to the wedge as certain Whitehead products. 

In Part III the results from Parts I and II are used to prove Theorem~\ref{main} 
in Section~\ref{sec:cocat}. To explicitly calculate the homotopy nilpotency 
classes, a special class of $H$-spaces called retractile $H$-spaces is 
discussed in Section~\ref{sec:retractile}, and a criterion for applying Theorem~\ref{main} 
to retractile $H$-spaces is proved in Section~\ref{sec:cocat2}. Examples 
are then given in Section~\ref{sec:examples}. 

The author would like to thank the referee for many valuable comments, 
and for pointing out Corollary~\ref{MVH}.

\Large 
\part{Homotopy decompositions of $\Omega\uxa^{K}$ and $\Omega\uxa^{K}_{D}$.} 
\normalsize 
\medskip 

\section{Decompositions via idempotents} 
\label{sec:loopdecomp} 

The focus for the most part is on $\Omega\uxa^{K}$, with the decomposition 
for $\Omega\uxa^{K}_{D}$ being a consequence. The decomposition 
will be constructed by using a family of commuting idempotents. 
The idempotents will be defined on $\uxa^{K}$ but in order to take a 
product of their telescopes a multiplication is needed, which is why 
loop spaces are taken. We begin with some general information about 
decompositions of $H$-groups using idempotents. 

In general, let $G$ be a path-connected $H$-group. Suppose that 
for $1\leq j\leq k$ there is a family of maps  
\(e_{j}\colon\namedright{G}{}{G}\) 
such that: (i) each $e_{j}$ is an idempotent, (ii) $e_{i}\circ e_{j}\simeq\ast$ if 
$i\neq j$, and (iii) $1\simeq e_{1}+\cdots + e_{k}$, where $1$ is the identity map 
on $G$ and the addition refers to the group structure on~$[G,G]$ induced by 
the multiplication on $G$. The family $\{e_{j}\}_{j=1}^{k}$ is called a set of 
\emph{mutually orthogonal idempotents}. Let $T(e_{j})$ be the telescope of~$e_{j}$, 
that is, $T(e_{j})=\hocolim_{e_{j}} G$, and let 
\(\namedright{G}{}{T(e_{j})}\) 
be the map to the telescope. Observe that, being a telescope, there is a map 
\(\namedright{T(e_{j})}{}{G}\), 
and as $e_{j}$ is an idempotent the composite 
\(\nameddright{T(e_{j})}{}{G}{}{T(e_{j})}\) 
is a homotopy equivalence. In particular, $T(e_{j})$ is a retract of an $H$-space 
and so is itself an $H$-space. 
Observe also that $\hlgy{T(e_{j})}\cong\mbox{Im}\,(e_{j})_{\ast}$. 
Since $e_{i}\circ e_{j}\simeq\ast$ for $i\neq j$ we may form the direct sum 
$\oplus_{j=1}^{k}\hlgy{T(e_{j})}$, and since $1\simeq e_{1}+\cdots + e_{k}$ 
we obtain an isomorphism of modules 
\[\hlgy{G}\cong\oplus_{j=1}^{k}\hlgy{T(e_{j})}.\] 
The product of the telescope maps  
\[\namedright{G}{}{\prod_{j=1}^{k} T(e_{j})}\] 
therefore induces an isomorphism in homology. Since both $G$ and $\prod_{j=1}^{k} T(e_{j})$ 
are $H$-spaces they are nilpotent. Therefore, by Dror's~\cite[Example 4.3]{Dr} generalization 
of Whitehead's Theorem, the product of telescope maps is a homotopy equivalence. 

\begin{lemma} 
   \label{iddecomp} 
   Let $G$ be a path-connected $H$-group and suppose that $\{e_{j}\}_{j=1}^{k}$ 
   is a family of mutually orthogonal idempotents on $G$. Then there is a 
   homotopy equivalence 
   \(\namedright{G}{}{\prod_{j=1}^{k} T(e_{j})}\). 
   $\qqed$ 
\end{lemma} 

Next, suppose that for $1\leq j\leq m$ there is a family of commuting idempotents 
\(e_{j}\colon\namedright{G}{}{G}\). 
Observe that each pair $(e_{j}, 1-e_{j})$ is mutually orthogonal, and the 
larger family of idempotents $\{e_{j}, 1-e_{j}\}_{j=1}^{m}$ all commute. 
Let $\mathcal{J}$ be the collection of $2^{m}$ sequences $(a_{1},\ldots,a_{m})$, 
where each $a_{j}\in\{0,1\}$. For $(a_{1},\ldots,a_{m})\in\mathcal{J}$, define 
\[f_{(a_{1},\ldots,a_{m})}\colon\namedright{G}{}{G}\] 
by the composite 
\[f_{(a_{1},\ldots,a_{m})}=f_{a_{1}}\circ\ldots\circ f_{a_{m}}\qquad\mbox{where}\qquad 
     f_{a_{j}}=\left\{\begin{array}{ll} e_{j} & \mbox{if $a_{j}=0$} \\ 
                 1-e_{j} & \mbox{if $a_{j}=1$}. \end{array}\right.\] 
We record three properties of the maps $f(a_{1},\ldots, a_{m})$. First, 
since the idempotents $\{e_{j},1-e_{j}\}_{j=1}^{m}$ commute, each map 
$f_{(a_{1},\ldots,a_{m})}$ is also an idempotent. Second, if $f(a'_{1},\ldots,a'_{m})$ 
is another such idempotent distinct from $f(a_{1},\ldots,a_{m})$, then at 
least one $1\leq j\leq m$ satisfies $a'_{j}\neq a_{j}$, for otherwise the 
two maps agree on every $a_{j}$ and so are identical. Therefore the 
$j^{th}$ term in the composite for $f(a_{1},\ldots,a_{m})$ is $e_{j}$ 
and that for $f(a'_{1},\ldots,a'_{m})$ is $1-e_{j}$, or vice-versa. But 
as $e_{j}\circ(1-e_{j})$ is null homotopic and the idempotents commute, 
we obtain $f(a_{1},\ldots,a_{m})\circ f(a'_{1},\ldots,a'_{m})\simeq\ast$. 
Third, since $1=e_{j}+(1-e_{j})$ for each $j$, we obtain 
$1=(e_{1}+(1-e_{1}))\circ\cdots\circ (e_{m}+(1-e_{m}))$. Expanding 
gives $1=\Sigma_{(a_{1},\ldots,a_{m})\in\mathcal{J}} f_{(a_{1},\ldots,a_{m})}$. 
The three properties together imply that the collection of idempotents 
$\{f_{(a_{1},\ldots,a_{m})}\}_{(a_{1},\ldots,a_{m})\in\mathcal{J}}$ is mutually 
orthogonal. Therefore, if $T(a_{1},\ldots,a_{m})$ is the telescope of 
$f_{(a_{1},\ldots,a_{m})}$ then Lemma~\ref{iddecomp} implies the following. 

\begin{lemma} 
   \label{fiddecomp} 
   Let $G$ be a path-connected $H$-group and suppose that $\{e_{j}\}_{j=1}^{m}$ 
   is a family of commuting idempotents on $G$. Then there is a homotopy 
   equivalence 
   \(\namedright{G}{}{\prod_{(a_{1},\dots,a_{m})\in\mathcal{J}} T(a_{1},\ldots,a_{m})}.\) 
   $\qqed$ 
\end{lemma} 

We wish to construct a family of commuting idempotents on $\uxa^{K}$. 
Recall from the Introduction that if $I=(i_{1},\ldots,i_{k})\subseteq [m]$ 
then $K_{I}$ is a full subcomplex of $K$ and there is a map of simplicial complexes 
\(\namedright{K_{I}}{}{K}\)  
but not a map 
\(\namedright{K}{}{K_{I}}\). 
However, the situation improves on the level of polyhedral products. 
Let $X^{I}=X_{i_{1}}\times\cdots\times X_{i_{k}}$, let 
\(\namedright{X^{I}}{}{X^{m}}\) 
be the map defined by sending the $j^{th}$-factor of $X^{I}$ to the $(i_{j})^{th}$-factor 
of $X^{m}$, and let 
\(\namedright{X^{m}}{}{X^{I}}\) 
be the projection. The inclusion of~$K_{I}$ into $K$ induces a map 
\(\iota_{I}\colon\namedright{\uxa^{K_{I}}}{}{\uxa^{K}}\). 
An immediate consequence of the definition of the polyhedral product as a union 
of coordinate subspaces of $X^{m}=X_{1}\times\cdots\times X_{m}$ is the 
following lemma~\cite[Lemma 2.2.3]{DS}. 

\begin{lemma} 
   \label{polyinverse} 
   Let $K_{I}$ be a full subcomplex of $K$. The following hold: 
   \begin{letterlist} 
      \item the inclusion  
                \(\namedright{X^{I}}{}{X^{m}}\) 
                induces a map of polyhedral products 
                \(\namedright{\uxa^{K_{I}}}{}{\uxa^{K}}\) 
                which equals $\iota_{I}$; 
      \item the projection  
                \(\namedright{X^{m}}{}{X^{I}}\) 
                induces a map of polyhedral products 
                \(\varphi_{I}\colon\namedright{\uxa^{K}}{}{\uxa^{K_{I}}}\); 
      \item the composite 
                \(\nameddright{\uxa^{K_{I}}}{\iota_{I}}{\uxa^{K}}{\varphi_{I}}{\uxa^{K_{I}}}\) 
                is the identity map. 
   \end{letterlist} 
   $\qqed$ 
\end{lemma}  
                
For $1\leq j\leq m$, let $I_{j}=[m]\backslash\{j\}$. By Lemma~\ref{polyinverse} 
there is a composite of polyhedral products 
\[e_{j}\colon\nameddright{\uxa^{K}}{\varphi_{I_{j}}}{\uxa^{K_{I_{j}}}}{\iota_{I_{j}}}{\uxa^{K}}.\] 

\begin{lemma} 
   \label{idempotents} 
   The following hold: 
   \begin{letterlist} 
      \item for $1\leq j\leq m$ the map $e_{j}$ is an idempotent; 
      \item for any $1\leq j,k\leq m$ we have $e_{j}\circ e_{k}=e_{k}\circ e_{j}$.
   \end{letterlist} 
\end{lemma} 

\begin{proof} 
Part~(a) follows immediately from the definition of $e_{j}$ and 
Lemma~\ref{polyinverse}~(c). For part~(b), if~$j=k$ then the statement 
is a tautology. Suppose that $j\neq k$. Let $I_{j,k}=[m]\backslash\{j,k\}$. 
Observe that the composites of projection and inclusion maps 
\(\namedddright{X^{m}}{}{X^{I_{j}}}{}{X^{m}}{}{X_{I_{k}}}\longrightarrow X^{m}\) 
and 
\(\namedddright{X^{m}}{}{X^{I_{k}}}{}{X^{m}}{}{X_{I_{j}}}\longrightarrow X^{m}\) 
both equal the composite 
\(\nameddright{X^{m}}{}{X^{I_{j,k}}}{}{X^{m}}\). 
Therefore, as in Lemma~\ref{polyinverse}, $e_{j}\circ e_{k}=e_{k}\circ e_{j}$.  
\end{proof} 

By Lemma~\ref{idempotents} the maps $\{e_{j}\}_{j=1}^{m}$ are commuting 
idempotents on $\uxa^{K}$. However, this space is not an $H$-space 
in general so we must loop to obtain one. Since the loop map of an 
idempotent is an idempotent, $\{\Omega e_{j}\}_{j=1}^{m}$ are commuting 
idempotents on $\Omega\uxa^{K}$. For $(a_{1},\ldots,a_{m})\in\mathcal{J}$, define 
\[f_{(a_{1},\ldots,a_{m})}\colon\namedright{\Omega\uxa^{K}}{}{\Omega\uxa^{K}}\] 
by the composite 
\[f_{(a_{1},\ldots,a_{m})}=f_{a_{1}}\circ\ldots\circ f_{a_{m}}\qquad\mbox{where}\qquad 
     f_{a_{j}}=\left\{\begin{array}{ll} \Omega e_{j} & \mbox{if $a_{j}=0$} \\ 
                 1-\Omega e_{j} & \mbox{if $a_{j}=1$}. \end{array}\right.\] 
Then each $f_{a_{j}}$ is an idempotent and as $\{\Omega e_{j}\}_{j=1}^{m}$ commute, 
the composite $f_{(a_{1},\ldots,a_{m})}$ is also an idempotent. Let 
\[T(a_{1},\ldots,a_{m})=\hocolim_{f_{(a_{1},\ldots,a_{m})}}\Omega\uxa^{K}.\] 
Lemma~\ref{fiddecomp} implies the following. 

\begin{proposition} 
   \label{uxadecomp} 
   Assume that $\uxa^{K}$ is simply-connected. 
   Then there is a homotopy equivalence 
   \[\Omega\uxa^{K}\simeq\prod_{(a_{1},\ldots,a_{m})\in\mathcal{J}} T(a_{1},\ldots,a_{m}).\]  
   $\qqed$ 
\end{proposition}  

We wish to relate the factors $T(a_{1},\ldots,a_{m})$ in the 
decomposition of $\Omega\uxa^{K}$ in Proposition~\ref{uxadecomp} 
to the factors that appear in the corresponding decomposition for 
$\Omega\uxa^{K_{I}}$. 

\begin{lemma} 
   \label{limID1} 
   If $j\notin I$ then there is a commutative diagram 
   \[\diagram 
         \uxa^{K}\rto^-{e_{j}}\dto^{\varphi_{I}} & \uxa^{K}\dto^{\varphi_{I}} \\ 
         \uxa^{K_{I}}\rdouble & \uxa^{K_{I}}. 
     \enddiagram\] 
\end{lemma} 

\begin{proof} 
Recall that $I_{j}=[m]\backslash\{j\}$. Since $j\notin I$, 
$K_{I}$ is a full subcomplex of $K_{I_{j}}$. So the projection 
\(\namedright{X^{m}}{}{X^{I}}\) 
is the same as the composite 
\(\namedddright{X^{m}}{}{X^{I_{j}}}{}{X^{m}}{}{X^{I}}\).  
As in Lemma~\ref{polyinverse}, the induced maps of polyhedral products 
\(\namedright{\uxa^{K}}{}{\uxa^{K_{I}}}\) 
and 
\(\namedddright{\uxa^{K}}{}{\uxa^{K_{I_{j}}}}{}{\uxa^{K}}{}{\uxa^{K_{I}}}\) 
are the same. The lemma now follows. 
\end{proof} 

Let $T(e_{j})$ be the telescope of $e_{j}$. Taking telescopes of 
the horizontal maps in Lemma~\ref{limID1} immediately implies the following. 

\begin{corollary} 
   \label{limID2} 
   If $j\notin I$ then there is a commutative diagram 
   \[\diagram 
         \uxa^{K}\rto\dto^{\varphi_{I}} & T(e_{j})\dto \\ 
         \uxa^{K_{I}}\rdouble & \uxa^{K_{I}}. 
     \enddiagram\] 
   $\qqed$ 
\end{corollary}  

Now consider the map 
$f_{(a_{1},\ldots,a_{m})}=f_{a_{1}}\circ\cdots\circ f_{a_{m}}$ where 
$f_{a_{j}}$ is either $\Omega e_{j}$ or $1-\Omega e_{j}$. Suppose 
that $j\notin I$. If $a_{j}=0$ then $f_{a_{j}}=\Omega e_{j}$.  
By Lemma~\ref{limID1}, $\Omega\varphi_{I}\circ\Omega e_{j}=\Omega\varphi_{I}$. 
So as the idempotents $\{\Omega e_{j}\mid 1\leq j\leq m\}$ commute, 
we obtain 
\[\Omega\varphi_{I}\circ f_{a_{1}}\circ\cdots f_{a_{j-1}}\circ\Omega e_{j}\circ 
     f_{a_{j+1}}\circ\cdots\circ f_{a_{m}}= 
    \Omega\varphi_{I}\circ f_{a_{1}}\circ\cdots f_{a_{j-1}}\circ 
     f_{a_{j+1}}\circ\cdots\circ f_{a_{m}}.\] 
If $a_{j}=1$ then $f_{a_{j}}=1-\Omega e_{j}$. By Corollary~\ref{limID2}, 
$\Omega\varphi_{I}\circ(1-\Omega e_{j})\simeq\ast$. So as the 
idempotents $\{\Omega e_{j}\mid 1\leq j\leq m\}$ commute, we obtain 
\[\Omega\varphi_{I}\circ f_{a_{1}}\circ\cdots f_{a_{j-1}}\circ(1-\Omega e_{j})\circ 
     f_{a_{j+1}}\circ\cdots\circ f_{a_{m}}\simeq\ast.\] 
Doing this for every $j\notin I$ gives the following. 

\begin{lemma} 
   \label{varphireduction} 
   Let $I=\{i_{1},\ldots,i_{k}\}\subseteq [m]$. The following hold: 
   \begin{letterlist} 
      \item if $a_{j}=0$ for every $j\notin I$ then 
                $\Omega\varphi_{I}\circ f_{a_{1}}\circ\cdots\circ f_{a_{m}}\simeq  
                    \Omega\varphi_{I}\circ f_{a_{i_{1}}}\circ\cdots\circ f_{a_{i_{k}}}$; 
      \item if $a_{j}=1$ for some $j\notin I$ then 
                $\Omega\varphi_{I}\circ f_{a_{1}}\circ\cdots\circ f_{a_{m}}\simeq\ast$. 
   \end{letterlist} 
   $\qqed$ 
\end{lemma}  
 
Let $\mathcal{J}_{I}$ be the index set for the $2^{\vert I\vert}$ idempotents 
$f_{(a_{i_{1}},\ldots,a_{i_{k}})}$ used to decompose $\Omega\uxa^{K_{I}}$ in 
Proposition~\ref{uxadecomp}. Notice that by Lemma~\ref{varphireduction}~(a), 
$f_{(a_{i_{1}},\ldots,a_{i_{k}})}$ corresponds precisely to the idempotent 
$f_{(a_{1},\ldots,a_{m})}$ on $\Omega\uxa^{K}$ where every $j\notin I$ 
has $a_{j}=0$. Let $T(a_{i_{1}},\ldots,a_{i_{k}})$ be the telescope 
of $f_{(a_{i_{1}},\ldots,a_{i_{k}})}$. Then Lemma~\ref{varphireduction} implies the 
following. 

\begin{lemma} 
   \label{varphitels} 
   Let $I=\{i_{1},\ldots,i_{k}\}\subseteq [m]$. The following hold: 
   \begin{letterlist} 
      \item if $a_{j}=0$ for every $j\notin I$ then there is a commutative diagram 
                \[\diagram 
                       \Omega\uxa^{K}\rto\dto^{\Omega\varphi_{I}} 
                              & T(a_{1},\ldots,a_{m})\dto^{\simeq} \\ 
                       \Omega\uxa^{K_{I}}\rto & T(a_{i_{1}},\ldots,a_{i_{k}});  
                   \enddiagram\] 
      \item if $a_{j}=1$ for some $j\notin I$ then the composite 
                \(T(a_{1},\ldots,a_{m})\hookrightarrow\namedright 
                       {\Omega\uxa^{K}}{\Omega\varphi_{I}}{\Omega\uxa^{K_{I}}}\) 
                is null homotopic. 
   \end{letterlist} 
   $\qqed$ 
\end{lemma} 

From Lemma~\ref{varphitels} we obtain compatibility for 
the decompositions of $\Omega\uxa^{K}$ and $\Omega\uxa^{K_{I}}$. 

\begin{proposition} 
   \label{compatdecomps} 
   Let $I\subseteq [m]$. There is a homotopy commutative diagram 
   \[\diagram 
         \Omega\uxa^{K}\rto^-{\simeq}\dto^{\Omega\varphi_{I}} 
             & \prod_{(a_{1},\ldots,a_{m})\in\mathcal{J}} T(a_{1},\ldots,a_{m})\dto^{\pi_{I}} \\ 
         \Omega\uxa^{K_{I}}\rto^-{\simeq}   
            & \prod_{(a_{i_{1}},\ldots,a_{i_{k}})\in\mathcal{J}_{I}} T(a_{i_{1}},\ldots,a_{i_{k}}) 
     \enddiagram\] 
    where $\pi_{I}$ projects away from factors with $a_{j}=1$ for some $j\notin I$ 
    and identifies $T(a_{1},\ldots,a_{m})$ with $T(a_{i_{1}},\ldots,a_{i_{k}})$ 
    for factors with $a_{j}=0$ for all $j\notin I$.~$\qqed$ 
\end{proposition} 

Next, we bring in the dual polyhedral product. By definition,  
\[\uxa^{K}_{D}=\underset{\substack{\leftarrow\!\!\relbar\!\!\relbar\!\!\relbar\!\!\relbar \\ 
      I\subsetneq [m]}}{\mbox{holim}}\,\uxa^{K_{I}}\] 
where the homotopy inverse limit is taken over the maps of polyhedral products 
\(\namedright{\uxa^{K_{J}}}{}{\uxa^{K_{I}}}\) 
induced by the projection 
\(\namedright{X^{J}}{}{X^{I}}\) 
when $I\subseteq J$. Looping, we obtain 
\[\Omega\uxa^{K}_{D}=\underset{\substack 
     {\leftarrow\!\!\relbar\!\!\relbar\!\!\relbar\!\!\relbar \\ I\subsetneq [m]}}  
     {\mbox{holim}}\,\Omega\uxa^{K_{I}}.\] 
On the other hand, by Proposition~\ref{compatdecomps} the decompositions 
of the spaces $\Omega\uxa^{K_{I}}$ in Proposition~\ref{uxadecomp} are 
compatible with the maps 
\(\namedright{\Omega\uxa^{K_{I}}}{}{\Omega\uxa^{K_{J}}}\) 
and induce projections onto factors. Therefore 
$\underset{\substack{\leftarrow\!\!\relbar\!\!\relbar\!\!\relbar\!\!\relbar \\ I\subsetneq [m]}}  
     {\mbox{holim}}\,\Omega\uxa^{K_{I}}$  
is precisely the product of all possible distinct factors that appear in the decompositions 
of $\Omega\uxa^{K_{I}}$ for any $I\subsetneq [m]$. Put another 
way, the only factors of $\Omega\uxa^{K}$ which are not also factors 
of $\Omega\uxa^{K}_{D}$ are those that project trivially under every 
map $\pi_{I}$ for all $I\subsetneq [m]$. There is only one 
such factor, $T(1,\ldots,1)$, so we obtain the following. 

\begin{proposition} 
   \label{uxaKDdecomp} 
   There is a homotopy commutative diagram 
   \[\diagram 
         \Omega\uxa^{K}\rto^-{\simeq}\dto^{\Omega\varphi} 
             & \prod_{(a_{1},\ldots,a_{m})\in\mathcal{J}} T(a_{1},\ldots,a_{m})\dto^{\pi} \\ 
         \Omega\uxa^{K}_{D}\rto^-{\simeq}   
            & \prod_{(a_{1},\ldots,a_{m})\in\mathcal{J}\backslash(1,\ldots,1)} T(a_{1},\ldots,a_{m}) 
     \enddiagram\] 
   where $\pi$ is the projection.~$\qqed$ 
\end{proposition} 

Since every factor of $\Omega\uxa^{K}_{D}$ is also a factor of 
$\Omega\uxa^{K}$ and the only additional factor of $\Omega\uxa^{K}$ 
is $T(1,\ldots,1)$, Proposition~\ref{uxaKDdecomp} immediately implies 
the following. 

\begin{theorem} 
   \label{dualretract} 
   Assume that $\uxa^{K}$ is simply-connected. The map 
   \(\namedright{\Omega\uxa^{K}}{\Omega\varphi}{\Omega\uxa^{K}_{D}}\) 
   has a right homotopy inverse and there is a homotopy equivalence 
   \[\Omega\uxa^{K}\simeq\Omega\uxa^{K}_{D}\times T(1,\ldots,1).\] 
   $\qqed$ 
\end{theorem}

\section{Further properties of the decompositions} 
\label{sec:properties} 

Next, we show that each of the factors $T(a_{1},\ldots,a_{m})$ in 
the decompositions for $\Omega\uxa^{K}$ and $\Omega\uxa^{K}_{D}$ 
is a loop space, and prove Theorem~\ref{introuxadecomp}. Define the 
space $F_{[m]}$ be the homotopy fibration 
\[\nameddright{F_{[m]}}{}{\uxa^{K}}{\varphi}{\uxa^{K}_{D}}.\] 

\begin{lemma} 
   \label{Tloop1} 
   There is a homotopy equivalence $T(1,\ldots,1)\simeq\Omega F_{[m]}$.  
\end{lemma} 

\begin{proof} 
On the one hand, by Proposition~\ref{uxaKDdecomp}, the homotopy fibre 
of $\Omega\varphi$ is the same as that of~$\pi$, which is $T(1,\ldots,1)$. 
On the other hand, by definition of $F_{[m]}$, the homotopy fibre of 
$\Omega\varphi$ is $\Omega F_{[m]}$. Thus there is an induced map 
of fibres 
\(\namedright{T(1,\ldots,1)}{}{\Omega F_{[m]}}\) 
which induces isomorphisms on homotopy groups by the Five Lemma. 
Since all spaces are $CW$-complexes, this implies the map of fibres 
is a homotopy equivalence. 
\end{proof} 

Fix a sequence $(a_{1},\ldots,a_{m})\in\mathcal{J}$. Let 
$\{a_{i_{1}},\ldots,a_{i_{k}}\}$ be the set consisting of all the elements in the sequence  
which are $1$ and let $\{a_{j_{1}},\ldots,a_{j_{\ell}}\}$ be the set consisting of 
all the elements in the sequence which are $0$. Let $I=\{i_{1},\ldots,i_{k}\}$ 
and $J=\{j_{1},\ldots,j_{\ell}\}$. Note that $k+\ell=m$ and 
$I\cap J=\emptyset$. 

By Lemma~\ref{varphitels}~(a) there is a homotopy equivalence 
$T(a_{1},\ldots,a_{m})\simeq T(a_{i_{1}},\ldots,a_{i_{k}})$, where 
$T(a_{i_{1}},\ldots,a_{i_{k}})$ is the telescope of the idempotent 
$f_{a_{i_{1}}}\circ\cdots\circ f_{a_{i_{k}}}= 
    (1-\Omega e_{i_{1}})\circ\cdots\circ(1-\Omega e_{i_{k}})$ 
on $\Omega\uxa^{K_{I}}$. Note that each $a_{i_{t}}$ equals $1$ for 
$1\leq t\leq k$, so applying Lemma~\ref{Tloop1} to the case 
of~$\Omega\uxa^{K_{I}}$ immediately implies the following. 

\begin{lemma} 
   \label{Tloop2} 
   There is a homotopy equivalence 
   \[T(a_{1},\ldots,a_{m})\simeq\Omega F_{I}\] 
   where $I=\{i_{1},\ldots,i_{k}\}$ consists of those indices in $[m]$ for 
   which $a_{i_{t}}=1$ and $F_{I}$ is the homotopy fibre of the map 
   \(\namedright{\uxa^{K_{I}}}{}{\uxa^{K_{I}}_{D}}\). 
   $\qqed$ 
\end{lemma} 
   
Consequently, the decompositions of $\Omega\uxa^{K}$ in 
Proposition~\ref{uxadecomp} and of $\Omega\uxa^{K}_{D}$ in 
Proposition~\ref{uxaKDdecomp} can be rewritten as follows. 

\begin{theorem} 
   \label{uxadecomp2} 
   There is a homotopy commutative diagram 
   \[\diagram 
          \Omega\uxa^{K}\rto^-{\simeq}\dto^{\Omega\varphi} 
                & \prod_{I\subseteq [m]}\Omega F_{I}\dto^{\pi} \\  
        \Omega\uxa^{K}_{D}\rto^-{\simeq} & \prod_{I\subsetneq [m]}\Omega F_{I}. 
     \enddiagram\] 
   where $\pi$ is the projection.~$\qqed$ 
\end{theorem} 

\begin{proof}[Proof of Theorem~\ref{introuxadecomp}] 
This is simply the homotopy equivalence in the first row of Theorem~\ref{uxadecomp2}. 
\end{proof} 

\begin{remark} 
It is tempting to suspect that the homotopy decompositions in 
Theorem~\ref{uxadecomp2} deloop, that is, that there are homotopy equivalences 
$\uxa^{K}\simeq\prod_{I\subseteq [m]} F_{I}$ and 
$\uxa^{K}_{D}\simeq\prod_{I\subsetneq [m]} F_{I}$. 
But $\uxa^{K}$ and $\uxa^{K}_{D}$ are not $H$-spaces so any hope 
of delooping the homotopy equivalences would come from the homotopy 
fibrations 
\(\nameddright{F_{I}}{}{\uxa^{K_{I}}}{\varphi}{\uxa^{K_{I}}_{D}}\) 
having a splitting of the form  
\(\namedright{\uxa^{K_{I}}}{}{F_{I}}\). 
But this does not happen even in the simplest cases. For example, 
let $K$ be two disjoint points. By the definition of the polyhedral 
product, $\uxa^{K}=(X_{1}\times A_{2})\cup(A_{1}\times X_{2})$. 
The proper full subcomplexes of $K=\{1\}\coprod\{2\}$ are $\{1\}$, 
$\{2\}$ and $\emptyset$. The corresponding polyhedral products are 
$X_{1}$, $X_{2}$ and $\ast$. Therefore, $\uxa^{K}_{D}$ is the homotopy inverse 
limit of the diagram $X_{1}\longrightarrow\ast\longleftarrow X_{2}$, 
which is $X_{1}\times X_{2}$. This implies that space $F_{[2]}$ is the 
homotopy fibre of the inclusion 
\(\namedright{(X_{1}\times A_{1})\cup(A_{1}\times X_{2})}{}{X_{1}\times X_{2}}\). 
Specializing to $A_{1}=A_{2}=\ast$, this fibre is homotopy equivalent to that 
of the inclusion 
\(\namedright{X_{1}\vee X_{2}}{}{X_{1}\times X_{2}}\), 
which by the Hilton-Milnor Theorem, is $\Omega X_{1}\ast\Omega X_{2}$. 
The only case when the map 
\(\namedright{\Omega X_{1}\ast\Omega X_{2}}{}{X_{1}\vee X_{2}}\) 
has a left homotopy inverse is when at least one of $X_{1}$ or $X_{2}$ 
is trivial. 
\end{remark}

\section{Further refinement in the case of $\cxx^{K}$} 
\label{sec:cxxcase} 

Recall from the Introduction that if the simplicial complex $K$ is 
totally homology fillable (a property that includes shifted, shellable 
and sequentially Cohen-Macaulay complexes) then there is a homotopy 
equivalence 
\[\cxx^{K}\simeq\bigvee_{I\notin K}\Sigma  
      \vert K_{I}\vert\wedge\widehat{X}^{I}\]  
and $\Sigma\vert K_{I}\vert$ is homotopy equivalent to a wedge of spheres. 
Thus $\cxx^{K}$ is homotopy equivalent to a wedge sum of spaces of the form 
$\Sigma^{t} X_{i_{1}}\wedge\cdots\wedge X_{i_{k}}$ for various $t\geq 1$ 
and $1\leq i_{1}<\cdots <i_{k}\leq m$. In this section we show that 
for this class of polyhedral products the spaces $F_{I}$ that appear 
in the decompositions of $\Omega\uxa^{K}$ and $\Omega\uxa^{K}_{D}$ 
can be more explicitly identified.  

To prepare, we require two general lemmas. For spaces $B$ and $C$ the 
\emph{right half-smash} of $B$ and~$C$ is the quotient space 
$B\rtimes C=(B\times C)/\sim$ where $(\ast,c)\sim\ast$. It is well-known 
that if $B$ is a co-$H$-space then there is a homotopy equivalence 
$B\rtimes C\simeq B\vee (B\wedge C)$. The following lemma is well 
known and follows easily from the methods in~\cite{G2} (a more detailed 
statement and its proof appear later in Theorem~\ref{WhGanea}).  

\begin{lemma} 
   \label{Ganea} 
   Let $B$ be a path-connected pointed space and $C$ a simply-connected, 
   pointed space. Let 
   \(\namedright{B\vee C}{}{C}\) 
   be the pinch map. Then there is a homotopy fibration 
   \[\nameddright{B\rtimes\Omega C}{f}{B\vee C}{}{C}.\] 
   This homotopy fibration is natural for maps 
   \(\namedright{B}{}{B'}\) 
   and 
   \(\namedright{C}{}{C'}\)~$\qqed$ 
\end{lemma} 

One feature of the James construction~\cite{J} (see~\cite[Proposition 7.9.1]{Se} for a more 
modern presentation) is that if $Y$ is a pointed, path-connected space then 
there is a homotopy equivalence 
\[\Sigma\Omega\Sigma Y\simeq\bigvee_{n=1}^{\infty}\Sigma Y^{(n)}\] 
which is natural for maps 
\(\namedright{Y}{}{Y'}\). 
An  immediate consequence is the following. 

\begin{lemma} 
   \label{James} 
   Let $X$ and $Y$ be a pointed, path-connected spaces. Then there 
   is a homotopy equivalence 
   \[\Sigma X\wedge\Omega\Sigma Y\simeq 
        \bigvee_{n=1}^{\infty}(\Sigma X\wedge Y^{(n)})\] 
   which is natural for maps 
   \(\namedright{X}{}{X'}\) 
   and 
   \(\namedright{Y}{}{Y'}\).~$\qqed$ 
\end{lemma} 

We now give a construction that will identify the homotopy type of 
the space $F_{[m]}$ in the case of a polyhedral product $\cxx^{K}$ where $K$ is totally 
homology fillable. Recall that, for $1\leq j\leq m$, $I_{j}=[m]\backslash\{j\}$ 
and $e_{j}$ is the idempotent  
\[e_{j}\colon\nameddright{\cxx^{K}}{}{\cxx^{K_{I_{j}}}}{}{\cxx^{K}}\] 
induced by projecting $(CX)^{m}$ to $(CX)^{I_{j}}$ and then 
including back into $(CX)^{m}$. First consider~$e_{1}$. Since $\cxx^{K}$ is 
homotopy equivalent to a wedge of spaces of the form 
$\Sigma^{t} X_{i_{1}}\wedge\cdots\wedge X_{i_{k}}$ for various $t\geq 1$ 
and $1\leq i_{1}<\cdots <i_{k}\leq m$, we can write 
\[\cxx^{K}\simeq\Sigma B_{1}\vee\Sigma C_{1}\] 
where each wedge summand of $B_{1}$ has $X_{1}$ as a smash factor and each 
wedge summand of $C_{1}$ does not have $X_{1}$ as a smash factor. 

\begin{lemma} 
   \label{BClemma1} 
   The following hold: 
   \begin{letterlist} 
      \item the restriction of $e_{1}$ to $\Sigma B_{1}$ is null homotopic; 
      \item the restriction of $e_{1}$ to $\Sigma C_{1}$ is the inclusion 
                of $\Sigma C_{1}$ into $\cxx^{K}$; 
      \item part~(a) implies that there is a homotopy commutative diagram 
                \[\diagram 
                         \Sigma B_{1}\vee\Sigma C_{1}\rto^-{q_{1}}\dto^{\simeq} 
                              & \Sigma C_{1}\dto \\ 
                         \cxx^{K}\rto & \cxx^{K_{I_{1}}} 
                  \enddiagram\] 
                where $q_{1}$ is the pinch map. 
   \end{letterlist} 
\end{lemma} 

\begin{proof} 
The wedge decomposition of $\cxx^{K}$ is natural with respect to maps 
of simplicial complexes. Applying this to the composite 
\(e_{1}\colon\nameddright{\cxx^{K}}{}{\cxx^{K_{I_{1}}}}{}{\cxx^{K}}\), 
the fact that $I_{1}=\{2,\ldots,m\}$ implies that any wedge summand of $\cxx^{K}$ 
involving $X_{1}$ is mapped trivially to $\cxx^{K_{1}}$ while any wedge summand 
not involving $X_{1}$ is mapped identically to itself by $e_{1}$. This proves 
parts~(a) and (b). Part~(c) follows immediately from part~(a). 
\end{proof} 

Define the space $G_{1}$ and the map $g_{1}$ by the homotopy fibration 
\[\nameddright{G_{1}}{g_{1}}{\Sigma B_{1}\vee\Sigma C_{1}}{q_{1}}{\Sigma C_{1}}.\] 
By Lemmas~\ref{Ganea} and~\ref{James} there are natural homotopy equivalences 
\begin{equation} 
  \label{G1rtimes} 
  G_{1}\simeq\Sigma B_{1}\rtimes\Omega\Sigma C_{1}\simeq 
      \Sigma B_{1}\vee (\Sigma B_{1}\wedge\Omega\Sigma C_{1})\simeq 
      \Sigma B_{1}\vee \left(\bigvee_{n=1}^{\infty}(\Sigma B_{1}\wedge (C_{1})^{(n)})\right). 
\end{equation}  
Observe that, by definition, each wedge summand in $B_{1}$ 
is a smash product with $X_{1}$ as a factor, so every wedge summand of 
$B_{1}\wedge(C_{1})^{(n)}$ is also a smash product with $X_{1}$ as a factor. 
Therefore, every wedge summand of $G_{1}$ is the suspension of a smash product that 
has $X_{1}$ as a factor. We now separate out those wedge summands that also 
have $X_{2}$ as a factor. As the wedge summands are all suspensions, we can write 
\[G_{1}=\Sigma B_{2}\vee\Sigma C_{2}\] 
where each wedge summand of $B_{2}$ is a smash product with $X_{1}$ and $X_{2}$ 
as factors and each wedge summand of $C_{2}$ is a smash product with $X_{1}$ 
as a factor but not $X_{2}$. Let $\psi_{1}$ be the composite 
\[\psi_{1}\colon\nameddright{G_{1}}{g_{1}}{\Sigma B_{1}\vee\Sigma C_{1}}{\simeq} 
       {\cxx^{K}}.\] 

\begin{lemma} 
   \label{BClemma2} 
   For the composite 
   \(\nameddright{G_{1}=\Sigma B_{2}\vee\Sigma C_{2}}{\psi_{1}}{\cxx^{K}} 
        {e_{2}}{\cxx^{K}}\) 
   the following hold: 
   \begin{letterlist} 
      \item the restriction of $e_{2}\circ\psi_{1}$ to $\Sigma B_{2}$ is null homotopic; 
      \item the restriction of $e_{2}\circ\psi_{1}$ to $\Sigma C_{2}$ is homotopic 
                to the restriction of $\psi_{1}$ to $\Sigma C_{1}$; 
      \item part~(a) implies that there is a homotopy commutative diagram 
                \[\diagram 
                         \Sigma B_{2}\vee\Sigma C_{2}\rto^-{q_{2}}\dto^{\simeq} 
                              & \Sigma C_{2}\ddto \\ 
                         G_{1}\dto^{\psi_{1}} & \\ 
                         \cxx^{K}\rto & \cxx^{K_{I_{2}}} 
                  \enddiagram\] 
                where $q_{2}$ is the pinch map. 
   \end{letterlist} 
\end{lemma} 

\begin{proof} 
The proofs of parts~(a) and~(b) start by reorganizing the data in order to  
apply the naturality of Lemma~\ref{Ganea}. First consider 
$\cxx^{K}\simeq\Sigma B_{1}\vee\Sigma C_{1}$. 
As all the wedge summands of $\Sigma B_{1}$ and~$\Sigma C_{1}$ are 
suspensions, we may write $\Sigma B_{1}=\Sigma B_{1,1}\vee\Sigma B_{1,2}$ 
and $\Sigma C_{1}=\Sigma C_{1,1}\vee\Sigma C_{1,2}$ where 
$\Sigma B_{1,1}$ ($\Sigma C_{1,1}$ respectively) consists of those wedge 
summands of $\Sigma B_{1}$ ($\Sigma C_{1}$) having $X_{1}$ 
as a smash factor but not~$X_{2}$, and $\Sigma B_{1,2}$ ($\Sigma C_{1,2}$) 
consists of those wedge summands of $\Sigma B_{1}$ ($\Sigma C_{1}$) having 
both~$X_{1}$ and~$X_{2}$ as smash factors. The pinch map 
\(\namedright{\Sigma B_{1}\vee\Sigma C_{1}}{q_{1}}{\Sigma C_{1}}\) 
can then be rewritten as a pinch map 
\(\namedright{\Sigma B_{1,1}\vee\Sigma B_{1,2}\vee\Sigma C_{1,1}\vee\Sigma C_{1,2}} 
      {}{\Sigma C_{1,1}\vee\Sigma C_{1,2}}\). 

Let $\Sigma B'_{2}=\Sigma B_{1,2}\vee\Sigma C_{1,2}$ and 
let $\Sigma C'_{2}=\Sigma B_{1,1}\vee\Sigma C_{1,1}$. Notice that 
$\cxx^{K}\simeq\Sigma B'_{2}\vee\Sigma C'_{2}$ where 
$\Sigma B'_{2}$ consists of all wedge summands in $\cxx^{K}$ which 
are smash products with $X_{2}$ as a factor and $\Sigma C'_{2}$ consists 
of all wedge summands in $\cxx^{K}$ which are smash products not 
having~$X_{2}$ as a factor. As in Lemma~\ref{BClemma1}, the restriction 
of $e_{2}$ to~$\Sigma B'_{2}$ is null homotopic and the restriction 
to $\Sigma C'_{2}$ is the inclusion of $\Sigma C'_{2}$ into $\cxx^{K}$.
 Therefore the composite 
\(\nameddright{\Sigma B'_{2}\vee\Sigma C'_{2}}{\simeq}{\cxx^{K}}{}{\cxx^{K_{I_{2}}}}\) 
factors through the pinch map 
\(q_{2}'\colon\namedright{\Sigma B'_{2}\vee\Sigma C'_{2}}{}{\Sigma C'_{2}}\). 
Reordering the wedge summands, $q'_{2}$ can be regarded as the wedge sum   
\(\llnamedright{\Sigma B_{1,1}\vee\Sigma B_{1,2}\vee\Sigma C_{1,1}\vee\Sigma C_{1,2}} 
     {q_{B}\vee q_{C}}{\Sigma B_{1,1}\vee\Sigma C_{1,1}}\) 
of the pinch maps 
\(q_{B}\colon\namedright{\Sigma B_{1}=\Sigma B_{1,1}\vee\Sigma B_{1,2}} 
    {}{\Sigma B_{1,1}}\) 
and  
\(q_{C}\colon\namedright{\Sigma C_{1}=\Sigma C_{1,1}\vee\Sigma C_{1,2}} 
    {}{\Sigma C_{1,1}}\). 

Putting $q_{1}$ and $q'_{2}$ together, by the naturality of Lemma~\ref{Ganea} 
there is a homotopy fibration diagram 
\[\diagram 
        \Sigma B_{1}\rtimes\Omega\Sigma C_{1}\rto\dto^{q_{B}\rtimes\Omega q_{C}}  
          & \Sigma B_{1,1}\vee\Sigma B_{1,2}\vee\Sigma C_{1,1}\vee\Sigma C_{1,2} 
                \rto^-{q_{1}}\dto^{q_{B}\vee q_{C}} 
          & \Sigma C_{1,1}\vee\Sigma C_{1,2}\dto^-{q_{C}} \\ 
        \Sigma B_{1,1}\rtimes\Omega\Sigma C_{1,1}\rto 
          & \Sigma B_{1,1}\vee\Sigma C_{1,1}\rto^{q}  
          & \Sigma C_{1,1}  
   \enddiagram\] 
where $q$ is the pinch map. Putting the left square together with the factorization 
of $e_{2}$ through $q'_{2}$ gives a homotopy commutative diagram 
\[\diagram 
        \Sigma B_{1}\rtimes\Omega\Sigma C_{1}\rto\dto^{q_{B}\rtimes\Omega q_{C}}  
          & \Sigma B_{1,1}\vee\Sigma B_{1,2}\vee\Sigma C_{1,1}\vee\Sigma C_{1,2} 
                \rdouble\dto^{q_{B}\vee q_{C}} 
          & \Sigma B'_{2}\vee\Sigma C'_{2}\rto^-{\simeq}\dto^{q'_{2}} 
          & \cxx^{K}\dto \\ 
        \Sigma B_{1,1}\rtimes\Omega\Sigma C_{1,1}\rto 
          & \Sigma B_{1,1}\vee\Sigma C_{1,1}\rdouble   
          & \Sigma C'_{2}\rto & \cxx^{K_{I_{2}}}.  
   \enddiagram\] 
Notice that $G_{1}\simeq\Sigma B_{1}\rtimes\Omega\Sigma C_{1}$ and 
the top row is homotopic to $\psi_{1}$. 

The maps $q_{B}$ and $q_{C}$ pinch out any wedge 
summands of $\Sigma B_{1}$ and $\Sigma C_{1}$ respectively that have $X_{2}$ as 
a smash factor. As $q_{B}$ and $q_{C}$ are suspensions, the naturality of 
the wedge decomposition~(\ref{G1rtimes}) of 
$G_{1}\simeq\Sigma B_{1}\rtimes\Omega\Sigma C_{1}$ 
implies that $q_{B}\rtimes\Omega q_{C}$ pinches out any wedge summand 
of $G_{1}$ having both $X_{1}$ and~$X_{2}$ as factors, and sends any wedge 
summand having $X_{1}$ as a smash factor but not $X_{2}$ identically to 
itself in $\Sigma B_{1,1}\rtimes\Omega\Sigma C_{1,1}$. That is, 
$q_{B}\rtimes\Omega q_{C}$ is the same as the map 
\(\namedright{G_{1}=\Sigma B_{2}\vee\Sigma C_{2}}{q_{2}}{\Sigma C_{2}}\).  
Therefore, the homotopy commutativity of the previous diagram implies 
that the restriction of $e_{2}\circ\psi_{1}$ to $\Sigma B_{2}$ is null homotopic 
and the restriction to $\Sigma C_{2}$ is the restriction of $\psi_{1}$. This 
proves parts~(a) and~(b). Part~(c) follows immediately from part~(a). 
\end{proof}

Now proceed as before by taking the homotopy fibre of the pinch map  
\(\namedright{G_{1}=\Sigma B_{2}\vee\Sigma C_{2}}{q_{2}}{\Sigma C_{2}}\).  
Iterating, for $1\leq j<m$ we obtain homotopy fibrations 
\[\nameddright{G_{j+1}}{g_{j+1}}{G_{j}\simeq\Sigma B_{j}\vee\Sigma C_{j}} 
      {q_{j}}{\Sigma C_{j}}\] 
where $q_{j}$ is the pinch map; every wedge summand of $B_{j}$ is a smash 
product with $X_{1},\ldots,X_{j}$ as factors; every wedge summand of $C_{j}$ 
is a smash product with $X_{1},\ldots,X_{j-1}$ as factors but not $X_{j}$; 
and there are natural homotopy equivalences
\[G_{j+1}\simeq\Sigma B_{j}\rtimes\Omega\Sigma C_{j}\simeq 
      \Sigma B_{j}\vee (\Sigma B_{j}\wedge\Omega\Sigma C_{j})\simeq 
      \Sigma B_{j}\vee \left(\bigvee_{n=1}^{\infty}(\Sigma B_{j}\wedge (C_{j})^{(n)})\right).\] 
Further, if $\psi_{j}$ is the composite 
\[\psi_{j}\colon\nameddright{G_{j}}{g_{j}}{G_{j-1}}{\psi_{j-1}}{\cxx^{K}}\] 
then arguing as in Lemma~\ref{BClemma2} the following hold. 

\begin{lemma} 
   \label{BClemmak} 
   For the composite 
   \(\nameddright{G_{j}=\Sigma B_{j}\vee\Sigma C_{j}}{\psi_{j}} 
         {\cxx^{K}}{e_{j}}{\cxx^{K}}\)  
   the following hold: 
   \begin{letterlist} 
      \item the restriction of $e_{j}\circ\psi_{j}$ to $\Sigma B_{j}$ is null homotopic; 
      \item the restriction of $e_{j}\circ\psi_{j}$ to $\Sigma C_{j}$ is homotopic 
                to the restriction of $\psi_{j}$ to $\Sigma C_{j}$; 
      \item part~(a) implies that there is a homotopy commutative diagram 
                \[\diagram 
                         \Sigma B_{j}\vee\Sigma C_{j}\rto^-{q_{j}}\dto^{\simeq} 
                              & \Sigma C_{j}\ddto \\ 
                         G_{j}\dto^{\psi_{j}} & \\ 
                         \cxx^{K}\rto & \cxx^{K_{I_{j}}} 
                  \enddiagram\] 
                where $q_{j}$ is the pinch map. 
   \end{letterlist} 
   $\qqed$ 
\end{lemma} 

For $1\leq j\leq m$, let $\theta_{j}$ be the composite  
\[\theta_{j}\colon\Sigma C_{j}\hookrightarrow 
      \namedright{\Sigma B_{j}\vee\Sigma C_{j}=G_{j}}{\psi_{j}}{\cxx^{K}}\]
and let $\theta$ be the product of the maps $\Omega\theta_{j}$: 
\[\theta\colon\namedright{\prod_{j=1}^{m}\Omega\Sigma C_{j}}{}{\Omega\cxx^{K}}.\] 
Let $\theta\cdot\psi_{m}$ be the composite 
\[\theta\cdot\psi_{m}\colon\llnameddright 
     {(\prod_{j=1}^{m}\Omega\Sigma C_{j})\times\Omega G_{m}}{\theta\times\psi_{m}} 
     {\Omega\cxx^{K}\times\Omega\cxx^{K}}{\mu}{\Omega\cxx^{K}}\] 
where $\mu$ is the loop multiplication. 

\begin{lemma} 
   \label{loopcxx1} 
   Suppose that $K$ is totally homology fillable. Then the map 
   \(\llnamedright{(\prod_{j=1}^{m}\Omega\Sigma C_{j})\times\Omega G_{m}} 
         {\theta\cdot\psi_{m}}{\Omega\cxx^{K}}\) 
   is a homotopy equivalence.  
\end{lemma} 

\begin{proof} 
In general, the homotopy fibration 
\(\nameddright{B\rtimes\Omega C}{}{B\vee C}{}{C}\) 
in Lemma~\ref{Ganea} has a section 
\(\namedright{C}{}{B\vee C}\) 
given by the inclusion of the wedge summand. Therefore, after looping, 
there is a homotopy equivalence 
\(\namedright{\Omega C\times(\Omega B\rtimes\Omega C)}{\simeq}{\Omega(B\vee C)}\). 
In our case, for $1\leq j<m$, from the homotopy fibration 
\(\nameddright{G_{j+1}}{g_{j+1}}{G_{j}=\Sigma B_{j}\vee\Sigma C_{j}}{q_{j}} 
     {\Sigma C_{j}}\) 
we obtain a homotopy equivalence 
\(\namedright{\Omega\Sigma C_{j}\times\Omega G_{j+1}}{\simeq}{\Omega G_{j}}\). 
Iteratively substituting the homotopy equivalence for $\Omega G_{j+1}$ into that 
for $\Omega G_{j}$, we obtain a homotopy equivalence 
\(\namedright{(\prod_{j=1}^{m}\Omega\Sigma C_{j})\times\Omega G_{m}}{\simeq} 
     {\Omega\cxx^{K}}\). 
Notice that the restriction of this homotopy equivalence to $\Omega\Sigma C_{j}$ 
is the definition of $\Omega\theta_{j}$ and the restriction to~$\Omega G_{m}$ 
is $\Omega\psi_{m}$. Thus the equivalence is $\theta\cdot\Omega\psi_{m}$. 
\end{proof} 

Theorem~\ref{uxadecomp2} and Lemma~\ref{loopcxx1} give different homotopy 
equivalences for $\Omega\cxx^{K}$. We wish to compare them. 

\begin{lemma} 
   \label{Gmnull} 
   The composite 
   \(\nameddright{G_{m}}{\psi_{m}}{\cxx^{K}}{\varphi}{\cxx^{K}_{D}}\) 
   is null homotopic. 
\end{lemma} 

\begin{proof} 
Fix an integer $j$ for $1\leq j<m$. Consider the homotopy fibration 
\(\nameddright{G_{j+1}}{g_{j+1}}{G_{j}}{q_{j}}{\Sigma C_{j}}\). 
By Lemma~\ref{BClemmak}~(c), the composite 
\(\gamma_{j}\colon\nameddright{G_{j}}{\psi_{j}}{\cxx^{K}}{}{\cxx^{K_{I_{j}}}}\) 
factors through $q_{j}$. Therefore $\gamma_{j}\circ g_{j+1}$ is null 
homotopic. By definition, $\psi_{j+1}=\psi_{j}\circ g_{j+1}$, so the composite  
\(\nameddright{G_{j+1}}{\psi_{j+1}}{\cxx^{K}}{}{\cxx^{K_{I_{j}}}}\) 
is null homotopic. The recursive definition of $\psi_{m}$ implies that it 
factors through $\psi_{j}$ so the composite 
\(\nameddright{G_{m}}{\psi_{m}}{\cxx^{K}}{}{\cxx^{K_{I_{j}}}}\) 
is null homotopic. This holds for all $1\leq j\leq m$, so $\psi_{m}$ 
composes trivially into $\cxx^{K_{I_{j}}}$ for all $j$. 
But this implies that $\psi_{m}$ composes trivially to $\cxx^{K_{I}}$ for every 
proper subset $I$ of $[m]$ because the projection 
\(\namedright{(CX)^{m}}{}{(CX)^{I}}\) 
has to factor through some $(CX)^{I_{t}}$. Looping, $\Omega\psi_{m}$ 
composes trivially to $\Omega\cxx^{K_{I}}$ for any $I\subsetneq [m]$. But by 
Proposition~\ref{uxaKDdecomp}, $\Omega\cxx^{K}_{D}$ decomposes 
as a product, each factor of which is a factor of $\Omega\cxx^{K_{I}}$ 
for some $I$, and this decomposition is compatible with $\Omega\varphi$. 
Thus the composite 
\(\nameddright{\Omega G_{m}}{\Omega\psi_{m}}{\Omega\cxx^{K}}{\Omega\varphi} 
       {\Omega\cxx^{K}_{D}}\) 
is null homotopic. 

To deloop this, observe that as $G_{m}$ is a suspension there is a map 
\(\namedright{G_{m}}{}{\Sigma\Omega G_{m}}\) 
which is a right homotopy inverse of the evaluation map 
\(ev\colon\namedright{\Sigma\Omega G_{m}}{}{G_{m}}\). 
The naturality of the evaluation map implies that there is a homotopy 
commutative diagram 
\[\diagram 
     G_{m}\rto\drdouble & \Sigma\Omega G_{m}\rto^-{\Sigma\Omega\psi_{m}}\dto^{ev} 
        & \Sigma\Omega\cxx^{K}\rto^-{\Sigma\Omega\varphi}\dto^{ev} 
        & \Sigma\Omega\cxx^{K}_{D}\dto^{ev} \\ 
     & G_{m}\rto^-{\psi_{m}} & \cxx^{K}\rto^-{\varphi} & \cxx^{K}_{D}. 
  \enddiagram\] 
The null homotopy for $\Omega\varphi\circ\Omega\psi_{m}$ therefore implies 
that $\varphi\circ\psi_{m}$ is also null homotopic. 
\end{proof} 

Recall that there is a homotopy fibration 
\(\nameddright{F_{[m]}}{}{\cxx^{K}}{\varphi}{\cxx^{K}_{D}}\). 

\begin{corollary} 
   \label{Gmnullcor} 
   There is a lift 
   \[\xymatrix{
         & G_{m}\ar[d]^{\psi_{m}}\ar@{-->}[dl]_{\lambda} & \\ 
         F_{[m]}\ar[r] & \cxx^{K}\ar[r]^-{\varphi} & \cxx^{K}_{D} 
     }\] 
   for some map $\lambda$. Further, $\Omega\lambda$ has a 
   left homotopy inverse. 
\end{corollary} 

\begin{proof} 
The existence of the lift follows immediately from Lemma~\ref{Gmnull}. 
By Lemma~\ref{loopcxx1}, $\Omega\psi_{m}$ has a left homotopy inverse. 
The fact that $\lambda$ factors through $\psi_{m}$ then implies that $\Omega\lambda$ 
also has a left homotopy inverse. 
\end{proof} 

\begin{lemma} 
   \label{loopcxx2} 
   \(\nameddright{\prod_{j=1}^{m}\Omega\Sigma C_{j}}{\theta}{\Omega\cxx^{K}} 
          {\Omega\varphi}{\Omega\cxx^{K}_{D}}\) 
   has a left homotopy inverse. 
\end{lemma} 

\begin{proof} 
Fix an integer $j$ for $1\leq j<m$. By Lemma~\ref{loopcxx1}, the map 
\(\namedright{\Omega\Sigma C_{j}}{\Omega\theta_{j}}{\Omega\cxx^{K}}\) 
has a left homotopy inverse. On the other hand, by Lemma~\ref{BClemmak}~(b), 
$e_{j}\circ\theta_{j}\simeq\theta_{j}$. This implies that the composite 
\(\namedddright{\Omega\Sigma C_{j}}{\Omega\theta_{j}}{\Omega\cxx^{K}} 
     {}{\Omega\cxx^{K_{j}}}{}{\Omega\cxx^{K}}\) 
has a left homotopy inverse. Therefore the composite 
\(\nameddright{\Omega\Sigma C_{j}}{\Omega\theta_{j}}{\Omega\cxx^{K}} 
     {}{\Omega\cxx^{K_{j}}}\) 
has a left homotopy inverse. By definition of $\cxx^{K}_{D}$ as a homotopy 
limit, the map 
\(\namedright{\cxx^{K}}{}{\cxx^{K_{j}}}\) 
factors through $\cxx^{K}_{D}$. Thus the composite 
\(\nameddright{\Omega\Sigma C_{j}}{\Omega\theta_{j}}{\Omega\cxx^{K}} 
     {\Omega\varphi}{\Omega\cxx^{K}_{D}}\) 
has a left homotopy inverse. This is true for all $1\leq j\leq m$. Collectively, 
no overlap in the factors in the different retractions occurs because the product map 
\(\namedright{\prod_{j=1}^{m}\Omega\Sigma C_{j}}{\theta}{\Omega\cxx^{K}}\) 
has a left homotopy inverse. Hence the composite 
\(\nameddright{\prod_{j=1}^{m}\Omega\Sigma C_{j}}{\theta}{\Omega\cxx^{K}} 
      {\Omega\varphi}{\Omega\cxx^{K}_{D}}\) 
also has a left homotopy inverse.  
\end{proof}  

\begin{proposition} 
   \label{Tloopwedge} 
   Let $K$ be a totally homology fillable simplicial complex. Then the map 
   \(\namedright{G_{m}}{\lambda}{F_{[m]}}\) 
   in Corollary~\ref{Gmnullcor} is a homotopy equivalence. 
\end{proposition} 

\begin{remark} 
It is worth repeating at this point that $G_{m}$ is homotopy equivalent 
to a wedge of spaces of the form $\Sigma^{t} X_{i_{1}}\wedge\cdots\wedge X_{i_{k}}$ 
where $t\geq 1$ and $\{1,\ldots,m\}\subseteq\{i_{1},\ldots,i_{k}\}$. 
That is, each $X_{i}$ for $1\leq i\leq m$ appears as a smash factor in every 
wedge summands of $G_{m}$. 
\end{remark} 

\begin{proof} 
On the one hand, by Theorem~\ref{uxadecomp2}, the homotopy fibration 
\(\nameddright{F_{[m]}}{}{\Omega\cxx^{K}}{\Omega\varphi}{\Omega\cxx^{K}_{D}}\) 
splits to give a homotopy equivalence 
$\Omega\cxx^{K}\simeq\Omega\cxx^{K}_{D}\times\Omega F_{[m]}$. 
On the other hand, by Lemma~\ref{loopcxx1}, the composite 
\(\llnamedright{(\prod_{j=1}^{m}\Omega\Sigma C_{j})\times\Omega G_{m}} 
         {\theta\cdot\psi_{m}}{\Omega\cxx^{K}}\) 
is a homotopy equivalence. We compare the two decompositions. 

By Lemma~\ref{loopcxx2}, the composite 
\(h\colon\nameddright{\prod_{j=1}^{m}\Omega\Sigma C_{j}}{\theta}{\Omega\cxx^{K}} 
       {\Omega\varphi}{\Omega\cxx^{K}_{D}}\) 
has a left homotopy inverse.  By Corollary~\ref{Gmnullcor}, the map 
\(\namedright{\Omega G_{m}}{\Omega\psi_{m}}{\Omega\cxx^{K}}\) 
lifts to a map 
\(\namedright{\Omega G_{m}}{\Omega\lambda}{\Omega F_{[m]}}\) 
and~$\Omega\lambda$ has a left homotopy inverse. 
Thus the product map 
\[\Gamma\colon\llnamedddright{\Omega\cxx^{K}}{\simeq} 
    {(\prod_{j=1}^{m}\Omega\Sigma C_{k})\times\Omega G_{m}}{h\times\Omega\lambda} 
    {\Omega\cxx^{K}_{D}\times\Omega F_{[m]}}{\simeq}{\Omega\cxx^{K}}\] 
has the property that $h\times\Omega\lambda$ has a left homotopy inverse. 
In particular, $h\times\Omega\lambda$ induces an injection in homology and 
therefore $\Gamma_{\ast}$ is an injection in homology with any coefficients. 
Taking $\mathbb{Z}/p\mathbb{Z}$ or $\mathbb{Q}$ coefficients, $\Gamma_{\ast}$ 
is a self-map of a finite type module which is an injection, and so it is an 
isomorphism. Thus $\Gamma$ induces an isomorphism in mod-$p$ and 
rational homology and so induces an isomorphism in integral homology. 
Therefore $\Gamma$ is a homotopy equivalence. This implies that 
$h\times\Omega\lambda$ is a homotopy equivalence. As $h\times\Omega\lambda$ 
is a product map, each of $h$ and $\Omega\lambda$ must therefore 
be a homotopy equivalence. 

Finally, since $\Omega\lambda$ is a homotopy equivalence, it induces an 
isomorphism on homotopy groups, and therefore so does $\lambda$. 
Hence $\lambda$ is also a homotopy equivalence, as asserted. 
\end{proof} 

By Theorem~\ref{uxadecomp2}, each factor in the homotopy decompositions 
of $\Omega\cxx^{K}$ and $\Omega\cxx^{K}_{D}$ is of the form $\Omega F_{I}$ 
where $F_{I}$ is the homotopy fibre of the map 
\(\namedright{\cxx^{K_{I}}}{}{\cxx^{K_{I}}_{D}}\). 
Further, by~\cite{IK2} any subcomplex of a totally homology fillable simplicial 
complex is itself totally homology fillable. Therefore Proposition~\ref{Tloopwedge} 
applies to each $F_{I}$ to obtain the following. 

\begin{theorem} 
   \label{cxxdecomp} 
   Let $K$ be a totally homology fillable simplicial complex. Then the homotopy 
   decompositions 
   \[\Omega\cxx^{K}\simeq\prod_{I\subset [m]}\Omega F_{I}\qquad 
         \Omega\cxx^{K}_{D}\simeq\prod_{I\subsetneq [m]}\Omega F_{I}\] 
   in Theorem~\ref{uxadecomp2} have the property that each space $F_{I}$ 
   is homotopy equivalent to a wedge of summands, where each wedge 
   summand is the suspension of a smash product having $X_{i}$ as a 
   factor for all $i\in I$.~$\qqed$ 
\end{theorem}

\section{A generalization to $(\underline{X},\underline{\ast})^{K}$} 
\label{sec:DJK} 
 
In this section we show that Theorem~\ref{cxxdecomp} gives useful 
information for a wider range of polyhedral products. Let $\{X_{i}\}_{i=1}^{m}$ 
be a collection of pointed $CW$-complexes and let $K$ be a simplicial 
complex on the vertex set $[m]$. Bahri, Bendersky, Cohen and 
Gitler~\cite[Corollary 2.32]{BBCG}, relying heavily on a result of Denham 
and Suciu~\cite{DS}, show that there is a homotopy fibration 
\begin{equation} 
  \label{clxfib} 
  \nameddright{(\underline{C\Omega X},\underline{\Omega X})^{K}}{} 
      {(\underline{X},\underline{\ast})^{K}}{}{\prod_{i=1}^{m} X_{i}}. 
\end{equation}  
The spaces $(\underline{X},\underline{\ast})^{K}$ include many familiar 
spaces: if $K$ is $m$ disjoint points then $(\underline{X},\underline{\ast})^{K}$ 
is homotopy equivalent to the wedge $X_{1}\vee\cdots\vee X_{m}$; if 
$K$ is the boundary of the standard $(m-1)$-simplex then 
$(\underline{X},\underline{\ast})^{K}$ is the fat wedge in $\prod_{i=1}^{m} X_{i}$; 
and crucially to toric topology, if each $X_{i}=\mathbb{C}P^{\infty}$ then 
$(\underline{X},\underline{\ast})^{K}$ is the Davis-Januszkiewicz space $DJ_{K}$. 

Each vertex $i$ of $K$ is the full subcomplex $K_{\{i\}}$, and the polyhedral 
product $(\underline{X},\underline{\ast})^{K_{\{i\}}}$ is simply~$X_{i}$. So 
including $i$ into $K$ we obtain a map of polyhedral products 
\(\namedright{X_{i}}{}{(\underline{X},\underline{\ast})^{K}}\) 
which, when composed to $\prod_{i=1}^{m} X_{i}$, is the inclusion 
of the $i^{th}$ factor. After looping we can take the product of all 
such maps for $1\leq i\leq m$ to obtain a section for the map 
\(\namedright{\Omega(\underline{X},\underline{\ast})^{K}}{} 
     {\prod_{i=1}^{m}\Omega X_{i}}\), 
implying the following. 

\begin{lemma} 
   \label{clxsplitting} 
   The homotopy fibration~(\ref{clxfib}) splits after looping, resulting in  
   a homotopy equivalence 
   \[\Omega(\underline{X},\underline{\ast})^{K}\simeq(\prod_{i=1}^{m}\Omega X_{i}) 
          \times\Omega\clxx^{K}.\] 
   $\qqed$ 
\end{lemma} 

We wish to show that the homotopy equivalence in Lemma~\ref{clxsplitting} is 
compatible with that in Theorem~\ref{cxxdecomp}. The following lemma does this. 
Write $F_{I}(\underline{X},\underline{\ast})$ and $F_{I}\clxx$ for the 
spaces $F_{I}$ that appear in the respective decompositions of 
$\Omega(\underline{X},\underline{\ast})^{K}$ and $\Omega\clxx^{K}$ in 
Theorem~\ref{uxadecomp2}. 

\begin{lemma} 
   \label{vertexfactor} 
   Let $1\leq i\leq m$. The following hold: 
   \begin{letterlist} 
      \item $F_{\{i\}}(\underline{X},\underline{\ast})\simeq X_{i}$;  
      \item $F_{\{i\}}\clxx\simeq\ast$; 
      \item for $I\subseteq [m]$ and $I\neq\{i\}$ for any $1\leq i\leq m$, the map 
               \(\namedright{\clxx^{K}}{}{(\underline{X},\underline{\ast})^{K}}\) 
               induces a homotopy equivalence 
               $F_{I}\clxx\simeq F_{I}(\underline{X},\underline{\ast})$.   
   \end{letterlist} 
\end{lemma} 

\begin{proof} 
In general, for $\uxa^{K}$, the definition of the polyhedral product implies that 
$\uxa^{K_{\{i\}}}=X_{i}$. By the definition of the dual polyhedral product 
as a homotopy colimit over the proper subcomplexes of $K$, we obtain 
$\uxa^{K_{\{i\}}}_{D}=\ast$, and by definition, $F_{\{i\}}$ is the homotopy 
fibre of the map 
\(\namedright{\uxa^{K_{\{i\}}}}{}{\uxa^{K_{\{i\}}}_{D}}\). 
Thus $F_{\{i\}}\simeq X_{i}$. In particular, we obtain 
$F_{\{i\}}(\underline{X},\underline{\ast})\simeq X_{i}$ and 
$F_{\{i\}}\clxx\simeq C\Omega X_{i}\simeq\ast$. 
Part~(c) now follows from parts~(a) and~(b) and Lemma~\ref{clxsplitting}. 
\end{proof} 

Consequently, by applying Theorem~\ref{cxxdecomp} to $\Omega\clxx^{K}$ 
we obtain the following. 

\begin{theorem} 
   \label{djksplitting} 
   Let $K$ be a totally homology fillable simplicial complex. Then there are homotopy 
   decompositions 
   \[\Omega(\underline{X},\underline{\ast})^{K}\simeq 
            (\prod_{i=1}^{m}\Omega X_{i})\times\prod_{\substack{I\subseteq [m] \\ I\neq\{i\}}}  
             \Omega F_{I}\qquad 
        \Omega(\underline{X},\underline{\ast})^{K}_{D}\simeq 
            (\prod_{i=1}^{m}\Omega X_{i})\times\prod_{\substack{I\subsetneq [m] \\ I\neq\{i\}}}  
             \Omega F_{I}\] 
   where each $F_{I}$ is homotopy equivalent to a wedge of summands, and each 
   wedge summand is the suspension of a smash product having $\Omega X_{i}$ as 
   a factor for all $i\in I$.~$\qqed$ 
\end{theorem} 

In particular, if each $X_{i}$ is $\mathbb{C}P^{\infty}$ then 
$(\underline{X},\underline{\ast})^{K}= DJ(K)$ and $\Omega X_{i}\simeq S^{1}$. 
So Theorem~\ref{djksplitting} gives a homotopy decomposition for $\Omega DJ(K)$ 
in which every $F_{I}$ is homotopy equivalent to a wedge of spheres.

\Large 
\part{Whitehead products.} 
\normalsize 

\section{Whitehead products and Ganea's Theorem} 
\label{sec:WhGanea} 

In studying thin products it is important to have a good grip on the 
homotopy theory of the wedge $\bigvee_{i=1}^{m} X_{i}$. In this section 
we examine the special case when $m=2$, in Section~\ref{sec:WhPorter} 
we examine the general case when $m\geq 2$, and in Section~\ref{sec:Whthin} 
we relate this to thin products. 

Consider the following two maps: the inclusion 
\(i\colon\namedright{X\vee Y}{}{X\times Y}\) 
of the wedge into the product and the pinch map 
\(q\colon\namedright{X\vee Y}{}{Y}\) 
onto the right wedge summand. Observe that $q$ is also the composite 
\(\nameddright{X\vee Y}{i}{X\times Y}{\pi_{2}}{Y}\), 
where $\pi_{2}$ is the projection onto the second factor. Define 
spaces~$F$ and $G$, and maps $f$ and $g$, by the homotopy pullback diagram 
\[\diagram 
         F\rto\ddouble & G\rto\dto^{g} & X\dto^{i_{1}} \\ 
         F\rto^-{f} & X\vee Y\rto^-{i}\dto^{q} & X\times Y\dto^{\pi_{2}} \\ 
         & Y\rdouble & Y
  \enddiagram\] 
where $i_{1}$ is the inclusion of the first factor. Ganea~\cite{G2} identified the 
homotopy type of $F$ as $\Sigma\Omega X\wedge\Omega Y$ and the homotopy 
classes of $f$ as a Whitehead product. The homotopy type of $G$ is well known 
to be $X\rtimes\Omega Y$ but the homotopy class of $g$ is not readily 
identifiable in terms of otherwise known maps. When 
$X$ is a suspension then $X\rtimes\Omega Y\simeq X\vee(X\wedge\Omega Y)$, 
in which case the homotopy class of $g$ should be identifiable. As 
the author can find no reference for this, we give a proof. There is no 
claim of anything new here, as the methods and results go back to 
Ganea~\cite{G2} and he surely knew everything stated in this section. 

Let 
\(i_{X}\colon\namedright{X}{}{X\vee Y}\) 
and 
\(i_{Y}\colon\namedright{Y}{}{X\vee Y}\) 
be the inclusions of the respective wedge summands. Define $ev_{X}$ and 
$ev_{Y}$ by the composites 
\(ev_{X}\colon\nameddright{\Sigma\Omega X}{ev}{X}{i_{X}}{X\vee Y}\) 
and 
\(ev_{Y}\colon\nameddright{\Sigma\Omega Y}{ev}{Y}{i_{Y}}{X\vee Y}\), 
where $ev$ is the canonical evaluation map.  

\begin{theorem} 
   \label{WhGanea} 
   The following hold: 
   \begin{letterlist} 
      \item There is a homotopy fibration 
                \[\lllnameddright{\Sigma\Omega X\wedge\Omega Y}{[ev_{X},ev_{Y}]} 
                     {X\vee Y}{i}{X\times Y};\] 
      \item there is a homotopy fibration  
                \[\nameddright{X\rtimes\Omega Y}{g}{X\vee Y}{q}{Y}\] 
                where the restriction of $g$ to $X$ is $i_{X}$; 
      \item if $X=\Sigma X'$, then there is a choice of a homotopy equivalence 
                $\Sigma X'\rtimes\Omega Y\simeq\Sigma X'\vee (\Sigma X'\wedge\Omega Y)$ 
                such that the restriction of $g$ to $\Sigma X'\wedge\Omega Y$  
               is $[i_{X},ev_{Y}]$. 
   \end{letterlist} 
\end{theorem} 

\begin{proof} 
The goal is really part~(c) since Ganea~\cite[Equation (9) and Lemma 5.1]{G2} proved 
part~(a) and part~(b) is well known (see, for example, \cite[Theorem 7.7.7]{Se}). However, 
to set up for part~(c) we re-prove the identifications of the homotopy types of the fibres 
of $i$ and $q$ in parts~(a) and~(b). 

In general, for a space~$Z$, let $PZ$ be the path space of $Z$, where paths start 
at the basepoint of $Z$ at time $t=0$. Let 
\(ev_{1}\colon\namedright{PZ}{}{Z}\) 
be the map which evaluates a path at time $t=1$. The homotopy fibre 
of a map 
\(h\colon\namedright{A}{}{Z}\) 
is homotopy equivalent to the topological pullback of $h$ and $ev_{1}$. 

In our case, let $Q$ be the topological pullback of the maps 
\(\namedright{X\vee Y}{i}{X\times Y}\) 
and  
\(\lllnamedright{PX\times PY}{ev_{1}\times ev_{1}}{X\times Y}\). 
Notice that the part of $Q$ sitting over $\ast\vee Y$ is 
$\Omega X\times PY$, the part sitting over $X\vee\ast$ is 
$PX\times\Omega Y$, and the part sitting over $\ast\vee\ast$ 
is $\Omega X\times\Omega Y$. Thus 
$Q=\Omega X\times PY\cup_{\Omega X\times\Omega Y} PX\times\Omega Y$. 
Let $C\Omega X$ be the reduced cone on $\Omega X$, where the 
cone point is at $t=1$. It is well known that there is a homotopy equivalence 
\(\psi_{X}\colon\namedright{C\Omega X}{}{PX}\) 
given by $\psi_{X}(t,\omega)=\omega_{t}$, where $\omega_{t}$ is the 
path defined by $\omega_{t}(s)=\omega((1-t)s)$. Note that $\omega_{0}=\omega$ 
and $\omega_{1}$ is the constant path to the basepoint of $X$. Thus 
$Q\simeq\Omega X\times C\Omega Y\cup_{\Omega X\times\Omega Y} 
       C\Omega X\times\Omega Y$, 
and the right side is the definition of the join $\Omega X\ast\Omega Y$. 
Ganea shows that the composite  
\(\nameddright{\Sigma\Omega X\wedge\Omega Y\simeq\Omega X\ast\Omega Y} 
      {\simeq}{Q}{}{X\vee Y}\) 
is the Whitehead product $[ev_{X},ev_{Y}]$, proving part~(a). 

Next, let $P$ be the topological pullback of the maps 
\(\namedright{X\vee Y}{q}{Y}\) 
and 
\(\namedright{PY}{ev_{1}}{Y}\). 
Notice that the part of $P$ sitting over $\ast\vee Y$ is 
$\ast\times PY$, the part sitting over $X\vee\ast$ is $X\times\Omega Y$, 
and the part sitting over $\ast\vee\ast$ is $\ast\times\Omega Y$. Thus 
$P=\ast\times PY\cup_{\ast\times\Omega Y} X\times\Omega Y$. The composite 
\(\nameddright{X}{}{P}{}{X\vee Y}\) 
is exactly $i_{X}$ and contracting $PY$ we obtain $P\simeq X\rtimes\Omega Y$. 
This proves part~(b). 

For part~(c), observe that the projection 
\(\namedright{X\times Y}{\pi_{2}}{Y}\) 
induces a projection 
\(\namedright{PX\times PY}{\pi_{2}}{PY}\), 
so there is an induced map of topological pullbacks 
\(\namedright{Q}{}{P}\). 
In terms of the identifications for $Q$ and $P$ above, this map of pullbacks is   
\[\namedright{\Omega X\times PY\cup_{\Omega X\times\Omega Y} PX\times\Omega Y} 
      {}{\ast\times PY\cup_{\ast\times\Omega Y} X\times\Omega Y}\] 
where 
\(\llnamedright{PX\times PY}{ev_{1}\times 1}{X\times PY}\) 
has been restricted to the subspaces $\Omega X\times PY$, $PX\times\Omega Y$ 
and $\Omega X\times\Omega Y$. Notice that the composite 
\(\nameddright{C\Omega X}{\psi_{X}}{PX}{ev_{1}}{X}\) 
sends $(t,\omega)$ to $\omega_{t}(1)=\omega(1-t)$. That is, 
$ev_{1}\circ\psi_{X}=-ev$, where 
\(\namedright{\Sigma\Omega X}{ev}{X}\) 
is the canonical evaluation map. Thus the composite  
\(\namedddright{\Sigma\Omega X\wedge\Omega Y\simeq\Omega X\ast\Omega Y} 
     {\simeq}{Q}{}{P\simeq X\rtimes\Omega Y}{}{X\wedge\Omega Y}\) 
is $-ev\wedge 1$. Consequently, if $X=\Sigma X'$ and $s$ is 
the composite 
\(s\colon\lnamedddright{\Sigma X'\wedge\Omega Y}{\Sigma E\wedge 1} 
     {\Sigma\Omega\Sigma X'\wedge\Omega Y}{\simeq}{Q}{}{P\simeq X\rtimes\Omega Y} \), 
then  
\(\nameddright{\Sigma X'\wedge\Omega Y}{\gamma}{\Sigma X'\rtimes\Omega Y}{} 
      {\Sigma X'\wedge\Omega Y}\) 
is homotopic to $-id$, where $id$ is the identity map. Thus $s$ is a 
section for the right map in the cofibration 
\(\nameddright{\Sigma X'}{j}{\Sigma X'\rtimes\Omega Y}{}{\Sigma X'\wedge\Omega Y}\), 
where $j$ is the inclusion. Therefore 
\(\namedright{\Sigma X'\vee (\Sigma X'\wedge\Omega Y)}{j+s}{X\rtimes\Omega Y}\) 
is a homotopy equivalence.  

Finally, from the pullback map 
\(\namedright{Q}{}{P}\) 
and the definition of $s$ there is a homotopy commutative diagram 
\[\diagram 
       \Sigma X'\wedge\Omega Y\rto^-{\Sigma E\wedge 1}\drrto^-{s}  
            & \Sigma\Omega\Sigma X'\wedge\Omega Y\rto^-{\simeq} 
            & Q\rto\dto & X\vee Y\ddouble \\ 
       & & P\rto & X\vee Y. 
  \enddiagram\] 
By part~(a) the map 
\(\namedright{\Sigma\Omega\Sigma X'\wedge\Omega Y\simeq Q}{}{X\vee Y}\) 
is homotopic to $[ev_{X},ev_{Y}]$. Therefore the top row is homotopic 
to $[i_{X},ev_{Y}]$. Hence, choosing the homotopy equivalence for 
$P\simeq\Sigma X'\rtimes\Omega Y$ determined by $j+s$, the restriction of 
\(\namedright{\Sigma X.\rtimes\Omega Y}{}{X\vee Y}\) 
to $\Sigma X'\wedge\Omega Y$ is homotopic to $[i_{X},ev_{Y}]$. 
\end{proof}

\section{Whitehead products and Porter's Theorem} 
\label{sec:WhPorter} 

Define the space $\Gamma(\underline{X})$ and the map $\gamma(\underline{X})$ by 
the homotopy fibration 
\[\nameddright{\Gamma(\underline{X})}{\gamma(\underline{X})}{\bigvee_{i=1}^{m} X_{i}} 
       {}{\prod_{i=1}^{m} X_{i}}.\] 
When $m=2$ Theorem~\ref{WhGanea}~(a) identifies $\Gamma(\underline{X})$ 
as $\Sigma\Omega X_{1}\wedge\Omega X_{2}$ and the homotopy class of 
$\gamma(\underline{X})$ as a Whitehead product $[ev_{X},ev_{Y}]$. 
Porter~\cite{P} partially generalized this by identifying the homotopy type 
of $\Gamma(\underline{X})$ for any $m\geq 2$. 

\begin{theorem} 
   \label{Porter}  
   Let $X_{1},\ldots,X_{m}$ be simply-connnected, pointed $CW$-complexes. 
   Then there is a homotopy equivalence 
   \[\Gamma(\underline{X})\simeq\bigvee_{k=2}^{m} 
        \left(\bigvee_{1\leq i_{1}<\cdots <i_{k}\leq m} (k-1)
        (\Sigma\Omega X_{i_{1}}\wedge\cdots\wedge\Omega X_{i_{k}})\right)\] 
   which is natural for maps 
   \(\namedright{X_{i}}{}{Y_{i}}\).~$\qqed$ 
\end{theorem} 

However, Porter did not identify the homotopy class of the map 
$\gamma(\underline{X})$. It is folklore that it too is determined by  
Whitehead products that depend on evaluation maps, but there 
seems to be no proof of this in the literature. As this is important 
in what follows, in this section we carry out the identification. 

For $1\leq j\leq m$, let  
\[s_{j}\colon\namedright{X_{j}}{}{\bigvee_{i=1}^{m} X_{i}}\] 
be the inclusion of the $j^{th}$-wedge summand. Let $t_{j}$ be 
the composite 
\[t_{j}\colon\nameddright{\Sigma\Omega X_{j}}{ev_{j}}{X_{j}}{s_{j}} 
       {\bigvee_{i=1}^{m} X_{i}}\] 
where $ev_{j}$ is the canonical evaluation map. We say that an iterated Whitehead 
product has \emph{length}~$k$ if it is formed from $k$ ingredient maps. For example, 
$[t_{i_{1}},t_{i_{2}}]$ has length $2$. 
       
\begin{theorem}  
   \label{WhPorter} 
   Let $X_{1},\ldots,X_{m}$ be simply-connnected, pointed $CW$-complexes. 
   Then there is a natural homotopy equivalence 
   \[\Gamma(\underline{X})\simeq
   \bigvee_{k=2}^{m} 
        \left(\bigvee_{1\leq i_{1}<\cdots <i_{k}\leq m} (k-1)
        (\Sigma\Omega X_{i_{1}}\wedge\cdots\wedge\Omega X_{i_{k}})\right)\]
   which may be chosen so that the restriction of $\gamma(\underline{X})$ to 
   $\Sigma\Omega X_{i_{1}}\wedge\cdots\wedge\Omega X_{i_{k}}$ 
   is an iterated Whitehead product of length $k$ formed from the 
   maps $t_{i_{1}},\ldots,t_{i_{k}}$.    
\end{theorem} 

Theorem~\ref{WhPorter} will be proved in two stages. First we consider the 
special case when each $X_{i}=\Sigma Y_{i}$ and in Theorem~\ref{Porterfine} 
identify the homotopy fibre of the inclusion  
\(\namedright{\bigvee_{i=1}^{m}\Sigma Y_{i}}{}{\prod_{i=1}^{m}\Sigma Y_{i}}\),  
but in a formulation different from that of Theorem~\ref{WhPorter}. This formulation 
will also be useful in subsequent sections. Second, we consider the wedge 
sum of evaluation maps 
\(\namedright{\bigvee_{i=1}^{m}\Sigma\Omega X_{i}}{}{\bigvee_{i=1}^{m} X_{i}}\) 
and use the formulation for the first stage in the case of 
$\Sigma Y_{i}=\Sigma\Omega X_{i}$ in order to prove Theorem~\ref{WhPorter}.

\subsection{Identifying the homotopy fibre of the inclusion  
\(\namedright{\bigvee_{i=1}^{m}\Sigma Y_{i}}{}{\prod_{i=1}^{m}\Sigma Y_{i}}\)} 

This starts with a refinement of Theorem~\ref{WhGanea}~(c) in the case of suspensions. 
It is notationally convenient in what follows to pinch to the first wedge summand rather 
than the second. It was shown that there is a homotopy fibration 
\(\nameddright{\Omega\Sigma X\ltimes\Sigma Y}{g}{\Sigma X\vee\Sigma Y}{q}{\Sigma X}\) 
where $q$ is the pinch map, 
$\Omega\Sigma X\ltimes\Sigma Y\simeq(\Omega\Sigma X\wedge\Sigma Y)\vee\Sigma Y$, 
the restriction of $g$ to $\Sigma Y$ is the inclusion $i_{\Sigma Y}$, and the 
restriction of $g$ to $\Omega\Sigma X\wedge\Sigma Y$ is the Whitehead 
product $[ev_{\Sigma X},i_{\Sigma Y}]$. By Lemma~\ref{James} 
there is a homotopy equivalence 
$\Omega\Sigma X\wedge\Sigma Y\simeq\bigvee_{n=1}^{\infty}\Sigma X^{(n)}\wedge Y$. 

For $n\geq 1$, define 
\[ad^{n}(i_{\Sigma X})(i_{\Sigma Y})\colon\namedright{\Sigma X^{(n)}\wedge Y}{}{\Sigma X\vee\Sigma Y}\] 
recursively by $ad^{1}(i_{\Sigma X})(i_{\Sigma Y})=[i_{\Sigma X},i_{\Sigma Y}]$ and, for $n\geq 2$, 
$ad^{n}(i_{\Sigma X})(i_{\Sigma Y})=[i_{\Sigma X},ad^{n-1}(i_{\Sigma X})(i_{\Sigma Y})]$. 
The following was essentially proved in~\cite[Theorem 9.3.1]{N}. 
 
\begin{theorem} 
   \label{Ganeafine} 
   Let $X$ and $Y$ be pointed, path-connected $CW$-complexes. There is 
   a homotopy fibration 
   \[\nameddright{\left(\bigvee_{n=1}^{\infty}\Sigma X^{(n)}\wedge Y\right)\vee\Sigma Y} 
        {\Psi}{\Sigma X\vee\Sigma Y}{q}{\Sigma X}\] 
   where the restriction of $\Psi$ to $\Sigma Y$ is $i_{\Sigma Y}$ and the restriction 
   of $\Psi$ to $\Sigma X^{(n)}\wedge Y$ is the iterated Whitehead product 
   $ad^{n}(i_{\Sigma X})(i_{\Sigma Y})$. All of this is natural for maps 
   \(\namedright{X}{}{X'}\) 
   and 
   \(\namedright{Y}{}{Y'}\). 
\end{theorem} 

\begin{proof} 
The precise statement of~\cite[Theorem 9.3.1]{N} is that the composite 
\begin{equation} 
   \label{NHM} 
   \lllnamedright{\Omega(\left(\bigvee_{n=1}^{\infty}\Sigma X^{(n)}\wedge Y\right)\vee\Sigma Y)\times 
      \Omega\Sigma X}{\Omega\psi\times\Omega i_{\Sigma X}} 
      {\Omega(\Sigma X\vee\Sigma Y)\times\Omega(\Sigma X\vee\Sigma Y)}  
      \stackrel{\mu}{\larrow}{\Omega(\Sigma X\vee\Sigma Y)} 
\end{equation} 
is a homotopy equivalence, where $\mu$ is the loop multiplication. We will show that 
this implies the existence of the homotopy fibration in the statement of the theorem. 

Observe that the naturality of the Whitehead product, the fact that 
$q\circ i_{\Sigma X}$ is the identity map on $\Sigma X$, and the fact that $q\circ i_{\Sigma Y}$ 
is null homotopic together imply that 
\[q\circ ad^{1}(i_{\Sigma X})(i_{\Sigma Y})\simeq ad^{1}(q\circ i_{\Sigma X})(q\circ i_{\Sigma Y})\simeq 
      ad^{1}(1_{\Sigma X})(\ast)=[i_{\Sigma X},\ast]\simeq\ast.\] 
Inducting on the definition of $ad^{n}$ then implies that 
$q\circ ad^{n}(i_{\Sigma X})(i_{\Sigma Y})\simeq\ast$ for each $n\geq 1$. Thus 
$q\circ\Psi$ is null homotopic. Recall from the discussion preceding the theorem 
that the homotopy fibre of $q$ is $\Omega\Sigma X\ltimes\Sigma Y$. So we obtain a lift 
\begin{equation} 
  \label{HNMdgrm} 
  \diagram 
      & \left(\bigvee_{n=1}^{\infty}\Sigma X^{(n)}\wedge Y\right)\vee\Sigma Y\dto^{\Psi}\dlto_(0.6){\ell} \\ 
      \Omega\Sigma X\ltimes\Sigma Y\rto^-{g} & \Sigma X\vee\Sigma Y 
  \enddiagram 
\end{equation} 
for some map $\ell$. If $\ell$ is a homotopy equivalence, then $\Psi$ may be substituted 
for the homotopy fibre of $q$, giving the homotopy fibration asserted by the theorem. 

It remains to show that $\ell$ is a homotopy equivalence. The homotopy 
equivalence~(\ref{NHM}) implies that the map $\Omega\Psi$ induces an injection 
in homology. Looping the homotopy commutative triangle in~(\ref{HNMdgrm}) 
then implies that $\Omega\ell$ induces an injection in homology. In the discussion 
preceding the theorem it was observed that 
$\left(\bigvee_{n=1}^{\infty}\Sigma X^{(n)}\wedge Y\right)\vee\Sigma Y$ 
and $\Omega\Sigma X\ltimes\Sigma Y$ have the same homotopy type. As both 
spaces are simply-connected, their loop spaces then also have the same homotopy 
type. Thus $(\Omega\ell)_{\ast}$ being an injection implies that it is an isomorphism. 
Now Dror's~\cite{Dr} generalization of Whitehead's Theorem implies that $\Omega\ell$ 
is a homotopy equivalence. Hence $\ell$ induces an isomorphism on homotopy 
groups, so as all spaces have the homotopy type of $CW$-complexes, $\ell$ is 
a homotopy equivalence. 
\end{proof} 

We will use Theorem~\ref{Ganeafine} iteratively. Let $Y_{1},\ldots,Y_{m}$ be pointed, 
path-connected $CW$-complexes. Let $A_{0}=\bigvee_{i=1}^{m}\Sigma Y_{i}$ and 
for $1\leq j\leq m$ let 
\[s_{j}\colon\namedright{\Sigma Y_{j}}{}{\bigvee_{i=1}^{m}\Sigma Y_{i}}\] 
be the inclusion of the $j^{th}$-wedge summand. Let 
$B_{1}=\bigvee_{i=2}^{m} Y_{i}$ so that $A_{0}=\Sigma Y_{1}\vee\Sigma B_{1}$. 
By Theorem~\ref{Ganeafine} there is a homotopy fibration 
\[\llnameddright{A_{1}}{\Psi_{1}}{\Sigma Y_{1}\vee\Sigma B_{1}}{q_{1}}{\Sigma Y_{1}}\] 
where $q_{1}$ is the pinch map, 
$A_{1}=\left(\bigvee_{n=1}^{\infty}\Sigma Y_{1}^{(n)}\wedge B_{1}\right)\vee\Sigma B_{1}$, 
the restriction of $\Psi_{1}$ to $\Sigma B_{1}$ is the inclusion $i_{\Sigma B_{1}}$, and the restriction 
of $\Psi_{1}$ to $\Sigma Y_{1}^{(n)}\wedge B_{1}$ is the iterated Whitehead product  
$ad^{n}(i_{\Sigma Y_{1}})(i_{\Sigma B_{1}})$. Since $B_{1}$ is the wedge 
$\bigvee_{i=2}^{m} Y_{i}$ and the inclusion $i_{\Sigma B_{1}}$ is the wedge sum 
of the inclusions $s_{j}$ for $2\leq j\leq m$, and since $i_{\Sigma Y_{1}}=s_{1}$, 
the Whitehead product $ad^{n}(i_{\Sigma Y_{1}})(i_{\Sigma B_{1}})$ is the wedge sum 
of the Whitehead products $ad^{n}(s_{1})(s_{j})$ for $2\leq j\leq m$. 

Observe that $A_{1}$ is a suspension and has $\Sigma Y_{2}$ as a wedge summand 
(via it being a wedge summand of $\Sigma B_{1})$. Write 
\[A_{1}=\Sigma Y_{2}\vee\Sigma B_{2}\] 
where 
$B_{2}=(\bigvee_{i=3}^{m} Y_{i})\vee\left(\bigvee_{n=1}^{\infty} Y_{1}^{(n)}\wedge B_{1}\right)$. 
Now iterate this process by considering the pinch map  
\(\lnamedright{\Sigma Y_{2}\vee\Sigma B_{2}}{q_{2}}{\Sigma Y_{2}}\).  
We obtain a sequence of homotopy fibrations 
\[\llnameddright{A_{k}}{\Psi_{k}}{\Sigma Y_{k}\vee\Sigma B_{k}}{q_{k}}{\Sigma Y_{k}}\] 
for $1\leq k\leq m$, where $A_{k-1}=\Sigma Y_{k}\vee\Sigma B_{k}$, $q_{k}$ is the pinch map, 
$A_{k}=\left(\bigvee_{n=1}^{\infty}\Sigma Y_{k}^{(n)}\wedge B_{k}\right)\vee\Sigma B_{k}$, 
the restriction of~$\Psi_{k}$ to $\Sigma B_{k}$ is the inclusion $i_{\Sigma B_{k}}$ into 
$\Sigma Y_{k}\vee\Sigma B_{k}$, and the restriction of $\Psi_{k}$ to 
$\Sigma Y_{k}^{(n)}\wedge B_{k}$ is the iterated Whitehead product  
$ad^{n}(i_{\Sigma Y_{k}})(i_{\Sigma B_{k}})$. Consider the maps $i_{\Sigma Y_{k}}$ 
and $i_{\Sigma B_{k}}$ composing into $\bigvee_{i=1}^{m}\Sigma Y_{i}$ by the composite 
\[\Phi_{k-1}\colon\lnamedright{A_{k-1}=\Sigma Y_{k}\vee\Sigma B_{k}}
      {\Psi_{k-1}}{A_{k-2}}\larrow\cdots\larrow\llnamedright{A_{1}}{\Psi_{1}} 
      {A_{0}=\bigvee_{i=1}^{m}\Sigma Y_{i}}.\] 
In the iteration we have 
$B_{k}=(\bigvee_{i=k+1}^{m} Y_{i})\vee\left(\bigvee_{n=1}^{\infty} Y_{k-1}^{(n)}\wedge B_{k-1}\right)$, 
the restriction of $\Phi_{k-1}$ to $\bigvee_{i=k+1}^{m}\Sigma B_{i}$ is the inclusion, the restriction 
of $\Phi_{k-1}$ to $\bigvee_{n=1}^{\infty}\Sigma Y_{k-1}^{(n)}\wedge B_{k-1}$ is a wedge sum of 
iterated Whitehead products formed from the maps $s_{j}$ for $1\leq j\leq m$, and the 
restriction of $\Phi_{k-1}$ to $\Sigma Y_{k}$ is the inclusion. Thus $\Phi_{k-1}\circ i_{\Sigma B_{k}}$ is a 
wedge sum of the inclusions $s_{j}$ for $k+1\leq j\leq m$ and iterated Whitehead products 
of the $s_{j}$'s for $1\leq j\leq m$, and $\Phi_{k-1}\circ i_{\Sigma Y_{k}}$ is the inclusion $s_{j}$.  
Hence $\Phi_{k}=\Phi_{k-1}\circ\Psi_{k}$ is the wedge sum of the inclusions $s_{j}$ for $k+1\leq j\leq m$ 
and iterated Whitehead products of the $s_{j}$'s for $1\leq j\leq m$. 

Consequently, at the end of the iteration, we have $A_{m}$ a wedge of suspensions 
and the composite 
\[\Phi_{m}\colon\namedright{A_{m}}{\Psi_{m}}{A_{m-1}}\longrightarrow\cdots\longrightarrow 
       \llnamedright{A_{1}}{\Psi_{1}}{\bigvee_{i=1}^{m}\Sigma Y_{i}}\] 
is a wedge sum of iterated Whitehead products formed from the maps $s_{j}$ 
for $1\leq j\leq m$. 

\begin{theorem} 
  \label{Porterfine} 
  There is a homotopy fibration 
  \[\nameddright{A_{m}}{\Phi_{m}}{\bigvee_{i=1}^{m}\Sigma Y_{i}}{}{\prod_{i=1}^{m}\Sigma Y_{i}}\] 
  which is natural for maps 
  \(\namedright{Y_{i}}{}{Y'_{i}}\). 
\end{theorem} 
      
\begin{proof} 
Noting that $\Phi_{1}=\Psi_{1}$, we start with the homotopy fibration 
\(\lnameddright{A_{1}}{\Phi_{1}}{\bigvee_{i=1}^{m}\Sigma Y_{i}}{q_{1}}{\Sigma Y_{1}}\). 
Assume inductively that for some $2\leq k\leq m$ that there is a homotopy fibration 
\(\nameddright{A_{k-1}}{\Phi_{k-1}}{\bigvee_{i=i}^{m}\Sigma Y_{i}}{}{\prod_{i=1}^{k-1}\Sigma Y_{i}}\).  
Observe that the map 
\(\namedright{A_{k-1}}{q_{k}}{\Sigma Y_{k}}\) 
factors as the composite 
\(\lnameddright{A_{k-1}}{\Phi_{k-1}}{\bigvee_{i=1}^{m}\Sigma Y_{i}}{}{\Sigma Y_{k}}\). 
Therefore there is a homotopy fibration diagram  
\[\diagram 
       A_{k-1}\rto^-{\Phi_{k-1}}\dto^{q_{k}} & \bigvee_{i=1}^{m}\Sigma Y_{i}\rto\dto 
            & \prod_{i=k}^{m}\Sigma Y_{i}\ddouble \\ 
       \Sigma Y_{k-1}\rto & \prod_{i=k-1}^{m}\Sigma Y_{i}\rto^-{\pi_{k-1}}  
            & \prod_{i=1}^{k-1} \Sigma Y_{i}    
  \enddiagram\] 
where $\pi_{k-1}$ is the projection. Since the left square is a homotopy pullback, from 
the homotopy fibration 
\(\lnameddright{A_{k}}{\Psi_{k}}{A_{k-1}}{q_{k-1}}{\Sigma Y_{k}}\) 
and the fact that $\Phi_{k}=\Phi_{k-1}\circ\Psi_{k}$, we obtain a homotopy fibration 
\(\lnameddright{A_{k}}{\Phi_{k}}{\bigvee_{i=1}^{m}\Sigma Y_{i}}{} 
     {\prod_{i=k-1}^{m}\Sigma Y_{i}}\). 
By induction, such a homotopy fibration exists for all \mbox{$2\leq k\leq m$}, so when $k=m$ 
we obtain a homotopy fibration 
\(\nameddright{A_{m}}{\Phi_{m}}{\bigvee_{i=1}^{m}\Sigma Y_{i}}{}{\prod_{i=1}^{m}\Sigma Y_{i}}\) 
as asserted. 
 
Finally, since $\Phi_{m}$ consists of iterated Whitehead products formed from the 
inclusions $s_{j}$ for \mbox{$1\leq j\leq m$}, the naturality of the Whitehead product and 
the inclusions implies that this homotopy fibration is natural for maps 
\(\namedright{Y_{i}}{}{Y'_{i}}\). 
\end{proof}

\subsection{Identifying the homotopy fibre of the inclusion 
\(\namedright{\bigvee_{i=1}^{m} X_{i}}{}{\prod_{i=1}^{m} X_{i}}\)} 
We turn to Theorem~\ref{WhPorter}. The idea is to use the information in 
Theorem~\ref{Porterfine} in the case of $\bigvee_{i=1}^{m}\Sigma\Omega X_{i}$ and 
compose it with evaluation maps to $\bigvee_{i=1}^{m} X_{i}$. We first consider 
the basic case of Theorem~\ref{Ganeafine} applied to $\Sigma\Omega X\vee\Sigma\Omega Y$ 
and compose with evaluation map to $X\vee Y$. Note here that we assume~$X$ and $Y$ are 
simply-connected so that $\Omega X$ and $\Omega Y$ are path-connected, as needed in 
Theorem~\ref{Ganeafine}. 

\begin{lemma} 
   \label{sevlift} 
   Let $X$ and $Y$ be simply-connected, pointed $CW$-complexes. There is a lift 
   \[\diagram 
          \bigvee_{n=1}^{\infty}\Sigma\Omega X^{(n)}\wedge\Omega Y 
                   \xto[rrr]^-{\bigvee_{n=1}^{\infty} ad^{n}(i_{\Sigma\Omega X})(i_{\Sigma\Omega Y})} 
                   \ddashed|>\tip^{\lambda} 
                & & & \Sigma\Omega X\vee\Sigma\Omega Y\dto^{ev\vee ev} \\ 
          \Sigma\Omega X\wedge\Omega Y\xto[rrr]^-{[ev_{X},ev_{Y}]} & & & X\vee Y 
     \enddiagram\] 
   for some map $\lambda$. 
\end{lemma} 

\begin{proof} 
By Theorem~\ref{WhGanea}~(a), there is a homotopy fibration 
\(\lllnameddright{\Sigma\Omega X\wedge\Omega Y}{[ev_{X},ev_{Y}]}{X\vee Y}{i}{X\times Y}\). 
The naturality of the Whitehead product implies that the composition 
$(ev_{X}\vee ev_{Y})\circ  ad^{n}(i_{\Sigma\Omega X})(i_{\Sigma\Omega Y})$ 
is the Whitehead product $ad^{n}(ev_{X})(ev_{Y})$. Every Whitehead product into $X\vee Y$ 
composes trivially with $i$ and so lifts through $[ev_{X},ev_{Y}]$. The lemma now follows. 
\end{proof} 

\begin{remark} 
\label{sevliftremark} 
In analogy with the definition of $t_{j}$, let $t_{X}=ev\circ i_{\Sigma\Omega X}$ and 
$t_{Y}=ev\circ i_{\Sigma\Omega Y}$. Then Lemma~\ref{sevlift} equivalently says 
that $ad^{n}(t_{X})(t_{Y})$ factors through $[ev_{X},ev_{Y}]=[t_{X},t_{Y}]$. 
\end{remark} 

Recall that the map $t_{j}$ is defined as the composite 
\(\nameddright{\Sigma\Omega X_{j}}{ev_{j}}{X_{j}}{s_{j}}{\bigvee_{i=1}^{m} X_{i}}\). 
The naturality of the evaluation map implies that $t_{j}$ can also be regarded as the composite 
\[t_{j}\colon\lllnameddright{\Sigma\Omega X_{i}}{s_{j}}{\bigvee_{i=1}^{m}\Sigma\Omega X_{i}} 
        {\bigvee_{i=1}^{m} ev_{i}}{\bigvee_{i=1}^{m} X_{i}}.\]    
         
Now run through the construction of the spaces $A_{k}$ with the input being 
$Y_{i}=\Omega X_{i}$ for $1\leq i\leq m$. Starting with $A_{0}=\bigvee_{i=1}^{m}\Sigma\Omega X_{i}$, 
the first homotopy fibration 
\(\nameddright{A_{1}}{\Psi_{1}}{A_{0}}{q_{0}}{\Sigma\Omega X_{1}}\) 
has 
\[A_{1}=\bigvee_{n=1}^{\infty}\left(\Omega X_{1}^{(n)}\wedge(\bigvee_{i=2}^{m}\Sigma\Omega X_{i})\right) 
       \vee\left(\bigvee_{i=2}^{m}\Sigma\Omega X_{i}\right),\] 
the restriction of $\Psi_{1}$ to $\bigvee_{i=1}^{m}\Sigma\Omega X_{i}$ is a wedge of 
the inclusions $s_{j}$ for $2\leq j\leq m$, and the restriction of $\Psi_{1}$ to 
$\Omega X_{1}^{(n)}\wedge(\bigvee_{i=2}^{m}\Sigma\Omega X_{i})$ 
is a wedge of iterated Whitehead products $ad^{n}(s_{1})(s_{j})$ for $2\leq j\leq m$. Composing 
$\Phi_{1}=\Psi_{1}$ with the wedge of evaluation maps 
\[e\colon\lllnamedright{\bigvee_{i=1}^{m}\Sigma\Omega X_{i}}{\bigvee_{i=1}^{m} ev_{i}} 
       {\bigvee_{i=1}^{m} X_{i}}\]  
the maps $s_{j}$ are sent to $t_{j}$ and, by Lemma~\ref{sevlift}, 
the iterated Whitehead products $ad^{n}(s_{1})(s_{j})$ factor through 
$[t_{1},t_{j}]$. At the next stage, writing $A_{1}=\Sigma B_{2}\vee\Sigma\Omega X_{2}$, 
we have 
\[A_{2}=\bigvee_{n=1}^{\infty}\left(\Omega X_{2}^{(n)}\wedge\Sigma B_{2}\right) 
       \vee\Sigma B_{2}.\] 
First, $\Sigma B_{2}$ consists of all wedge summands in $A_{1}$ 
except $\Sigma\Omega X_{2}$, and $e\circ\Phi_{2}$ maps these into 
$\bigvee_{i=1}^{m} X_{i}$ exactly as does $e\circ\Phi_{1}$, 
that is, by a~$t_{j}$ if $3\leq j\leq m$ (excluding~$t_{2}$ since $\Sigma\Omega X_{2}$ 
is not in $B_{2}$) or by factoring through a $[t_{1},t_{j}]$ for \mbox{$2\leq j\leq m$}. 
Second, $\Phi_{2}$ maps the summands 
$\bigvee_{n=1}^{\infty}\left(\Omega X_{2}^{(n)}\wedge\Sigma B_{2}\right)$ 
into $\bigvee_{i=1}^{m}\Sigma\Omega X_{i}$ by iterated Whitehead products 
$ad^{n}(s_{2})(f)$ where $f$ comes from~$\Phi_{1}$. 
Composing with the wedge of evaluations $e$, the naturality of the Whitehead 
product implies that $e\circ f$ factors through a $t_{j}$ for $3\leq j\leq m$ 
or a $[t_{1},t_{j}]$ for $2\leq j\leq m$, and Lemma~\ref{sevlift} then implies that 
\mbox{$e\circ ad^{n}(s_{2})(f)=ad^{n}(e\circ s_{2})(e\circ f)$} 
factors through $[t_{2},t_{j}]$ for $3\leq j\leq m$ or $[t_{2},[t_{1},t_{j}]]$ 
for $2\leq j\leq m$. Refining a bit, observe that if $j=2$ then 
$[t_{2},[t_{1},t_{2}]]\simeq [t_{1},[t_{2},t_{1}]]\circ (\Sigma 1\wedge T)$ 
where~$1$ is the identity map on $\Omega X_{2}$ and 
\(T\colon\namedright{\Omega X_{1}\wedge\Omega X_{2}}{}{\Omega X_{2}\wedge\Omega X_{1}}\) 
interchanges wedge summands. But $[t_{2},[t_{2},t_{1}]]=ad^{2}(t_{2})(t_{1})$ and, 
by Remark~\ref{sevliftremark}, $ad^{2}(t_{2})(t_{1})$ factors through $[t_{2},t_{1}]$. 
Thus, in fact, \mbox{$e\circ ad^{n}(s_{2})(f)=ad^{n}(e\circ s_{2})(e\circ f)$} 
factors through $[t_{2},t_{j}]$ for $3\leq j\leq m$ or $[t_{2},[t_{1},t_{j}]]$ 
for $3\leq j\leq m$.

Continuing in this manner we see that $e\circ\Phi_{m}$ factors through 
a wedge sum of the Whitehead products $[t_{r_{s}},[t_{r_{s-1}},[\cdots,[t_{r_{1}},t_{j}]\cdots ]$ 
where $j>r_{1}$, $1\leq r_{1}<\cdots <r_{s}\leq m$, and $j\notin\{r_{1},\ldots,r_{s}\}$.  

\begin{remark} 
\label{reindex} 
We need to re-index the enumeration of these Whitehead products. Let $C$ be the 
collection of iterated Whitehead products of the form  
$[t_{r_{s}},[t_{r_{s-1}},[\cdots,[t_{r_{1}},t_{j}]\cdots ]$ 
where $j>r_{1}$ and $1\leq r_{1}<\cdots <r_{s}\leq m$. 
Fixing a sequence $(i_{1},\ldots,i_{k})$ with $1\leq i_{1}<\cdots <i_{k}\leq m$, 
we wish to determine how many brackets in $C$ are formed from this sequence. 
Note that $r_{1}$ must be the smallest index, so we must take $r_{1}=i_{1}$. Note 
that $j$ only needs to be bigger than $r_{1}$, so there are $k-1$ possible choices for $j$; 
we can take $j=i_{s}$ for $2\leq s\leq k$. Suppose that $j=i_{s_{0}}$. Then the 
remaining terms are $i_{2},\ldots,\hat{i}_{s_{0}},\ldots,i_{k}$ where $\hat{i}_{s_{0}}$ 
is removed from the list. As this list is ordered, and we are substituting them in 
for the ordered list $r_{2},\ldots,r_{k-1}$, we are forced to take 
$(r_{2},\ldots,r_{k})=(i_{2},\ldots,\hat{i}_{s_{0}},\ldots,i_{k})$. Thus in total 
there are $k-1$ possible brackets in $C$ we can form from $(i_{1},\ldots,i_{k})$. 

Conversely, any element in $C$ corresponds to a sequence $(i_{1},\ldots,i_{s+1})$ 
by simply taking the sequence $(r_{1},\ldots,r_{s})$ and inserting $j$ in order. 
Thus there is a one-to-one correspondence between $C$ and the collection of $k-1$ brackets 
for each sequence $(i_{1},\ldots,i_{k})$ with $1\leq i_{1}<\cdots <i_{k}\leq m$. 
\end{remark} 

To compress notation, let 
\begin{equation} 
  \label{Wnotation} 
  W(\underline{X})=\bigvee_{k=2}^{m} 
        \left(\bigvee_{1\leq i_{1}<\cdots <i_{k}\leq m} (k-1)
        (\Sigma\Omega X_{i_{1}}\wedge\cdots\wedge\Omega X_{i_{k}})\right). 
\end{equation}  
Using Remark~\ref{reindex} to correspond indexing sets, let 
\[\Theta(\underline{X})\colon\namedright{W(\underline{X})}{}{\bigvee_{i=1}^{m} X_{i}}\] 
be the wedge sum of the Whitehead products in $C$. We have proved the following. 

\begin{lemma} 
   \label{Pfact} 
   Let $X_{1},\ldots,X_{m}$ be simply-connected, pointed $CW$-complexes. 
   Then there is a factorization 
   \[\diagram 
           A_{m}\rto^-{\Phi_{m}}\dto^{\epsilon} 
                & \bigvee_{i=1}^{m}\Sigma\Omega X_{i}\dto^{\bigvee_{i=1}^{m} ev} \\ 
           W(\underline{X})\rto^-{\Theta(\underline{X})} & \bigvee_{i=1}^{m} X_{i}  
      \enddiagram\] 
   for some map $\epsilon$.~$\qqed$ 
\end{lemma} 

Lemma~\ref{Pfact} will play a key role in the proof of Theorem~\ref{WhPorter}. 
Before starting the proof, some additional preliminary information is needed. 

\begin{lemma} 
   \label{loopgammainverse} 
   The homotopy fibration 
   \(\nameddright{\Gamma(\underline{X})}{\gamma(\underline{X})} 
        {\bigvee_{i=1}^{m} X_{i}}{}{\prod_{i=1}^{m} X_{i}}\) 
   has the property that $\Omega\gamma(\underline{X})$ has a left homotopy inverse.   
\end{lemma} 
                
\begin{proof} 
For $1\leq j\leq m$, the composite 
\(\nameddright{X_{j}}{s_{j}}{\bigvee_{i=1}^{m} X_{i}}{}{\prod_{i=1}^{m} X_{i}}\)  
is the inclusion of the $j^{th}$ factor. After looping the maps $\Omega s_{j}$ 
can be multiplied together to obtain a right homotopy inverse $s$ for the map  
\(\namedright{\Omega(\bigvee_{i=1}^{m} X_{i})}{}{\prod_{i=1}^{m}\Omega X_{i}}\). 
Consequently there is a homotopy equivalence 
\[\lllnameddright{\prod_{i=1}^{m}\Omega X_{i}\times\Omega\Gamma(\underline{X})} 
     {s\times\Omega\gamma(\underline{X})} 
     {\Omega(\bigvee_{i=1}^{m} X_{i})\times\Omega(\bigvee_{i=1}^{m} X_{i})} 
     {\mu}{\Omega(\bigvee_{i=1}^{m} X_{i})}\] 
where $\mu$ is the loop multiplication. Therefore $\Omega\gamma(\underline{X})$ 
has a left homotopy inverse. 
\end{proof} 
        
For $1\leq i\leq m$, we have the evaluation map 
\(ev_{i}\colon\namedright{\Sigma\Omega X_{i}}{}{X_{i}}\).  
Let 
\[E_{i}\colon\namedright{\Omega X_{i}}{}{\Omega\Sigma\Omega X_{i}}\] 
be the suspension map.. Since $ev_{i}$ is the left adjoint of the identity 
map on $\Omega X_{i}$ and $E_{i}$ is the right adjoint of the identity 
map on $\Sigma\Omega X_{i}$, the composite $\Omega ev_{i}\circ E_{i}$ is homotopic to the 
identity map on $\Omega X_{i}$. Building on this, for any sequence 
$1\leq i_{1}<\cdots <i_{k}\leq m$, there is a composite 
\[\nameddright{\Sigma\Omega X_{i_{1}}\wedge\cdots\wedge\Omega X_{i_{k}}}{} 
     {\Sigma\Omega\Sigma\Omega X_{i_{1}}\wedge\cdots\wedge\Omega\Sigma\Omega X_{i_{k}}} 
     {}{\Sigma\Omega X_{i_{1}}\wedge\cdots\wedge\Omega X_{i_{k}}}\] 
where the left map is $\Sigma E_{i_{1}}\wedge\cdots\wedge E_{i_{k}}$, the right map 
is $\Sigma\Omega ev_{i_{1}}\wedge\cdots\wedge\Omega ev_{i_{k}}$, and the 
composite is homotopic to the identity map. Doing this for each of the wedge summands 
in~(\ref{Wnotation}) we obtain maps   
\begin{equation} 
  \label{Wequiv} 
  \nameddright{W(\underline{X})}{\xi}{W(\underline{\Sigma\Omega X})}{\zeta}{W(\underline{X})} 
\end{equation}  
whose composite is a homotopy equivalence. 

\begin{proof}[Proof of Theorem~\ref{WhPorter}] 
Consider the diagram 
\begin{equation} 
  \label{crazydgrm} 
  \spreaddiagramcolumns{2pc}\diagram 
      & & \Omega W(\underline{X})\dto^{\Omega\xi} \\ 
      \Omega A_{m}\drto^{\Omega\Phi_{m}}\ddto^{\Omega\epsilon} 
       & & \Omega W(\underline{\Sigma\Omega X})\llto_-{\Omega\theta}  
           \dlto_{\Omega\gamma(\underline{\Sigma\Omega X})} 
           \ddto^{\Omega\zeta} \\ 
      & \Omega(\bigvee_{i=1}^{m}\Sigma\Omega X_{i}) 
           \ddto^(0.4){\Omega(\bigvee_{i=1}^{m} ev_{i})} & \\ 
      \Omega W(\underline{X})\drto^{\Omega\Theta(\underline{X})} 
       & & \Omega W(\underline{X})\dlto_{\Omega\gamma(\underline{X})} \\     
       & \Omega(\bigvee_{i=1}^{m} X_{i})\dto^{Q(\underline{X})} & \\ 
       & \Omega W(\underline{X}). & 
  \enddiagram 
\end{equation}  
Here, $Q(\underline{X})$ is a left homotopy inverse of $\Omega\gamma(\underline{X})$ 
that exists by Lemma~\ref{loopgammainverse}. The left quadrilateral homotopy 
commutes by Lemma~\ref{Pfact}, and the right quadrilateral homotopy commutes 
by the naturality of Theorem~\ref{Porter}. For the triangle, by Lemma~\ref{Pfact} 
and Theorem~\ref{Porter} respectively, both 
\(\nameddright{A_{m}}{\Phi_{m}}{\bigvee_{i=1}^{m}\Sigma\Omega X_{i}}{} 
      {\prod_{i=1}^{m}\Sigma\Omega X_{i}}\) 
and 
\(\llnameddright{W(\underline{\Sigma\Omega X})}{\gamma(\underline{\Sigma\Omega X})}
      {\bigvee_{i=1}^{m}\Sigma\Omega X_{i}}{}{\prod_{i=1}^{m}\Sigma\Omega X_{i}}\) 
are homotopy fibrations, so there is a homotopy equivalence  
\(\theta\colon\namedright{W(\underline{\Sigma\Omega X})}{}{A_{m}}\) 
such that $\gamma(\underline{\Sigma\Omega X})\simeq\Phi_{m}\circ\theta$. Looping 
this give the homotopy commutativity of the upper triangle above. Therefore the entire 
diagram homotopy commutes. 

By~(\ref{Wequiv}), the composite $\Omega\zeta\circ\Omega\xi$ is a homotopy 
equivalence and, by definition, $Q(\underline{X})$ is a left homotopy inverse for 
$\Omega\gamma(\underline{X})$. Therefore, on the right side of the diagram, 
the composite $Q(\underline{X})\circ\Omega\gamma(\underline{X})\circ\Omega\zeta\circ\Omega\xi$ 
is a homotopy equivalence. The homotopy commutativity of the diagram therefore implies that 
$Q(\underline{X})\circ\Omega\Theta(\underline{X})\circ\Omega\epsilon\circ\Omega\theta\circ\Omega\xi$ 
is a homotopy equivalence. Let $a=\Omega\epsilon\circ\Omega\theta\circ\Omega\xi$, let 
$b=Q(\underline{X})\circ\Omega\Theta(\underline{X})$ and 
consider the composite of self-maps 
\[\nameddright{\Omega W(\underline{X})}{a}{\Omega W(\underline{X})}{b}{\Omega W(\underline{X})}.\] 
We have just seen that $b\circ a$ is a homotopy equivalence. This implies that, in 
integral homology, $a_{\ast}$ is an injection. Since $\hlgy{\Omega W(\underline{X})}$ is a 
finite type module, a self-map which is an injection must be an isomorphism, and hence 
$a_{\ast}$ is an isomorphism. Noting 
that $W(\underline{X})$ is a suspension so that $\Omega W(\underline{X})$ is nilpotent,  
Dror's~\cite{Dr} generalization of Whitehead's Theorem implies that $a$ is a homotopy 
equivalence. (While we do not strictly need this, notice that $a$ being a homotopy 
equivalence implies that~$b$ is also a homotopy equivalence, since $b=(b\circ a)\circ a^{-1}$ 
and both $b\circ a$ and $a^{-1}$ are homotopy equivalences.) 

Let $c=\Omega\zeta\circ\Omega\xi\circ a^{-1}$. As both $\Omega\zeta\circ\Omega\xi$ 
and $a^{-1}$ are homotopy equivalences, so is $c$. Further, the definition of $a$ and 
the homotopy commutativity of~(\ref{crazydgrm}) implies that 
\[\Omega\Theta(\underline{X})\simeq \Omega\Theta(\underline{X})\circ a\circ a^{-1}= 
      \Omega\Theta(\underline{X})\circ\Omega\epsilon\circ\Omega\theta\circ\Omega\xi\circ a^{-1}\simeq 
      \Omega\gamma(\underline{X})\circ\Omega\zeta\circ\Omega\xi\circ a^{-1}=
      \Omega\gamma(\underline{X})\circ c.\] 
Thus there is a homotopy commutative diagram 
\begin{equation} 
  \label{Wc} 
  \diagram 
       & &\Omega W(\underline{X})\dto^{\Omega\Theta(\underline{X})}\dllto_{c^{-1}} \\ 
       \Omega W(\underline{X})\rrto^-{\Omega\gamma(\underline{X})} 
            & & \Omega(\bigvee_{i=1}^{m} X_{i}). 
   \enddiagram 
\end{equation}  

We claim that~(\ref{Wc}) deloops. In general, if 
\(g\colon\namedright{\Sigma A}{}{\Sigma B}\) 
is a map with the property that $\Omega g$ is a homotopy equivalence, 
then $g$ is a homotopy equivalence. For $\Omega g$ induces an isomorphism 
on homotopy groups and therefore, as both $\Sigma A$ and $\Sigma B$ are 
simply-connected, so does $g$. As we are assuming spaces are $CW$-complexes, 
this implies that $g$ is a homotopy equivalence. In our case, by definition,  
$a=\Omega\bar{a}$ for $\bar{a}=\epsilon\circ\theta\circ\xi$. As $\Omega\bar{a}$ is a homotopy 
equivalence, so is $\bar{a}$. Therefore $a^{-1}\simeq\Omega(\bar{a}^{-1})$, implying 
that $c=\Omega\bar{c}$ for $\bar{c}=\zeta\circ\xi\circ\bar{a}^{-1}$. Again, as $c$ is 
a homotopy equivalence, so is $\bar{c}$. Since $W(\underline{X})$ is a suspension 
it retracts off $\Sigma\Omega W(\underline{X})$, implying that for any space $B$ the map 
\(\namedright{[W(\underline{X}),B]}{}{[\Omega W(\underline{X}),\Omega B]}\) 
sending $g$ to $\Omega g$ is an injection. Hence the homotopy commutativity 
of~(\ref{Wc}), with $c^{-1}$ written as $\Omega\bar{c}^{-1}$, deloops to give a 
homotopy commutative diagram 
\begin{equation} 
  \label{Wc2} 
  \diagram 
       & & W(\underline{X})\dto^{\Theta(\underline{X})}\dllto_{\bar{c}^{-1}} \\ 
       W(\underline{X})\rrto^-{\gamma(\underline{X})} 
            & & \bigvee_{i=1}^{m} X_{i}  
   \enddiagram 
\end{equation}  
where $\bar{c}^{-1}$ is a homotopy equivalence. 

By Theorem~\ref{Porter} there is a homotopy fibration 
\(\nameddright{W(\underline{X})}{\gamma(\underline{X})}{\bigvee_{i=1}^{m} X_{i}} 
      {}{\prod_{i=1}^{m} X_{i}}\). 
Thus~(\ref{Wc2}) implies that there is a homotopy fibration 
\(\nameddright{W(\underline{X})}{\Theta(\underline{X})}{\bigvee_{i=1}^{m} X_{i}} 
      {}{\prod_{i=1}^{m} X_{i}}\). 
This proves the first assertion of the theorem. The naturality assertion follows from 
the fact that $\Theta(\underline{X})$ is natural because the Whitehead product 
is natural. 
\end{proof}

\section{Thin products and Whitehead products}  
\label{sec:Whthin} 

Recall from the Introduction that if $X_{1},\ldots,X_{m}$ are pointed spaces 
then the thin product is defined by 
$P^{m}(\underline{X})=(\underline{X},\underline{\ast})^{K}_{D}$ 
where $K$ is the simplicial complex consisting of $m$ disjoint points. 
Noting that $(\underline{X},\underline{\ast})^{K}\simeq\bigvee_{i=1}^{m} X_{i}$, 
from the map 
\(\namedright{(\underline{X},\underline{\ast})^{K}}{\varphi} 
      {(\underline{X},\underline{\ast})^{K}_{D}}\) 
we obtain a homotopy fibration 
\[\llnameddright{F^{m}(\underline{X})}{f^{m}(\underline{X})}{\bigvee_{i=1}^{m} X_{i}} 
      {\varphi}{P^{m}(\underline{X})}.\] 

For $1\leq j\leq m$, recall again that  
\[s_{j}\colon\namedright{X_{j}}{}{\bigvee_{i=1}^{m} X_{i}}\] 
is the inclusion of the $j^{th}$-wedge summand. Suppose that for $1\leq k\leq t$ 
there are maps 
\(a_{k}\colon\namedright{\Sigma Y_{k}}{}{X_{j_{k}}}\) 
where $1\leq j_{k}\leq m$. Let $b_{k}$ be the composite 
\[b_{k}\colon\nameddright{\Sigma Y_{k}}{a_{k}}{X_{j_{k}}}{s_{j_{k}}} 
      {\bigvee_{i=1}^{m} X_{i}}.\] 
Let 
\[w\colon\namedright{\Sigma Y_{1}\wedge\cdots\wedge Y_{t}}{} 
     {\bigvee_{i=1}^{m} X_{i}}\] 
be an iterated Whitehead product formed from the maps $b_{k}$. 
We say that $w$ has length $t$ and involves the maps 
$\{s_{j_{1}},\ldots,s_{j_{t}}\}$. 

\begin{theorem} 
   \label{thinwh} 
   Let $w$ be a Whitehead product on $\bigvee_{i=1}^{m} X_{i}$ formed 
   from the maps $b_{k}$. Suppose that~$w$ has length $t\geq m$ and involves 
   all the maps $s_{j}$ for $1\leq j\leq m$. Then $w$ lifts through 
   \(\llnamedright{F^{m}(\underline{X})}{f^{m}(\underline{X})}{\bigvee_{i=1}^{m} X_{i}}\).  
\end{theorem} 

\begin{proof} 
Recall that $K$ is $m$ disjoint points. For any proper subset $I\subset [m]$, 
by Lemma~\ref{polyinverse} the projection 
\(\namedright{X^{m}}{}{X^{I}}\) 
induces a map of polyhedral products 
\(\namedright{(\underline{X},\underline{\ast})^{K}}{}{(\underline{X},\underline{\ast})^{K_{I}}}\)   
which, in this case, is the equivalent to the pinch map 
\(p_{I}\colon\namedright{\bigvee_{i=1}^{m} X_{i}}{}{\bigvee_{i\in I} X_{i}}\). 
Observe that if $j\notin I$ then the composite 
\(\nameddright{\Sigma X_{j}}{s_{j}}{\bigvee_{i=1}^{m} X_{i}}{p_{I}}{\bigvee_{i\in I} X_{i}}\) 
is trivial, since $s_{j}$ is the inclusion of $X_{j}$ into the wedge. 
Since $I$ is a proper subset of $[m]$, we can always find a $j$ such that 
$p_{I}\circ s_{j}$ is null homotopic. As $w$ is a Whitehead product involving 
all the maps $s_{j}$ for $1\leq j\leq m$, the naturality of the Whitehead 
product implies that $p_{I}\circ w$ is null homotopic. This holds for any 
proper subset $I$ of $[m]$, so $p_{I}\circ w$ is null homotopic for 
all $I\subsetneq [m]$. 

Let 
\(\widetilde{w}\colon\namedright{Y_{1}\wedge\cdots\wedge Y_{t}}{} 
    {\Omega(\bigvee_{i=1}^{m} X_{i})}\) 
be the adjoint of $w$ and consider the composite 
\[\nameddright{Y_{1}\wedge\cdots\wedge Y_{m}}{\widetilde{w}} 
     {\Omega(\bigvee_{i=1}^{m} X_{i})}{\Omega\varphi}{\Omega P^{m}(\underline{X})}.\]  
By Theorem~\ref{djksplitting}, every factor in the homotopy decomposition of 
$\Omega P^{m}(\underline{X})=\Omega(\underline{X},\underline{\ast})^{K}_{D}$ 
is also a factor of $\Omega(\underline{X},\underline{\ast})^{K_{I}}$ for some 
$I\subsetneq [m]$, and the decomposition is compatible with the maps 
$\Omega p_{I}$. Therefore as $p_{I}\circ w$ is null homotopic for all 
$I\subsetneq [m]$, so is $\Omega p_{I}\circ\widetilde{w}$. Thus  
$\Omega\varphi\circ\widetilde{w}$ is null homotopic, implying that its adjoint 
$\varphi\circ w$ is null homotopic. Hence $w$ lifts through $f^{m}(\underline{X})$. 
\end{proof} 

A special case is when each of the spaces $X_{i}$ equals a common 
space $X$. Then we write $P^{m}(X)$ for the thin product, and have a 
homotopy fibration 
\(\llnameddright{F^{m}(X)}{f^{m}(X)}{\bigvee_{i=1}^{m} X}{}{P^{m}(X)}\). 
In this case Theorem~\ref{thinwh} refines. Consider the composite 
\(\llnameddright{F^{m}(X)}{f^{m}(X)}{\bigvee_{i=1}^{m} X}{\nabla}{X}\) 
which is used to define the weak cocategory of $X$.  

\begin{lemma} 
   \label{nablalift} 
   Let $w$ be a Whitehead product on $X$ of length $t\geq m$.  
   Then $w$ lifts through $\nabla$ to a Whitehead 
   product on $\bigvee_{i=1}^{m} X$ of length~$t$ involving all the 
   inclusions $s_{j}$ for $1\leq j\leq m$. 
\end{lemma} 

\begin{proof} 
Suppose that $w$ is a Whitehead product on $X$ of length~$t\geq m$,   
where $w$ is formed from maps 
\(a_{k}\colon\namedright{\Sigma Y_{k}}{}{X}\) 
for $1\leq k\leq t$. For $1\leq k\leq m$, let $b_{j}$ be the composite 
\(\nameddright{\Sigma Y_{j}}{a_{j}}{X}{s_{j}}{\bigvee_{i=1}^{m} X}\). 
If~$t>m$, then for $m<k\leq t$, let $b_{k}$ be the composite 
\(\nameddright{\Sigma Y_{k}}{a_{k}}{X}{s_{m}}{\bigvee_{i=1}^{m} X}\). 
Let $\overline{w}$ be the Whitehead product of the maps $b_{k}$ for $1\leq k\leq t$, 
where the bracketing order is the same as for $w$. Then the naturality 
of the Whitehead product implies that $w\simeq \nabla\circ\overline{w}$. 
Thus $\overline{w}$ lifts $w$ through $\nabla$, it has the same length 
as $w$, and involves all $m$ inclusions $s_{j}$.  
\end{proof} 
 
Combining the lifts in Lemma~\ref{nablalift} and Theorem~\ref{thinwh} 
we immediately obtain the following. 

\begin{theorem} 
   \label{thinwhX} 
   Any Whitehead product on $X$ of length $t\geq m$ lifts through 
   the composite 
   \(\lnameddright{F^{m}(X)}{f^{m}(X)}{\bigvee_{i=1}^{m} X}{\nabla}{X}\).~$\qqed$  
\end{theorem} 

\begin{remark} 
It is interesting to compare Theorems~\ref{thinwhX} with~\cite[Theorem 2]{Hov}. 
Both prove the same statement but in different language, we use polyhedral 
products in the proof of Theorem~\ref{thinwh} while Hovey uses ``flasque diagrams". 
However, the key point in both cases is the same: that the Whitehead products 
considered on $\bigvee_{i=1}^{m} X$ vanish when composed to any proper subwedge. 
\end{remark} 

Next, we aim towards Theorem~\ref{Fwh} which is a sort of converse 
to Theorem~\ref{thinwh}. While Theorem~\ref{thinwh} says that any 
Whitehead product of length~$t\geq m$ involving all $m$ maps $s_{j}$ 
lifts through $f^{m}(\underline{X})$, Theorem~\ref{Fwh} says that the 
homotopy class of $f^{m}(\underline{X})$ is completely determined by 
length~$t\geq m$ Whitehead products involving all $m$ maps $s_{j}$. 
To see this we modify the proof of Proposition~\ref{Tloopwedge} that identified the 
homotopy type of $F^{m}(\underline{X})$ in order to take into account 
the Whitehead product information in Theorems~\ref{WhPorter} and~\ref{WhGanea}. 

First, we require a general lemma. 

\begin{lemma} 
   \label{Samwedge} 
   Suppose that there are maps 
   \[f=\bigvee_{i=1}^{s} f_{i}\colon\namedright{\bigvee_{i=1}^{s}\Sigma A_{i}}{}{Z}\qquad  
      g=\bigvee_{j=1}^{t} g_{j}\colon\namedright{\bigvee_{j=1}^{t}\Sigma B_{j}}{}{Z}.\] 
   Then the Whitehead product $[f,g]$ is homotopic to the wedge sum
   of Whitehead products $\bigvee_{i=1}^{s}\bigvee_{j=1}^{t}[f_{i},g_{j}]$. 
\end{lemma} 

\begin{proof} 
Denoting the adjoint of a map $u$ by $\tilde{u}$, it is equivalent to 
show that the Samelson product of $\langle\tilde{f},\tilde{g}\rangle$ is 
homotopy equivalent to the wedge of Samelson products 
$\bigvee_{i=1}^{s}\bigvee_{j=1}^{t}\langle\tilde{f}_{i},\tilde{g}_{j}\rangle$. 
But this is clear from the pointwise definition of the Samelson product 
of two maps $u$ and $v$ as $\langle u,v\rangle(x,y)=u(x)v(y)u(x)^{-1}v(y)^{-1}$. 
\end{proof} 
   
\begin{theorem} 
   \label{Fwh} 
   Let $X_{1},\ldots,X_{m}$ be a simply-connected, pointed $CW$-complexes. 
   Then the homotopy fibration  
   \[\llnameddright{F^{m}(\underline{X})}{f^{m}(\underline{X})} 
          {\bigvee_{i=1}^{m} X_{i}}{}{P^{m}(\underline{X})}\] 
   satisfies the following: 
   \begin{letterlist} 
      \item $F^{m}(\underline{X})\simeq\bigvee_{\alpha\in\mathcal{I}} 
                   \Sigma(\Omega X_{1})^{(\alpha_{1})}\wedge\cdots\wedge 
                        (\Omega X_{m})^{(\alpha_{m})}$ 
               where $\alpha_{i}\geq 1$ for each $1\leq i\leq m$; 
      \item the restriction of $f^{m}(\underline{X})$ to 
               $\Sigma(\Omega X_{1})^{(\alpha_{1})}\wedge\cdots\wedge 
                      (\Omega X_{m})^{(\alpha_{m})}$ 
               is an iterated Whitehead product of length~$t\geq m$ formed from the maps   
               \(t_{j}\colon\nameddright{\Sigma\Omega X_{j}}{ev}{X_{j}}{s_{j}} 
                      {\bigvee_{i=1}^{m} X_{i}}\), 
                where each $t_{j}$ for $1\leq j\leq m$ appears at least once;   
      \item parts (a) and (b) are natural for maps of spaces 
                \(\namedright{X_{i}}{}{Y_{i}}\). 
   \end{letterlist} 
\end{theorem} 

\begin{proof} 
The construction in Section~\ref{sec:cxxcase} worked for any polyhedral product $\cxx^{K}$ 
where $K$ is a totally homology fillable simplicial complex. Specialize to the case 
of $\clxx^{K}$ when $K$ is $m$ disjoint points. Then the homotopy fibration 
\(\nameddright{\clxx^{K}}{}{(\underline{X},\underline{\ast})^{K}}{}{\prod_{i=1}^{m} X_{i}}\) 
is a model for the homotopy fibration 
\(\nameddright{\Gamma(\underline{X})}{\gamma(\underline{X})}{\bigvee_{i=1}^{m} X_{i}} 
      {}{\prod_{i=1}^{m} X_{i}}\). 
Therefore, by Theorem~\ref{WhPorter} there is a homotopy decomposition of  
$\clxx^{K}$ as a wedge sum of spaces 
$\Sigma\Omega X_{i_{1}}\wedge\cdots\wedge\Omega X_{i_{k}}$  
which can be chosen so that the restriction of the map 
\(\namedright{\clxx^{K}}{\gamma(\underline{X})} 
          {(\underline{X},\underline{\ast})^{K}\simeq\bigvee_{i=1}^{m} X_{i}}\) 
to a wedge summand $\Sigma\Omega X_{i_{1}}\wedge\cdots\wedge\Omega X_{i_{t}}$ 
is an iterated Whitehead product of the maps $t_{j}$. 

With 
\(\namedright{\clxx^{K}}{}{\bigvee_{i=1}^{m} X_{i}}\) 
as the starting point in this case, the first step in the construction in 
Section~\ref{sec:cxxcase} was to divide the wedge summands of $\clxx^{K}$ so that  
$\clxx^{K}\simeq\Sigma B_{1}\vee\Sigma C_{1}$, where $\Sigma B_{1}$ 
consists of those wedge summands having $\Omega X_{1}$ as a smash 
factor and $\Sigma C_{1}$ consists of those wedge summands not having 
$\Omega X_{1}$ as a smash factor. Pinching to $\Sigma C_{1}$, 
we obtain a homotopy fibration 
\(\nameddright{\Sigma B_{1}\rtimes\Omega\Sigma C_{1}} 
      {g_{1}}{\Sigma B_{1}\vee\Sigma C_{1}}{}{\Sigma C_{1}}\). 
By Theorem~\ref{Ganeafine}, there is a homotopy equivalence 
\[\Sigma B_{1}\rtimes\Omega\Sigma C_{1}\simeq\Sigma B_{1}\vee  
     \bigvee_{n=1}^{\infty}(\Sigma B_{1}\wedge (C_{1})^{(n)})\] 
which may be chosen so that the restriction of $g_{1}$ to 
$\Sigma B_{1}$ is $i_{L}$ and the restriction to $\Sigma B_{1}\wedge (C_{1})^{(n)}$ 
is an iterated Whitehead product of the maps $i_{L}$ and $i_{R}$ 
where $i_{L}$ appears once and $i_{R}$ appears $n$ times. Regarding  
$\Sigma B_{1}$ and $\Sigma C_{1}$ as a wedge of spaces of the form 
$\Sigma\Omega X_{i_{1}}\wedge\cdots\wedge\Omega X_{i_{k}}$, 
by Lemma~\ref{Samwedge} each iterated Whitehead product of the 
maps~$i_{L}$ and $i_{R}$ is homotopic to a wedge sum of iterated 
Whitehead products of the inclusion maps of the summands 
$\Sigma\Omega X_{i_{1}}\wedge\cdots\wedge\Omega X_{i_{k}}$ 
into $\clxx^{K}$. Therefore, as 
\(\namedright{\clxx^{K}}{\gamma(\underline{X})}{}{\bigvee_{i=1}^{m} X_{i}}\) 
is a wedge sum of iterated Whitehead products of the maps~$t_{j}$, 
the naturality of the Whitehead product implies that the composite 
\(\nameddright{\Sigma B_{1}\rtimes\Omega\Sigma C_{1}}{g_{1}} 
     {\Sigma B_{1}\vee\Sigma C_{1}\simeq\clxx^{K}}{\gamma(\underline{X})} 
     {\bigvee_{i=1}^{m} X_{i}}\) 
is homotopic to a wedge sum of iterated Whitehead products (of iterated 
Whitehead products) of the maps $t_{j}$. Further, as each wedge summand 
of $\Sigma B_{1}\rtimes\Omega\Sigma C_{1}$ has~$\Omega X_{1}$ as a 
smash factor, each of the Whitehead products has $t_{1}$ appearing at least once. 

The next step in the construction identifying the homotopy type 
of $F^{m}(\underline{X})$ was to divide 
$G_{1}\simeq\Sigma B_{1}\rtimes\Omega\Sigma C_{1}$ into a wedge 
$\Sigma B_{2}\vee\Sigma C_{2}$ where each wedge summand of $\Sigma B_{2}$ 
has both $\Omega X_{1}$ and $\Omega X_{2}$ as smash factors while 
the wedge summands of $\Sigma C_{2}$ have $\Omega X_{1}$ as a smash 
factor but not $\Omega C_{2}$. Then one pinches $G_{1}$ to $\Sigma C_{2}$ 
and calls the homotopy fibre $G_{2}$. This process is iterated until $G_{m}$ is 
identified as being homotopy equivalent to $F^{m}(\underline{X})$. At each 
step in the iteration we may argue as in the previous paragraph to identify 
the composite 
\(\namedright{G_{i}}{}{G_{i-1}}\longrightarrow\cdots\longrightarrow 
      \nameddright{G_{1}}{}{\clxx^{K}}{}{\bigvee_{i=1}^{m} X_{i}}\) 
as homotopic to a wedge sum of iterated Whitehead products of the 
maps $t_{j}$. As each of $\Omega X_{1},\ldots,\Omega X_{m}$ appears 
as a smash factor in each wedge summand of $G_{m}$, each of the corresponding 
Whitehead products has $t_{1},\ldots,t_{m}$ appearing at least once. 
This proves parts~(a) and~(b). 

The naturality statement in part~(c) follows from the naturality of 
Theorems~\ref{WhPorter} and~\ref{WhGanea}. 
\end{proof}  
 
When each $X_{i}$ equals a common space $X$ Theorem~\ref{Fwh} implies  
the following. 

\begin{corollary} 
   \label{FwhX} 
   Let $X$ be a simply-connected, pointed $CW$-complex. Then the homotopy 
   fibration  
   \[\llnameddright{F^{m}(X)}{f^{m}(X)}{\bigvee_{i=1}^{m} X}{}{P^{m}(X)}\] 
   satisfies the following: 
   \begin{letterlist} 
      \item $F^{m}(X)\simeq\bigvee_{\alpha\in\mathcal{I}}\Sigma(\Omega X)^{(t_{\alpha})}$; 
      \item the restriction of $f^{m}(X)$ to $\Sigma(\Omega X)^{(t_{\alpha})}$ is an 
               iterated Whitehead product of length \mbox{$t\geq m$} formed from the maps   
               \(t_{j}\colon\nameddright{\Sigma\Omega X}{ev}{X}{s_{j}}{\bigvee_{i=1}^{m} X}\) 
               where each $t_{j}$ for \mbox{$1\leq j\leq m$} appears at least once;   
      \item parts (a) and (b) are natural for a map of spaces 
                \(\namedright{X}{}{Y}\). 
   \end{letterlist} 
   $\qqed$ 
\end{corollary} 

We also need a refined version of Theorem~\ref{Fwh} in the case 
when each $X_{i}$ is a suspension. 

\begin{theorem} 
   \label{Fwhrefined} 
   Let $Y_{1},\ldots,Y_{m}$ be simply-connected, pointed $CW$-complexes. 
   Then the homotopy fibration  
   \[\llnameddright{F^{m}(\underline{\Sigma Y})}{f^{m}(\underline{\Sigma Y})} 
          {\bigvee_{i=1}^{m}\Sigma Y_{i}}{}{P^{m}(\underline{\Sigma Y})}\] 
   satisfies the following: 
   \begin{letterlist} 
      \item $F^{m}(\underline{\Sigma Y})\simeq\bigvee_{\beta\in\mathcal{J}} 
                   \Sigma (Y_{1})^{(\beta_{1})}\wedge\cdots\wedge 
                        (Y_{m})^{(\beta_{m})}$ 
               where $\beta_{i}\geq 1$ for each $1\leq i\leq m$; 
      \item the restriction of $f^{m}(\underline{\Sigma Y})$ to 
               $\Sigma (Y_{1})^{(\beta_{1})}\wedge\cdots\wedge(Y_{m})^{(\beta_{m})}$  
               is an iterated Whitehead product of length~$t\geq m$ formed from the maps   
               \(\namedright{\Sigma Y_{j}}{s_{j}}{\bigvee_{i=1}^{m}\Sigma Y_{i}}\), 
                where each $s_{j}$ for $1\leq j\leq m$ appears at least once;   
      \item parts (a) and (b) are natural for maps of spaces 
                \(\namedright{Y_{i}}{}{Z_{i}}\). 
   \end{letterlist} 
\end{theorem} 

\begin{proof} 
Argue exactly as for Theorem~\ref{Fwh} using Theorem~\ref{Porterfine} 
as the starting point rather than Theorem~\ref{WhPorter}, and organizing 
the $B_{i}$, $C_{i}$ wedge summands so that $B_{i}$ contains all the smash products  
having $Y_{1},\ldots, Y_{i}$ as factors (rather than 
$\Omega\Sigma Y_{1},\ldots,\Omega\Sigma Y_{i}$) and $C_{i}$ contains all 
the smash products having $Y_{1},\ldots, Y_{i-1}$ as factors but not $Y_{i}$.  
\end{proof}

\Large 
\part{Cocategory and Nilpotence.} 
\normalsize

\section{Cocategory and Nilpotence I} 
\label{sec:cocat} 

The purpose of this section is to prove main result of the paper, Theorem~\ref{main}: if $X$ is 
simply-connected then $\wcocat(X)=m$ if and only if $\nil(\Omega X)=m$. 

\begin{proof}[Proof of Theorem~\ref{main}]  
We will use the homotopy fibration 
\begin{equation} 
   \label{BGfib} 
   \lllnameddright{F^{m+1}(X)}{f^{m+1}(X)}{\bigvee_{i=1}^{m+1} X}{}{P^{m+1}(X)}  
\end{equation} 
from Corollary~\ref{FwhX}.  

Suppose that $\wcocat(X)=m$. By the definition of weak cocategory, 
this implies that the composition 
\(\lllnameddright{F^{m+1}(X)}{f^{m+1}(X)}{\bigvee_{i=1}^{m+1} X}{\nabla}{X}\) 
is null homotopic. Consider the Samelson product 
\(\namedright{(\Omega X)^{(m+1)}}{c_{m}}{\Omega X}\), 
where $c_{m}$ is the $(m+1)$-fold Samelson product of the identity map 
on $\Omega X$ with itself. Let 
\(w_{m}\colon\namedright{\Sigma (\Omega X)^{(m+1)}}{}{X}\) 
be the adjoint of $c_{m}$. Then $w_{m}$ is the $(m+1)$-fold Whitehead product of 
the evaluation map 
\(\namedright{\Sigma\Omega X}{\overline{ev}}{X}\) 
with itself. By Theorem~\ref{thinwhX}, $w_{m}$ factors through 
$\nabla\circ f^{m+1}(X)$. But by hypothesis, $\nabla\circ f^{m+1}(X)$ is 
null homotopic. Therefore $w_{m}$ is null homotopic and hence its 
adjoint~$c_{m}$ is null homotopic. This implies that $\nil(\Omega X)\leq m$; 
that is, $\nil(\Omega X)\leq\wcocat(X)$.  

Conversely, suppose that $\nil(\Omega X)=m$. By Corollary~\ref{FwhX}, 
$F^{m+1}(X)\simeq\bigvee_{\alpha\in\mathcal{I}}\Sigma(\Omega X)^{(t_{\alpha})}$  
and the restriction of $f^{m+1}(X)$ to each wedge summand is an iterated 
Whitehead product of \mbox{length~$\geq m+1$}. Therefore 
$\nabla\circ f^{m+1}(X)$ maps each wedge summand $\Sigma(\Omega X)^{(t_{\alpha})}$ 
to $X$ by an iterated Whitehead product of \mbox{length~$\geq m+1$}. The adjoint 
\(\namedright{(\Omega X)^{(t_{\alpha})}}{}{\Omega X}\) 
is a Samelson product of length~$\geq m+1$, and so factors through $c_{m}$. 
By hypothesis, $c_{m}$ is null homotopic. Therefore, so is 
\(\namedright{(\Omega X)^{(t_{\alpha})}}{}{\Omega X}\) 
and its adjoint 
\(\namedright{\Sigma(\Omega X)^{(t_{\alpha})}}{}{\Omega X}\). 
This is true for every $\alpha\in\mathcal{I}$, so $\nabla\circ f^{m+1}(X)$ is 
null homotopic, implying that $\wcocat(X)\leq m$; that is, $\wcocat(X)\leq\nil(\Omega X)$.  

Combining the previous two paragraphs we obtain $\wcocat(X)=m$ if and 
only if $\nil(\Omega X)=m$. 
\end{proof}

\section{Retractile $H$-spaces} 
\label{sec:retractile} 

We now aim towards explicit calculations of the homotopy nilpotency 
class of quasi-$p$-regular exceptional Lie groups and nonmodular 
$p$-compact groups in Section~\ref{sec:examples}. This requires introducing 
a special family of finite $H$-spaces. Throughout this section we will assume 
that all spaces and maps have been localized (or completed) at a prime $p$, 
and homology is taken with mod-$p$ coefficients. For a $\mathbb{Z}/p\mathbb{Z}$-vector 
space $V$, let $\Lambda(V)$ be the exterior algebra generated by $V$. 

\begin{definition}  
Let $B$ be an $H$-space. Suppose that there is a space $A$ and a map 
\(i\colon\namedright{A}{}{B}\) 
satisfying the following: 
\begin{itemize} 
   \item[(i)] $\hlgy{B}\cong\Lambda(\rhlgy{A})$; 
   \item[(ii)] $i_{\ast}$ induces the inclusion of the generating set; 
   \item[(iii)] $\Sigma i$ has a left homotopy inverse. 
\end{itemize} 
Then the triple $(A,i,B)$ is a \emph{retractile} $H$-space. 
\end{definition} 

If $(A,i,B)$ is a retractile $H$-space then many properties of $B$ are determined 
by the restriction to~$A$~\cite{GT1,Th1,Th2}. We will show this is also the case for 
$\nil(G)$ when $G$ is a finite loop space and $(A,i,G)$ is retractile. A large family of 
retractile $H$-spaces was constructed in different ways by Cooke, Harper and 
Zabrodsky~\cite{CHZ} and Cohen and Neisendorfer~\cite{CN}. They show 
that if $p$ is a prime and $A$ is a finite $CW$-complex consisting of $\ell$ 
odd dimensional cells, where $\ell<p-1$, then there exists a $p$-local 
$H$-space $B$ and a map 
\(i\colon\namedright{A}{}{B}\) 
which makes $(A,i,B)$ retractile. Further, in the boundary case when 
$\ell=p-1$, they show that if there happens to be a finite $H$-space $B$ 
and a map 
\(\namedright{A}{i}{B}\) 
inducing an isomorphism $\hlgy{B}\cong\Lambda(\rhlgy{A})$, 
then $\Sigma i$ has a left homotopy inverse so $(A,i,B)$ is retractile. 
Notice that products of retractile $H$-spaces are retractile. For 
if each of $\{(A_{j},i_{j},B_{j})\}_{j=1}^{m}$ is retractile then so is 
$(\bigvee_{j=1}^{m} A_{j},\bigvee_{i=1}^{m} i_{j},\prod_{j=1}^{m} B_{j})$. 

Specific examples of retractile $H$-spaces are given by certain 
simply-connected simple compact Lie groups and $p$-compact groups. 
Fix a prime $p$. By~\cite{MNT}, when localized at $p$ every simply-connected, 
simple compact Lie group $G$ without $p$-torsion in cohomology decomposes 
as a product of $k$ spaces for some $1\leq k\leq p-1$ spaces. 
If the rank is low with respect to $p$ then by~\cite{Th2} the factors 
are all retractile, and so their product - the original group $G$ - is also 
retractile. The precise ranks involved are as follows: 
\begin{equation} 
  \label{Glist} 
  \begin{tabular}{ll}
       $SU(n)$ & $\mbox{if}\ n\leq (p-1)^{2}+1$ \\
       $Sp(n)$ & $\mbox{if}\ 2n\leq (p-1)^{2}$ \\
       $Spin(2n+1)$ & $\mbox{if}\ 2n\leq (p-1)^{2}$ \\
       $Spin(2n)$ & $\mbox{if}\ 2(n-1)\leq (p-1)^{2}$ \\
       $G_{2}, F_{4}, E_{6}$ & $\mbox{if}\ p\geq 5$ \\
       $E_{7}, E_{8}$ & $\mbox{if}\ p\geq 7$. 
   \end{tabular} 
\end{equation} 
Recall that $X$ is a $p$-compact group if $X\simeq\Omega BX$ where $BX$ 
is a $p$-complete space and the mod-$p$ homology of $X$ is finite. When $p$ 
is odd there is a classification of simply-connected simple $p$-compact 
groups without $p$-torsion in cohomology. We give the description 
in~\cite[Introduction]{Da}, which distills the classification from deep results 
in~\cite{AGMV}. There are four familes: (i) simply-connected, 
simple compact Lie groups; (ii) an infinite family of spaces that decompose 
as a product of spheres, sphere bundles over spheres, and factors of $SU(n)$; 
(iii) $30$ special nonmodular cases; and (iv) $4$ exotic modular cases. 
The retractile cases in (i) are covered by~(\ref{Glist}). The retractile cases 
in (ii) depend on $SU(n)$ being retractile, and from~\cite{Da} it follows that 
every case in (iii) and~(iv) is retractile. So to~(\ref{Glist}) we add a second list: 
\begin{equation} 
  \label{Glist2} 
  \begin{tabular}{ll}  
         (a) &  an infinite family of $p$-compact groups that decompose as a product 
              of spheres,   \\ 
             &  sphere bundles over spheres, and factors of $SU(n)$ for $n\leq(p-1)^{2}+1$; \\  
        (b) & the $30$ special nonmodular $p$-compact groups (see~\cite[Table 3.2]{Da}); \\
        (c) &  the $4$ exotic modular $p$-compact groups. 
   \end{tabular} 
\end{equation}   
 
Since a $p$-compact group $G$ is a loop space, it has a classifiying space $BG$ and 
there is a homotopy equivalence 
\(e\colon\namedright{G}{}{\Omega BG}\). 
Let $\overline{ev}$ be the composite 
\[\overline{ev}\colon\nameddright{\Sigma G}{\Sigma e}{\Sigma\Omega BG}{ev}{BG}\] 
where $ev$ is the canonical evaluation map. Let $j$ be the composite 
\[j\colon\nameddright{\Sigma A}{\Sigma i}{\Sigma G}{\overline{ev}}{BG}.\] 
A key property is the following. 

\begin{lemma} 
   \label{AtoBG} 
   Let $G$ be a $p$-compact group from list~(\ref{Glist}) or~(\ref{Glist2}).  
   Then there is a homotopy commutative diagram 
   \[\diagram 
          \Sigma G\rto^-{\overline{ev}}\dto^{r} & BG\ddouble \\ 
          \Sigma A\rto^-{j} & BG 
     \enddiagram\] 
   where $r$ is a left homotopy inverse of $\Sigma i$. 
\end{lemma} 

\begin{proof} 
By hypothesis, the map 
\(\namedright{\Sigma A}{\Sigma i}{\Sigma G}\) 
has a right homotopy inverse, so $\Sigma G\simeq\Sigma A\vee C$ for 
some space~$C$. In~\cite{GT1} the retractile properties were used to precisely 
describe the complementary factor~$C$ in the case when $G$ is from the 
list~(\ref{Glist}), but the proof adapts immediately to list~(\ref{Glist2}) as well. 
It is well known that there is a homotopy fibration 
\(\nameddright{\Sigma G\wedge G}{\mu^{\ast}}{\Sigma G}{\overline{ev}}{BG}\) 
where $\mu^{\ast}$ is the canonical Hopf construction. In~\cite{GT1} 
it is shown there is a homotopy equivalence 
\begin{equation} 
  \label{SigmaGdecomp} 
  \llnamedright{\Sigma A\vee C}{\Sigma i+s}{\Sigma G} 
\end{equation}  
where $C$ is a retract of $\Sigma G\wedge G$ and $s$ factors 
through the Hopf construction $\mu^{\ast}$. Consequently, $s$ 
composes trivially with $\overline{ev}$, so from the decomposition 
in~(\ref{SigmaGdecomp}) we obtain a homotopy commutative diagram 
\[\diagram 
     \Sigma G\rto^-{\overline{ev}}\dto^{r} & BG\ddouble \\ 
     \Sigma A\rto^-{j} & BG    
  \enddiagram\] 
where $r$ is a left homotopy inverse of $\Sigma i$. 
\end{proof} 

For a space $Y$, let 
\(E\colon\namedright{Y}{}{\Omega\Sigma Y}\) 
be the suspension map, which is defined as the (right) adjoint of the 
identity map on $\Sigma Y$. Notice that the evaluation map 
\(\namedright{\Sigma\Omega Y}{ev}{Y}\) 
is the (left) adjoint of the identity map on $\Omega Y$. So the composite 
\(\nameddright{\Omega Y}{E}{\Omega\Sigma\Omega Y}{\Omega ev}{\Omega Y}\) 
is homotopic to the identity map. This leads to the following corollary 
of Lemma~\ref{AtoBG}. 

\begin{corollary} 
   \label{GBG} 
   The composite 
   \(\namedddright{G}{E}{\Omega\Sigma G}{\Omega r}{\Omega\Sigma A} 
       {\Omega j}{\Omega BG}\) 
   is homotopic to the homotopy equivalence $e$.  
\end{corollary} 

\begin{proof} 
By Lemma~\ref{AtoBG}, $j\circ r\simeq\overline{ev}=ev\circ\Sigma e$. Thus 
$\Omega j\circ\Omega r\circ E\simeq\Omega ev\circ\Omega\Sigma e\circ E$. 
The naturality of~$E$ implies that the latter composite is homotopic to 
$\Omega ev\circ E\circ e$. But $\Omega ev\circ E$ is homotopic to the 
identity map on $\Omega BG$, so we obtain 
$\Omega j\circ\Omega r\circ E\simeq e$, as asserted.   
\end{proof}

\section{Cocategory and Nilpotence II} 
\label{sec:cocat2} 

In this section a restricted version of Theorem~\ref{main} is proved. 

\begin{definition} 
Let $G$ be an $H$-group. Suppose that $Y$ is a pointed, path-connected 
space and there is a map 
\(f\colon\namedright{Y}{}{G}\). 
We say that $G$ has \emph{homotopy nilpotency class $m$ with respect to $f$} 
if the composite 
\(\llnameddright{Y^{(m+1)}}{f^{(m+1)}}{G^{(m+1)}}{c_{m}}{G}\) 
is null homotopic but $c_{m-1}\circ f^{(m)}$ is nontrivial. In this case 
we write $\nil(Y,f,G)=m$. 
\end{definition} 

\begin{definition} 
\label{weakcocatmap} 
Let $X$ and $Y$ be pointed, path-connected spaces, and suppose that there 
is a map 
\(g\colon\namedright{Y}{}{X}\). 
We say that $X$ has \emph{weak cocategory $m$ with respect to $g$} if 
the composite 
\(\lllnamedddright{F_{m+1}(Y)}{f_{m+1}(Y)}{\bigvee_{i=1}^{m+1} Y}{\nabla}{Y}{g}{X}\) 
is null homotopic but $g\circ\nabla\circ f_{m}(Y)$ is nontrivial. In this case 
we write $\wcocat(Y,g,X)=m$. 
\end{definition} 

\begin{remark} 
Definition~\ref{weakcocatmap} is dual to the weak category of a 
map (see~\cite[Exercise 2.8]{CLOT}). 
\end{remark} 

\begin{theorem} 
   \label{Awcocatnil} 
   Let $G$ be a $p$-compact group from list~(\ref{Glist}) or~(\ref{Glist2}).  
   Then $\wcocat(\Sigma A,j,BG)=m$ if and only if $\nil(A,i,G)=m$. 
\end{theorem} 

\begin{proof} 
The naturality of the thin product and the naturality of the decomposition 
of $F^{m+1}(X)$ in Corollary~\ref{FwhX} implies that there 
is a homotopy fibration diagram 
\[\spreaddiagramcolumns{-0.7pc}\diagram  
       F^{m+1}(\Sigma A)\rrto^-{f^{m+1}(\Sigma A)}\dto^{g^{m+1}(j)} 
           & & \bigvee_{i=1}^{m+1}\Sigma A\rrto\dto^{\bigvee_{i=1}^{m+1} j} 
           & & P^{m+1}(\Sigma A)\dto^{P^{m+1}(j)} \\ 
       F_{m+1}(BG)\rrto^-{f^{m+1}(BG)} & & \bigvee_{i=1}^{m+1} BG\rrto & & P^{m+1}(BG)  
  \enddiagram\] 
where 
$F^{m+1}(\Sigma A)\simeq 
    \bigvee_{\alpha\in\mathcal{I}}\Sigma(\Omega\Sigma A)^{(t_{\alpha})}$,  
$F^{m+1}(BG)\simeq\bigvee_{\alpha\in\mathcal{I}}\Sigma(\Omega BG)^{(t_{\alpha})}$, and 
$g^{m+1}(j)\simeq\bigvee_{\alpha\in\mathcal{I}}\Sigma(\Omega j)^{(t_{\alpha})}$.  

Suppose that $\wcocat(\Sigma A,j,BG)=m$. Then the composite 
\(\lllnamedright{F^{m+1}(\Sigma A)}{f^{m+1}(\Sigma A)}{\bigvee_{i=1}^{m+1}\Sigma A} 
     \stackrel{\nabla}{\longrightarrow}\namedright{\Sigma A}{j}{BG}\) 
is null homotopic. Consider the Samelson product 
\(\llnameddright{A^{(m+1)}}{i^{(m+1)}}{G^{(m+1)}}{c_{m}}{G}\). 
Its adjoint is the composite 
\(\llnameddright{\Sigma A^{(m+1)}}{\Sigma i^{(m+1)}}{\Sigma G^{(m+1)}}{w_{m}}{BG}\), 
where $w_{m}$ is the $(m+1)$-fold Whitehead product of $\overline{ev}$ with 
itself, which is adjoint to $c_{m}$. 

We claim that there is a lift 
\begin{equation} 
   \label{wmlift} 
   \diagram 
       & & F^{m+1}(\Sigma A)\dto^{f^{m+1}(\Sigma A)} \\ 
       & & \bigvee_{i=1}^{m+1}\Sigma A\dto^{j\circ\nabla} \\  
       \Sigma A^{(m+1)}\rto^-{\Sigma i^{(m+1)}}\uurrdashed|>\tip 
            & \Sigma G^{(m+1)}\rto^-{w_{m}} & BG. 
  \enddiagram 
\end{equation} 
The lift is constructed in two stages. First, we give a specific lift of 
$w_{m}\circ\Sigma i^{(m+1)}$ through $j\circ\nabla$. Let  
\(s'_{j}\colon\namedright{\Sigma A}{}{\bigvee_{i=1}^{m+1}{\Sigma A}}\) 
and 
\(s_{j}\colon\namedright{BG}{}{\bigvee_{i=1}^{m+1} BG}\) 
be the inclusions of the $j^{th}$-wedge summands, and let $t'_{j}$ 
and $t_{j}$ be the composites 
\(t'_{j}\colon\nameddright{\Sigma\Omega\Sigma A}{ev}{\Sigma A} 
     {s'_{j}}{\bigvee_{i=1}^{m+1}\Sigma A}\) 
and 
\(t_{j}\colon\nameddright{\Sigma G}{\overline{ev}}{BG}{s_{j}}{\bigvee_{i=1}^{m+1} BG}\). 
Let $\overline{w}^{'}_{m}=[t'_{1},[t'_{2},\ldots,[t'_{m},t'_{m+1}]\ldots]$ 
and $\overline{w}_{m}=[t_{1},[t_{2},\ldots,[t_{m},t_{m+1}]\ldots]$. 
Since $G\simeq\Omega BG$, for convenience we write 
\(\namedright{\Omega\Sigma A}{\Omega j}{\Omega BG}\) 
as 
\(\namedright{\Omega\Sigma A}{\Omega j}{G}\). 
Consider the diagram 
\begin{equation} 
  \label{liftmed} 
  \diagram 
     \Sigma A^{(m+1)}\rto^-{\Sigma E^{(m+1)}}\drto_-{\Sigma i^{(m+1)}} 
        & \Sigma\Omega\Sigma A^{(m+1)} 
                 \rto^-{\overline{w}^{'}_{m}}\dto^{\Sigma(\Omega j)^{(m+1)}} 
        & \bigvee_{i=1}^{m+1}\Sigma A\rto^-{\nabla}\dto^{\bigvee_{i=1}^{m+1} j} 
        & \Sigma A\dto^{j} \\ 
     & \Sigma G^{(m+1)}\rto^-{\overline{w}_{m}} & \bigvee_{i=1}^{m+1} BG\rto^-{\nabla} & BG. 
  \enddiagram 
\end{equation} 
The left triangle homotopy commutes since $i\simeq j\circ E$. The middle 
square homotopy commutes by the naturality of the Whitehead product. 
The right square homotopy commutes by the naturality of the fold map. 
Thus the entire diagram homotopy commutes. Observe that since 
\(\nameddright{\Sigma A}{\Sigma E}{\Sigma\Omega\Sigma A}{ev}{\Sigma A}\) 
is homotopic to the identity map, we have $t'_{j}\circ E\simeq s'_{j}$. Thus the composite 
$\overline{w}^{'}_{m}\circ\Sigma E^{(m+1)}$ along the top row in~(\ref{liftmed})  
homotopic to $[s'_{1},[s'_{2},\ldots,[s'_{m},s'_{m+1}]\ldots ]$. On the other 
hand, since $w_{m}\simeq\nabla\circ\overline{w}_{m}$, the bottom direction 
around~(\ref{liftmed}) is homotopic to $w_{m}\circ\Sigma i^{(m+1)}$. Thus the 
homotopy commutativity of~(\ref{liftmed}) implies that  
$[s'_{1},[s'_{2},\ldots,[s'_{m},s'_{m+1}]\ldots ]$ lifts $w_{m}\circ\Sigma i^{(m+1)}$ 
through $j\circ\nabla$. Further, as $\overline{w}^{'}_{m}$ is an $(m+1)$-fold 
Whitehead product involving all $m+1$ spaces in the wedge 
$\bigvee_{i=1}^{m+1}\Sigma A$, Theorem~\ref{thinwh} implies 
that $\overline{w}^{'}_{m}$ lifts through $f^{m+1}(\Sigma A)$ 
to~$F^{m+1}(\Sigma A)$. 

By hypothesis, $\wcocat(\Sigma A,j,BG)=m$, so the composite 
\(\lllnamedright{F^{m+1}(\Sigma A)}{f^{m+1}(\Sigma A)}{\bigvee_{i=1}^{m+1}\Sigma A} 
      \stackrel{\nabla}{\longrightarrow}\namedright{\Sigma A}{j}{BG}\) 
is null homotopic. Therefore the lift of $w_{m}\circ\Sigma i^{(m+1)}$ through 
$\nabla\circ j\circ f^{m+1}(\Sigma A)$ in~(\ref{wmlift}) implies that 
$w_{m}\circ\Sigma i^{(m+1)}$ is null homotopic. Taking adjoints, 
$w_{m}\circ i^{(m+1)}$ is null homotopic, and so $\nil(A,i,G)\leq m$.  
That is, $\nil(A,i,G)\leq\wcocat(\Sigma A,j,BG)$. 

Conversely, suppose that $\nil(A,i,G)\leq m$. Then the composite 
\(\lnameddright{A^{(m+1)}}{i^{(m+1)}}{G^{(m+1)}}{c_{m}}{G}\) 
is null homotopic. Taking adjoints, the composite 
\(\llnameddright{\Sigma A^{(m+1)}}{\Sigma i^{(m+1)}}{\Sigma G^{(m+1)}}{w_{m}}{BG}\) 
is null homotopic. Consider the composite 
\(\lllnamedright{F^{m+1}(\Sigma A)}{f^{m+1}(\Sigma A)}{\bigvee_{i=1}^{m+1}\Sigma A} 
     \stackrel{\nabla}{\longrightarrow}\namedright{\Sigma A}{j}{BG}\). 
By Corollary~\ref{FwhX} we have 
$F^{m+1}(\Sigma A)\simeq\bigvee_{\gamma\in\mathcal{J}_{m+1}} 
      \Sigma A^{(t_{\gamma})}$ 
and $f^{m+1}(\Sigma A)$ restricted to $\Sigma A^{(t_{\gamma})}$ 
is an iterated Whitehead product $w_{t_{\gamma}}$ of length~$\geq m+1$ 
formed from the inclusions  
\(\Sigma A\hookrightarrow\bigvee_{i=1}^{m+1}\Sigma A\) 
of the wedge summands. Thus each $w_{t_{\gamma}}$ factors through 
$w_{m}\circ\Sigma i^{(m+1)}$, which by hypothesis is null homotopic.  
Therefore $w_{t_{\gamma}}$ is null homotopic. As this is true for every wedge 
summand of $F^{m+1}(\Sigma A)$, the map $f^{m+1}(\Sigma A)$ is null 
homotopic. Thus $j\circ\nabla\circ f^{m+1}(\Sigma A)$ is null homotopic, 
implying that $\wcocat(\Sigma A,j,BG)\leq m$. That is, 
$\wcocat(\Sigma A,j,BG)\leq\nil(A,i,G)$. 

Hence $\wcocat(\Sigma A,j,BG)=m$ if and only if $\nil(A,i,G)=m$. 
\end{proof} 

The next theorem links the equivalences in Theorems~\ref{main} 
and~\ref{Awcocatnil}. 

 \begin{theorem} 
   \label{wcocatreduction} 
   Let $G$ be a $p$-compact group from list~(\ref{Glist}) or~(\ref{Glist2}). Then 
   $\wcocat(BG)=m$ if and only if $\wcocat(\Sigma A,j,BG)=m$. 
\end{theorem} 

\begin{proof} 
In general, by Corollary~\ref{FwhX}, for a space $X$ there is a homotopy fibration 
\begin{equation} 
  \label{Ffib} 
  \llnameddright{F^{m+1}(X)}{f^{m+1}(X)}{\bigvee_{i=1}^{m+1} X}{}{P^{m}(X)} 
\end{equation}  
where the right map is the inclusion of the wedge into the thin product; 
$F^{m+1}(X)\simeq\bigvee_{\alpha\in\mathcal{I}}\Sigma(\Omega X)^{(t_{\alpha})}$; 
and $f^{m+1}(X)$ is a wedge sum of iterated Whitehead products of maps of the form  
\(\nameddright{\Sigma\Omega X}{ev}{X}{s_{j}}{\bigvee_{i=1}^{m+1} X}\). 
Moreover, all of this is natural for maps of spaces 
\(\namedright{X}{}{Y}\). 
Recall that $(A,i,G)$ being retractile means that there is a map 
\(\namedright{\Sigma G}{r}{\Sigma A}\) 
which is a left homotopy inverse of $\Sigma i$. Starting from the composite 
\(\nameddright{\Sigma G}{r}{\Sigma A}{j}{BG}\), 
consider the diagram 
\begin{equation} 
  \label{bigdgrm} 
  \diagram 
       \bigvee_{\alpha\in\mathcal{I}}\Sigma G^{(t_{\alpha})} 
              \dto^-{\bigvee_{\alpha\in\mathcal{I}}\Sigma E^{(t_{\alpha})}}\xto[rrrr]^{\simeq} 
        & & & & \bigvee_{\alpha\in\mathcal{I}}\Sigma(\Omega BG)^{(t_{\alpha})}\ddouble\\  
        \bigvee_{\alpha\in\mathcal{I}}\Sigma(\Omega\Sigma G)^{(t_{\alpha})} 
              \rrto^-{\bigvee_{\alpha\in\mathcal{I}}\Sigma(\Omega r)^{(t_{\alpha})}}\dto^{\simeq}  
        & & \bigvee_{\alpha\in\mathcal{I}}\Sigma(\Omega\Sigma A)^{(t_{\alpha})} 
              \rrto^-{\bigvee_{\alpha\in\mathcal{I}}\Sigma(\Omega j)^{(t_{\alpha})}}\dto^{\simeq}     
        & & \bigvee_{\alpha\in\mathcal{I}}\Sigma(\Omega BG)^{(t_{\alpha})}\dto^{\simeq} \\ 
        F^{m+1}(\Sigma G)\rrto\dto^-{f^{m+1}(\Sigma G)} 
            & & F^{m+1}(\Sigma A)\rrto\dto^{f^{m+1}(\Sigma A)} 
            & & F^{m+1}(BG)\dto^{f^{m+1}(BG)} \\ 
        \bigvee_{i=1}^{m+1}\Sigma G\rrto^-{\bigvee_{i=1}^{m+1} r}\dto^{\nabla} 
        & & \bigvee_{i=1}^{m+1}\Sigma A\rrto^-{\bigvee_{i=1}^{m+1} j}\dto^{\nabla} 
        & & \bigvee_{i=1}^{m+1} BG\dto^{\nabla} \\ 
        \Sigma G\rrto^-{r} & & \Sigma A\rrto^-{j} & & BG.   
  \enddiagram 
\end{equation} 
The bottom squares homotopy commute by the naturality of the fold map, 
the middle and top squares homotopy commute by the naturality of~(\ref{Ffib}) 
and the naturality of the decomposition of $F^{m+1}(X)$. By Corollary~\ref{GBG}, 
$\Omega j\circ\Omega r\circ E$ is a homotopy equivalence, so the top rectangle 
also homotopy commutes. 

Suppose that $\wcocat(BG)=m$. Then the composite 
\(\lllnameddright{F^{m+1}(BG)}{f^{m+1}(BG)}{\bigvee_{i=1}^{m+1} BG}{\nabla}{BG}\) 
is null homotopic. The homotopy commutativity of~(\ref{bigdgrm}) therefore 
implies that the composite 
\(\lllnamedright{F^{m+1}(\Sigma A)}{f^{m+1}(\Sigma A)}{\bigvee_{i=1}^{m+1}\Sigma A} 
      \stackrel{\nabla}{\longrightarrow}\namedright{\Sigma A}{j}{BG}\) 
is null homotopic. Thus $\wcocat(\Sigma A,j,BG)\leq m$. 

Suppose that $\wcocat(\Sigma A,j,BG)=m$. Then the composite 
\(\lllnamedright{F^{m+1}(\Sigma A)}{f^{m+1}(\Sigma A)}{\bigvee_{i=1}^{m+1}\Sigma A} 
      \stackrel{\nabla}{\longrightarrow}\namedright{\Sigma A}{j}{BG}\) 
is null homotopic. The homotopy commutativity of~(\ref{bigdgrm}) 
therefore implies that the upper direction around the diagram is null 
homotopic. But as the top row is a homotopy equivalence, this implies 
that the composite 
\(\lllnameddright{F^{m+1}(BG)}{f^{m+1}(BG)}{\bigvee_{i=1}^{m+1} BG}{\nabla}{BG}\) 
is null homotopic. That is, $\wcocat(BG)\leq m$. 

Hence $\wcocat(BG)=m$ if and only if $\wcocat(\Sigma A,j,BG)=m$. 
\end{proof} 

Combining the equivalences in Theorems~\ref{main}, \ref{Awcocatnil} 
and~\ref{wcocatreduction} we obtain the following. 

\begin{theorem} 
   \label{nilreduction} 
   Let $G$ be a $p$-compact group from~(\ref{Glist}) or~(\ref{Glist2}). Then 
   $\nil(G)=m$ if and only if $\nil(A,i,G)=m$.~$\qqed$ 
\end{theorem} 

This reduction from checking whether 
\(\namedright{G^{(m+1)}}{c_{m}}{G}\) 
is null homotopic to checking whether 
\(\lnameddright{A^{(m+1)}}{i^{(m+1)}}{G^{(m+1)}}{c_{m}}{G}\) 
is null homotopic gives a very practical tool for determining the 
homotopy nilpotency class of certain finite loop spaces. This 
will be used in the next section to give explicit calculations.

\section{Examples} 
\label{sec:examples} 

Fix an odd prime $p$. Let $G$ be a $p$-localized simply-connected simple 
compact Lie group, or in the $p$-complete setting, let $G^{\wedge}_{p}$ 
be a $p$-compact group. In the $p$-local case, $G$ is rationally homotopy 
equivalent to a product of odd dimensional spheres $\prod_{i=1}^{\ell} S^{2n_{i}-1}$, 
and in the $p$-complete case, $G^{\wedge}_{p}$ is rationally homotopy equivalent to 
the rationalization of a product of odd dimensional $p$-completed spheres 
$\prod_{i=1}^{\ell} (S^{2n_{i}-1})^{\wedge}_{p}$. Reordering the indices 
if need be, we may assume that $n_{1}\leq\cdots\leq n_{\ell}$. The \emph{type} 
of $G$ is $\{n_{1},\ldots,n_{\ell}\}$. From now on we drop the usual 
$(\ \ )^{\wedge}_{p}$ notation for $p$-completion for convenience, the 
context making it clear when it should be used. In the $p$-local ($p$-complete) 
case we say that $G$ is \emph{$p$-regular} if there is a $p$-local (or $p$-complete) 
homotopy equivalence $G\simeq\prod_{i=1}^{\ell} S^{2n_{i}-1}$. It is 
\emph{quasi-$p$-regular} if there is a $p$-local (or $p$-complete) homotopy 
equivalence $G\simeq\prod_{i=1}^{\ell} B_{i}$, where each $B_{i}$ is either a 
sphere  or a sphere bundle over a sphere. 

Kaji and Kishimoto~\cite{KK} have determined the homotopy nilpotency 
classes of all $p$-regular $p$-compact groups. Kishimoto~\cite{K} determined 
the homotopy nilpotency class of all quasi-$p$-regular cases of $SU(n)$. 
In both cases the calculations were intense and involved detailed information 
about nontrivial Samelson products on the group in question. In contrast, 
very little is known about nontrivial Samelson products in exceptional Lie 
groups, so one would ordinarily not hope to be able to determine their 
homotopy nilpotency classes. Nevertheless, in Theorem~\ref{qpreg} we will 
calculate the homotopy nilpotency class in most quasi-$p$-regular cases. 

The sphere bundles over spheres that appear as factors have a particular 
form, which we describe first. Let $\pi_{m}^{S}(S^{n})$ be the $m^{th}$-stable 
homotopy group of $S^{n}$. The least nonvanishing $p$-torsion homotopy 
group of $S^{3}$ is $\pi_{2p}(S^{3})\cong\mathbb{Z}/p\mathbb{Z}$. Let 
\(\alpha\colon\namedright{S^{2p}}{}{S^{3}}\) 
represent a generator. This map is stable and it represents the 
generator in $\pi_{n+2p-3}^{S}(S^{n})\cong\mathbb{Z}/p\mathbb{Z}$.  
For $n\geq 1$, define the space $B(2n+1,2n+2p-1)$ by the homotopy pullback 
\[\diagram 
       S^{2n+1}\rto\ddouble & B(2n+1,2n+2p-1)\rto\dto 
            & S^{2n+2p-1}\dto^{\alpha} \\ 
       S^{2n+1}\rto & S^{4n+3}\rto^-{w} & S^{2n+2} 
  \enddiagram\] 
where $w$ is the Whitehead product of the identity map with itself. 

Restrict to $p\geq 7$ and consider the exceptional simple compact Lie 
groups which are quasi-$p$-regular but not $p$-regular. By~\cite{MNT} 
a complete list is as follows:  
\begin{equation} 
\label{exceptionallist} 
  \begin{array}{lll} 
       F_{4} & p=7 & B(3,15)\times B(11,23) \\ 
               & p=11 & B(3,23)\times S^{11}\times S^{15} \\ 
       E_{6} & p=7 & F_{4}\times S^{9}\times S^{17} \\ 
               & p=11 & F_{4}\times S^{9}\times S^{17} \\ 
       E_{7} & p=11 & B(3,23)\times B(15,35)\times S^{11}\times S^{19}\times S^{27} \\ 
               & p=13 & B(3,27)\times B(11,35)\times S^{15}\times S^{19}\times S^{23} \\ 
               & p=17 & B(3,35)\times S^{11}\times S^{15}\times 
                                      S^{19}\times S^{23}\times S^{27} \\ 
       E_{8} & p=11 & B(3,23)\times B(15,35)\times B(27,47)\times B(39,59) \\ 
                & p=13 & B(3,27)\times B(15,39)\times B(23,47)\times B(35,59) \\ 
                & p=17 & B(3,35)\times B(15,47)\times B(27,59)\times S^{23}\times S^{39} \\ 
                & p=19 & B(3,39)\times B(23,59)\times S^{15}\times S^{27}
                                        \times S^{35}\times S^{47} \\ 
                & p=23 & B(3,47)\times B(15,59)\times S^{23}\times S^{27} 
                                        \times S^{35}\times S^{39} \\ 
                & p=29 & B(3,59)\times S^{15}\times S^{23}\times S^{27}\times S^{35} 
                                         \times S^{39}\times S^{47}. 
  \end{array} 
\end{equation}  
McGibbon~\cite{M} showed that none of the Lie groups listed 
in~(\ref{exceptionallist}) is homotopy commutative. Thus, in every case, 
$\nil(G)\geq 2$. 

Complete at $p\geq 7$ and consider the nonmodular $p$-compact 
groups in cases 4 through 34 of the Shepard-Todd numbering which 
are quasi-$p$-regular but not $p$-regular. We exclude case 28 which 
is~$F_{4}$. As in~\cite{Da}, a complete list is given by: 
\begin{equation} 
  \label{nonmodularlist} 
  \begin{array}{c|c|l} 
       \textit{Case} & \textit{Prime} & \textit{Space} \\ \hline 
        5 & 7 & B(11,23) \\ 
        9 & 17 & B(15,47) \\ 
        10 & 13 & B(23,47) \\ 
        14 & 19 & B(11,47) \\ 
        16 & 11 & B(39,59) \\ 
        17 & 41 & B(39,119) \\ 
        18 & 31 & B(59,119) \\ 
        20 & 19 & B(23,59) \\ 
        24 & 11 & B(7,27)\times S^{11} \\ 
        25 & 7 & B(11,23)\times S^{17} \\ 
        26 & 13 & B(11,35)\times S^{23} \\ 
        27 & 19 & B(23,59)\times S^{13} \\ 
        29 & 13 & B(15,39)\times S^{7}\times S^{23} \\ 
        29 & 17 & B(7,39)\times S^{15}\times S^{23} \\ 
        30 & 11 & B(3,23)\times B(39,59) \\ 
        30 & 19 & B(3,39)\times S^{27}\times S^{59} \\ 
        30 & 29 & B(3,59)\times S^{23}\times S^{39} \\ 
        31 & 13 & B(15,39)\times B(23,47) \\ 
        31 & 17 & B(15,47)\times S^{23}\times S^{39} \\ 
        32 & 13 & B(23,47)\times B(35,59) \\ 
        32 & 19 & B(23,59)\times S^{35}\times S^{47} \\ 
        33 & 13 & B(11,35)\times S^{7}\times S^{19}\times S^{23} \\ 
        34 & 31 & B(23,83)\times S^{11}\times S^{35}\times S^{47}\times S^{59} \\ 
        34 & 37 & B(11,83)\times S^{23}\times S^{35}\times S^{47}\times S^{59} 
  \end{array} 
\end{equation} 
Saumell~\cite{Sa} showed that the only groups in~(\ref{nonmodularlist}) 
which are homotopy commutative are: $B(15,47)$ at $p=17$, $B(11,47)$ at $p=19$, 
$B(39,119)$ at $p=41$, $B(23,59)$ at $p=19$, and $B(7,27)\times S^{11}$ 
at $p=11$. Otherwise, the groups are not homotopy commutative, in which 
case $\nil(G)\geq 2$. 

We now state the main result of this section, to be proved following  
several preparatory lemmas. 

\begin{theorem} 
   \label{qpreg} 
   Let $p\geq 7$ and let $G$ be a quasi-$p$-regular $p$-compact group
   from~(\ref{exceptionallist}) or~(\ref{nonmodularlist}). If $G$ is not homotopy 
   commutative then $\nil(G)=2$. 
\end{theorem} 

We begin with some properties of the space $B(2n+1,2n+2p-1)$, which has 
been well studied. It is a three-cell complex whose mod-$p$ homology is 
$\hlgy{B(2n+1,2n+2p-1)}\cong\Lambda(x_{2n+1},x_{2n+2p-1})$, 
where $\vert x_{t}\vert=t$. Let $A(2n+1,2n+2p-1)$ be the $(2n+2p-1)$-skeleton 
of $B$. Then 
\[\hlgy{B(2n+1,2n+2p-1)}\cong\Lambda(\rhlgy{A(2n+1,2np+1)})\] 
and the skeletal inclusion 
\(i\colon\namedright{A(2n+1,2n+2p-1)}{}{B(2n+1,2n+2p-1)}\) 
induces the inclusion of the generating set in homology. By~\cite{CHZ,CN}, 
$\Sigma i$ has a right homotopy inverse. Thus we obtain the following. 

\begin{lemma} 
   \label{Bretractile} 
   For $n\geq 1$, $(A(2n+1,2n+2p-1),i,B(2n+1,2n+2p-1))$ is a retractile triple.~$\qqed$ 
\end{lemma}  

Let $G$ be a quasi-$p$-regular $p$-compact group such that  
$G\simeq\prod_{i=1}^{\ell} B_{i}$, where each $B_{i}$ is a sphere 
or space of the form $B(2n_{i}+1,2n_{i}+2p-1)$. If $B_{i}=S^{2n_{i}-1}$ 
let $A_{i}=S^{2n_{i}-1}$ and let 
\(I_{i}\colon\namedright{A_{i}}{}{B_{i}}\) 
be the identity map. If $B_{i}=B(2n_{i}+1,2n_{i}+2p-1)$ let 
$A_{i}=A(2n_{i}+1,2n_{i}+2p-1)$. Let $A=\bigvee_{i=1}^{\ell} A_{i}$ and let 
\(I\colon\namedright{A}{}{B}\) 
be the wedge sum of the maps $I_{i}$ for $1\leq i\leq\ell$. The retractile 
property in Lemma~\ref{Bretractile} implies that $(A,I,G)$ is also retractile. 
Referring to $G$ itself as being retractile, this is stated as follows. 

\begin{lemma} 
   \label{Gretractile} 
   Let $G$ be a quasi-$p$-regular $p$-compact group which is a product of 
   spheres and spaces of the form $B(2n+1,2n+2p-1)$. Then $G$ is retractile.~$\qqed$ 
\end{lemma} 

Some elementary information about the homotopy groups of $G$ is also 
required. This requires a few preliminary lemmas. 

\begin{lemma} 
   \label{cellnull} 
   Let $X$ be a finite $CW$-complex with cells in dimensions 
   $\{m_{1},\ldots,m_{s}\}$. Let $Y$ be a space with the property 
   that $\pi_{m_{i}}(Y)=0$ for $1\leq i\leq s$. Then any map 
   \(f\colon\namedright{X}{}{Y}\) 
   is null homotopic. 
\end{lemma} 

\begin{proof} 
For a positive integer $t$, let $X_{t}$ be the $t$-skeleton of $X$. 
Since $X$ has cells in dimensions $\{m_{1},\ldots,m_{s}\}$, for 
$1\leq i\leq s$ there are cofibration sequences  
\[\namedddright{\bigvee S^{m_{i}-1}}{f_{i}}{X_{m_{i-1}}}{}{X_{m_{i}}}{q_{i}} 
     {\bigvee S^{m_{i}}}\] 
for attaching maps $f_{i}$ and pinch maps $q_{i}$, where $X_{m_{0}}$ is the 
basepoint. 

Clearly the restriction of 
\(\namedright{X}{f}{Y}\) 
to $X_{m_{0}}$ is null homotopic. 
Suppose inductively that the restriction of~$f$ to $X_{m_{i-1}}$ 
is null homotopic. Then from the cofibration 
\(\nameddright{X_{m_{i-1}}}{}{X_{m_{i}}}{q_{i}}{\bigvee S^{m_{i}}}\) 
we see that the restriction of $f$ to $X_{m_{i}}$ factors through $q_{i}$ 
to a map 
\(e_{i}\colon\namedright{\bigvee S^{m_{i}}}{}{Y}\). 
But as $\pi_{m_{i}}(Y)=0$, the map~$e_{i}$ is null homotopic. 
Therefore the restriction of $f$ to $X_{m_{i}}$ is null homotopic. 
By induction, the restriction of $f$ to $X_{m_{s}}=X$ - that is, 
$f$ itself - is null homotopic. 
\end{proof} 

A space $B$ is spherically resolved by spheres $S^{n_{1}},\ldots,S^{n_{t}}$ 
if for $1\leq j<t$ there are homotopy fibrations 
\[\nameddright{S^{n_{j}}}{}{B_{j}}{}{B_{j+1}}\] 
where $B_{1}=B$ and $B_{t}=S^{n_{t}}$. For example, $SU(n)$ is 
spherically resolved by $S^{3},S^{5},\ldots,S^{2n-1}$. From 
Lemma~\ref{cellnull} we obtain the following refinement in the 
case of spherically resolved spaces. 

\begin{lemma} 
   \label{cellnull2} 
   Let $X$ be a finite $CW$-complex with cells in dimensions 
   $\{m_{1},\ldots,m_{s}\}$. Let $B$ be spherically resolved by 
   spheres $S^{n_{1}},\ldots, S^{n_{t}}$. If $\pi_{m_{i}}(S^{n_{j}})=0$ 
   for all $1\leq i\leq s$ and $1\leq j\leq t$, then any map 
   \(f\colon\namedright{X}{}{B}\) 
   is null homotopic. 
\end{lemma} 

\begin{proof} 
Proceed by downward induction on $t$. When $t=n$, 
as $B_{t}=S^{n_{t}}$, the hypotheses immediately imply 
that $\pi_{m_{i}}(B_{t})=0$ for $1\leq i\leq s$. If $j<t$, suppose 
inductively that $\pi_{m_{i}}(B_{j+1})=0$ for $1\leq i\leq s$. 
Then as $\pi_{m_{i}}(S^{n_{j}})=0$ for $1\leq i\leq s$, the long 
exact sequence of homotopy groups induced by the homotopy 
fibration 
\(\nameddright{S^{n_{j}}}{}{B_{j}}{}{B_{j+1}}\) 
implies that $\pi_{m_{i}}(B_{j})=0$ for $1\leq i\leq s$. 
As $B=B_{1}$, by induction we have $\pi_{m_{i}}(B)=0$ 
for $1\leq i\leq s$. Lemma~\ref{cellnull} then implies that 
any map 
\(\namedright{X}{f}{B}\) 
is null homotopic. 
\end{proof} 

A special case is $B(3,2p+1)$. This can be regarded as the homotopy 
fibre of the map 
\(\overline{\alpha}\colon\namedright{S^{2p+1}}{}{BS^{3}}\) 
which is adjoint to 
\(\namedright{S^{2p}}{\alpha}{S^{3}}\).  
Since $\alpha$ has order $p$, so does its adjoint, and therefore $\overline{\alpha}\circ p$ 
is null homotopic, implying that the degree $p$ map on $S^{2p+1}$ lifts 
to the homotopy fibre of $\overline{\alpha}$. That is, there is a map 
\(c\colon\namedright{S^{2p+1}}{}{B(3,2p+1)}\) 
with the property that the composite 
\(\nameddright{S^{2p+1}}{c}{B(3,2p+1)}{}{S^{2p+1}}\) 
is of degree~$p$. For a space $X$, let $X\langle 3\rangle$ 
be the three-connected cover of $X$. Notice that the map $c$ lifts to a map 
\(\bar{c}\colon\namedright{S^{2p+1}}{}{B(3,2p+1)\langle 3\rangle}\). 
The following property was proved by Toda~\cite{To1}. 

\begin{lemma} 
   \label{TodaB3} 
   The map 
   \(\namedright{S^{2p+1}}{\bar{c}}{B(3,2p+1)\langle 3\rangle}\) 
   is~$(2p^{2}-2)$-connected. In particular, $\bar{c}$ induces 
   an isomorphism on the homotopy groups $\pi_{m}$ for 
   $m\leq 2p^{2}-2$.~$\qqed$ 
\end{lemma}

\begin{proposition} 
   \label{nilcondition} 
   Let $p$ be an odd prime. Let $G$ be a quasi-$p$-regular $p$-compact group 
   which is a product of spheres and spaces of the form $B(2n+1,2n+2p-1)$. 
   Suppose that either: 
   \begin{itemize} 
      \item[(i)] $G$ has type $\{n_{1},\ldots,n_{\ell}\}$ where   
                     $3n_{\ell}<\min\{n_{1}p,n_{1}+p^{2}-p\}$; or 
      \item[(ii)] $G\simeq B(3,2p+1)\times Y$ where $Y$ is as in part~(i) 
                      and its type also satisfies $2<n_{1}\leq p$ and $p<n_{\ell}$.                  
   \end{itemize} 
   Then $\nil(G)\leq 2$. 
\end{proposition} 

\begin{proof} 
\noindent 
\textit{Part (i)}: 
By Lemma~\ref{Gretractile}, the hypotheses on $G$ imply that there is a 
retractile triple $(A,i,G)$. By Theorem~\ref{nilreduction}, to show that $\nil(G)\leq 2$ 
it suffices to show that the composite 
\(\llnameddright{A\wedge A\wedge A}{i\wedge i\wedge i}{G\wedge G\wedge G} 
     {c_{2}}{G}\) 
is null homotopic. We will show more, that any map 
\(f\colon\namedright{A\wedge A\wedge A}{}{G}\) 
is null homotopic. 

Since $G$ is quasi-$p$-regular of type $\{n_{1},\ldots,n_{\ell}\}$, it is 
spherically resolved by the odd dimensional spheres 
$S^{2n_{1}-1},\ldots, S^{2n_{\ell}-1}$. The type of $G$ also implies 
that $A$ has a $CW$-structure with $\ell$ cells, in dimensions 
$\{2n_{1}-1,\ldots,2n_{\ell}-1\}$. So $A\wedge A\wedge A$ has 
a $CW$-structure consisting of cells in dimensions $\{m_{1},\ldots,m_{s}\}$ 
where each $m_{i}$ is odd. Therefore, by Lemma~\ref{cellnull2}, if 
$\pi_{m}(S^{2n_{i}-1})=0$ for each $m\in\{m_{1},\ldots,m_{s}\}$ 
and each $1\leq i\leq\ell$, then $f$ is null homotopic. 

The dimension of $A\wedge A\wedge A$ is $6n_{\ell}-3$, so we 
consider $\pi_{m}(S^{2n_{i}-1})$ for $m$ an odd number with 
$m\leq 6n_{\ell}-3$. Observe that the stable range for $S^{2n_{i}-1}$ 
is $2n_{i}p-4$. That is, $\pi_{m}(S^{2n_{i}-1})\cong\pi_{m}^{S}(S^{2n_{i}-1})$ 
if $m<2n_{i}p-3$. Since $n_{1}\leq\cdots\leq n_{\ell}$, $S^{2n_{1}-1}$ has the 
stable range of least dimension. So if $6n_{\ell}-3<2n_{1}p-3$, then 
$\pi_{m}(S^{2n_{i}-1})$ is stable for every $m\in\{m_{1},\ldots,m_{s}\}$ 
and every $1\leq i\leq\ell$. This inequality is equivalent to 
$3n_{\ell}<n_{1}p$, which is one of the hypotheses. Therefore, 
we need only consider the stable homotopy groups $\pi_{m}^{S}(S^{2n_{i}-1})$. 

Since $m$ is odd, the stable homotopy group 
$\pi_{m}^{S}(S^{2n_{i}-1})$ is in an even stem. By~\cite{To2}, the 
even stable stem of least dimension is $2p(p-1)-2$ (the stem of 
the stable generator $\beta_{1})$. Thus 
$\pi_{m}^{S}(S^{2n_{i}-1})=0$ for odd $m$ whenever 
$m<(2n_{i}-1)+2p(p-1)-2$. As $n_{1}\leq\cdots\leq n_{\ell}$ and 
$m\leq 6n_{\ell}-3$, we have $\pi_{m}^{S}(S^{2n_{i}-1})=0$ 
whenever $6n_{\ell}-3<(2n_{1}-1)+2p(p-1)-2$. This inequality is 
equivalent to $3n_{\ell}<n_{1}+p^{2}-p$,  which is one of the 
hypotheses. 

Hence $\pi_{m}(S^{2n_{i}-1})=0$ for each $m\in\{m_{1},\ldots,m_{s}\}$ 
and for each $1\leq i\leq n$, and the proof of part~(i) is complete. 
\medskip 

\noindent 
\textit{Part (ii)}: Since $G\simeq B(3,2p+1)\times Y$ with $Y$ 
as in part~(i), by Lemma~\ref{Gretractile} there is a retractile triple 
$(A,i,G)$ where $A=A_{1}\vee A_{2}$ for retractile triples 
$(A_{1},i_{1},B(3,2p+1))$ and $(A_{2},i_{2},Y)$. Notice that~$A_{1}$ 
is the $(2p+1)$-skeleton of $B(3,2p+1)$. 

Consider the composite  
\(\theta\colon\llnameddright{A\wedge A\wedge A}{i\wedge i\wedge i} 
      {G\wedge G\wedge G}{c_{2}}{G}\).  
Observe that $A\wedge A\wedge A$ is $8$-connected, so~$\theta$ 
factors through the three-connected cover 
$G\langle 3\rangle$. Since $G\simeq B(3,2p+1)\times Y$ 
and $n_{1}>2$ implies that $Y$ is at least $4$-connected, we have 
$G\langle 3\rangle\simeq B(3,2p+1)\langle 3\rangle\times Y$. 
By Lemma~\ref{TodaB3}, 
$\pi_{k}(B(3,2p+1)\langle 3\rangle)\cong\pi_{k}(S^{2p+1})$ 
for $k<2p^{2}-1$. Thus, in this dimensional range, it is as if 
$G$ is spherically resolved by $S^{2n_{1}-1},\ldots,S^{2n_{\ell}-1}$ 
and $S^{2p+1}$. The hypothesis that $p< n_{\ell}$ implies  
that the dimension of $A\wedge A\wedge A$ is that of 
$A_{2}\wedge A_{2}\wedge A_{2}$, which is $6n_{\ell}-3$. The hypothesis 
that $n_{1}\leq p$ implies that the arguments for $\pi_{m}(S^{2n_{i}-1})=0$ 
in part~(1) also imply that $\pi_{m}(S^{2p+1})=0$. Thus if 
$6n_{\ell}-3<2p^{2}-1$, then the argument in part~(i) goes through 
without change to show that $\theta$ is null homotopic and 
$\nil(G)\leq 2$.  
 
It remains to show that $6n_{\ell}-3<2p^{2}-1$. By hypothesis, 
$3n_{\ell}<\min\{n_{1}p,n_{1}+p^{2}-p\}$ and $n_{1}\leq p$. If 
$\min\{n_{1}p,n_{1}+p^{2}-p\}=n_{1}p$ then $3n_{\ell}<n_{1}p$ 
and $n_{1}\leq p$ implies that $6n_{\ell}-3<2p^{2}-3<2p^{2}-1$. If 
$\min\{n_{1}p,n_{1}+p^{2}-p\}=n_{1}+p^{2}-p$ then 
$3n_{\ell}<n_{1}+p^{2}-p$ and $n_{1}\leq p$ implies that  $3n_{\ell}<p^{2}$ 
and so $6n_{\ell}-3<2p^{2}-3<2p^{2}-1$. In either case, $6n_{\ell}-3<2p^{2}-1$,  
as required.   
\end{proof} 

\begin{proof}[Proof of Theorem~\ref{qpreg}] 
It has already been mentioned that every group $G$ in~(\ref{exceptionallist}) and 
every non-homotopy commutative group $G$ in~(\ref{nonmodularlist}) has 
$\nil(G)\geq 2$. In each such case, except $E_{8}$ at $p=11$, the hypotheses 
of Proposition~\ref{nilcondition} hold, so $\nil(G)\leq 2$. Hence $\nil(G)=2$. 

The one outstanding case is $E_{8}$ at $p=11$. As in part~(ii) of 
Proposition~\ref{nilcondition}, write 
$E_{8}\simeq B(3,23)\times Y$ where $Y$ has type 
$\{8,14,18,20,24,30\}$. Now $n_{1}=8$ and $n_{\ell}=30$, but 
$3n_{\ell}$ is not less than $\min\{n_{1}p,n_{1}+p^{2}-p\}=\min\{88,118\}$. 
The one obstruction that appears in the argument proving 
Proposition~\ref{nilcondition} corresponds to $6n_{\ell}-3=177$ (the 
dimension of $A\wedge A\wedge A$) being less than $2n_{1}p-3=173$ 
(the stable range of $S^{15}$). But from the homotopy fibration 
\(\nameddright{W_{8}}{}{S^{15}}{E^{2}}{\Omega^{2} S^{17}}\) 
induced by the double suspension $E^{2}$, since $W_{8}$ is homotopy 
equivalent to the Moore space $P^{175}(11)$ in dimensions~$\leq 180$ 
and $\pi_{177}(P^{175}(11))=0$, we see that $\pi_{177}(S^{15})$ is 
stable, and so the argument in Proposition~\ref{nilcondition} goes through. 
Therefore $\nil(E_{8})=2$ in this case as well.   
\end{proof} 

\begin{remark} 
There are other quasi-$p$-regular exceptional Lie groups: $G_{2}$, $F_{4}$ 
and $E_{6}$ at $p=3$ and~$p=5$. The cases $F_{4}$ and $E_{6}$ at $3$ 
have torsion in mod-$3$ homology; therefore a characterization of 
homotopy nilpotent Lie groups by Rao~\cite{R} implies that, localized at $3$, 
we have $\nil(F_{4})=\infty$ and $\nil(E_{6})=\infty$. McGibbon~\cite{M} 
showed that $G_{2}$ at $p=5$ is homotopy commutative; therefore 
localized at $5$ we have $\nil(G_{2})=1$. The remaining cases all have torsion free 
homology and are not homotopy commutative. Moreover, they do not satisfy the 
hypotheses of Proposition~\ref{nilcondition} and fail to do so in a way that 
does not allow for a sidestepping argument as for $E_{8}$ at $p=11$ in the 
proof of Theorem~\ref{qpreg}. So the precise value of $\nil(G)$ in these 
cases remains undetermined. 
\end{remark}

%%% The bibliography %%%
\bibliographystyle{amsalpha}

\end{document}